\documentclass[preprint]{imsart}
\RequirePackage[colorlinks,citecolor=blue,urlcolor=blue]{hyperref}
\usepackage{amsmath,amsmath,amssymb,amsfonts,qtree}
\usepackage{amsthm}
\usepackage[american]{babel}
\usepackage{graphicx}
\usepackage{flafter}
\usepackage[section]{placeins}
\usepackage{float}
\usepackage{makecell}
\usepackage{comment}
\usepackage[mathscr]{eucal}
\usepackage{bbm}
\usepackage{todonotes}
\presetkeys{todonotes}{color=red!30}{linecolor=black!50}

\startlocaldefs

\def\eps{{\varepsilon}}

\def\Bbb E{\mathbb{E}}
\def\Bbb R{\mathbb{R}}
\newtheorem{corollary}{Corollary}[section]
\usepackage{mathtools}
\makeatletter \@addtoreset{equation}{section}

\makeatother

\newtheorem{lemma}{Lemma}[section]
\newtheorem{theorem}{Theorem}[section]
\newtheorem{proposition}{Proposition}[section]

\newtheorem{remark}{Remark}[section]

 \makeatletter
\def\@fnsymbol#1{\ensuremath{\ifcase#1\or * \or \mathsection \or ** \else\@ctrerr\fi}}
\makeatother

\font\tencmmib=cmmib10 \skewchar\tencmmib '60
\newfam\cmmibfam
\textfont\cmmibfam=\tencmmib

\font\tenmsb=msbm10 
\def\Bbb#1{\hbox{\tenmsb#1}}

\def\lessim{\ \lower4pt\hbox{$
\buildrel{\displaystyle <}\over\sim$}\ }
\def\gessim{\ \lower4pt\hbox{$\buildrel{\displaystyle >}
\over\sim$}\ }

\def\eps{\varepsilon}

\def\go0{\to 0}

\def\leftitem#1{\item{\hbox to\parindent{\enspace#1\hfill}}}

\def\sg{\sigma}

\def\sg2{\sigma^2}

\def\__{_{\infty}}
\def \thetah{\hat{\theta}}

\numberwithin{equation}{section} 

\newcommand{\1}{{\rm 1}\kern-0.24em{\rm I}}

\endlocaldefs

\begin{document}

\begin{frontmatter}
\title{Estimation of smooth functionals in high-dimensional models: bootstrap chains and Gaussian approximation}
\runtitle{Estimation of smooth functionals}

\begin{aug}
\author{\fnms{Vladimir} \snm{Koltchinskii}\thanksref{t1}\ead[label=e1]{vlad@math.gatech.edu}} 
\thankstext{t1}{Supported in part by NSF grants DMS-1810958 and DMS-2113121}
\runauthor{V. Koltchinskii}

\affiliation{Georgia Institute of Technology\thanksmark{m1}}

\address{School of Mathematics\\
Georgia Institute of Technology\\
Atlanta, GA 30332-0160\\
\printead{e1}\\
}
\end{aug}
\vspace{0.2cm}
{\small \today}
\vspace{0.2cm}

\begin{abstract}
Let $X^{(n)}$ be an observation sampled from a distribution $P_{\theta}^{(n)}$
with an unknown parameter $\theta,$ $\theta$ being a vector in a Banach space $E$
(most often, a high-dimensional space of dimension $d$). 
We study the problem of estimation of $f(\theta)$ 
for a functional $f:E\mapsto {\mathbb R}$ of some smoothness $s>0$ based on an observation $X^{(n)}\sim P_{\theta}^{(n)}.$
Assuming that there exists an estimator $\hat \theta_n=\hat \theta_n(X^{(n)})$ of parameter $\theta$ 
such that $\sqrt{n}(\hat \theta_n-\theta)$ is sufficiently close in distribution to a mean zero Gaussian
random vector in $E,$ we construct a functional $g:E\mapsto {\mathbb R}$ such that $g(\hat \theta_n)$ is an asymptotically normal estimator of $f(\theta)$ with $\sqrt{n}$ rate
provided that $s>\frac{1}{1-\alpha}$ and $d\leq n^{\alpha}$ for some $\alpha\in (0,1).$
We also derive general upper bounds on Orlicz norm error rates for estimator $g(\hat \theta)$ depending on smoothness $s,$ dimension $d,$ sample size $n$ and the accuracy of normal  
approximation of $\sqrt{n}(\hat \theta_n-\theta).$ In particular, this approach yields asymptotically efficient estimators in 
high-dimensional log-concave exponential models. 
\end{abstract}



\end{frontmatter}


\newpage


\section{Introduction}

The problem of estimation of a smooth functional $f(\theta)$ of parameter $\theta$
of a high-dimensional statistical model will be studied in this paper in the case when there 
exists an estimator $\hat \theta$ of $\theta$ for which normal approximation holds 
as both the dimension $d$ and the sample size are reasonably large. 

Estimation of functionals of parameters of non-parametric and, more recently, high-dimensional 
statistical models has been studied by many authors since the 70s
\cite{Levit_1, Levit_2, Ibragimov, Bickel_Ritov, Ibragimov_Khasm_Nemirov,
Girko, Girko-1, Donoho_1, Donoho_2, Donoho_Nussbaum, Nemirovski_1990, BKRW,  Birge, Laurent, Lepski, 
Nemirovski, Cai_Low_2005a, Cai_Low_2005b, Klemela, Robins, van der Vaart, C_C_Tsybakov, Robins_1, Han, Mukherjee}. Most of the results have been obtained for special statistical models 
(Gaussian sequence model, Gaussian white noise model, density estimation model) and special functionals (linear and quadratic 
functionals, norms in classical Banach spaces, certain classes of integral functionals of unknown density). 
Estimation of general smooth functionals was studied in \cite{Ibragimov_Khasm_Nemirov, Nemirovski_1990, Nemirovski} for the model of an unknown infinite-dimensional function (signal) 
observed in a Gaussian white noise. Sharp thresholds on the smoothness 
of the functional depending on the complexity (smoothness) of the signal that guarantee 
efficient estimation of the functional were studied in these papers. 

Our approach is based on a bias reduction method that goes back to the idea of iterated bootstrap (see \cite{Hall_1, Hall}). This method  
has been recently studied in the case of high-dimensional normal models (see \cite{Koltchinskii_2017, Koltchinskii_2018, Koltchinskii_Zhilova, Koltchinskii_Zhilova_19}). In particular, it was shown that it yields efficient estimation 
of functionals of smoothness $s$ of unknown mean and covariance with parametric $\sqrt{n}$ convergence rate  provided that $s>\frac{1}{1-\alpha}$ and $d\leq n^{\alpha}$ for some $\alpha\in (0,1),$ $d$ being the dimension 
of the space. Moreover, the smoothness threshold $\frac{1}{1-\alpha}$ is sharp in the 
sense that for $s<\frac{1}{1-\alpha}$ the minimax optimal convergence rate is slower than $\sqrt{n}.$   
Our goal is to extend some of these results to more general high-dimensional models under an assumption 
that the model admits statistical estimators of unknown parameter for which normal approximation holds for large $n$ and sufficiently high dimension of the parameter.

\subsection{Bias reduction}
\label{sec:bias_reduce}

Let $X^{(n)}$ be an observation sampled from a probability distribution $P_{\theta}^{(n)}$ 
in a measurable space $(S^{(n)},{\mathcal A}^{(n)})$ with unknown parameter $\theta\in T.$
A particular example of interest is $X^{(n)}=(X_1,\dots, X_n),$ where 
$X_1,\dots, X_n$ are i.i.d. observations in a measurable spaces $(S,{\mathcal A}).$
It will be assumed in what follows that the parameter space $T$ is an open subset 
of a separable Banach space $E$ (which could be a high-dimensional or 
infinite-dimensional space). 
Let $\hat \theta = \hat \theta_n=\hat \theta (X^{(n)})\in T$ be an estimator
of $\theta$ based on the observation $X^{(n)}.$ 
We will be especially interested 
in estimators $\hat \theta$ that could be approximated in distribution by a Gaussian 
random vector in $E$ (whose distribution, of course, depends on unknown parameter 
$\theta\in T$ provided that $X^{(n)}\sim P_{\theta}^{(n)}$). More precisely, it will 
be assumed in what follows that, for all $\theta \in T$ (or in properly chosen subsets of $T$),
$\sqrt{n}(\hat \theta-\theta)$ is close in distribution to a mean zero Gaussian random vector 
$\xi(\theta)$ in $E.$ In Section \ref{sec:Main results}, it will be described more precisely 
in which sense this approximation should hold. 

Given a smooth functional $f:T \mapsto {\mathbb R},$ our main goal is to construct 
an estimator of $f(\theta)$ based on $X^{(n)}.$ It is well known that in high-dimensional 
and infinite-dimensional models the plug-in estimator $f(\hat \theta)$ is often sub-optimal 
even when the base estimator $\hat \theta$ is optimal. This is largely due to the fact that 
for non-linear functionals $f$ the plug-in estimator $f(\hat \theta)$ has a large bias even when 
$\hat \theta$ is unbiased, or has a small bias. Thus, the bias reduction becomes a crucial part
of the design of estimators of $f(\theta)$ with optimal error rates. 
To construct an unbiased estimator of $f(\theta)$ (which is not always possible)
one has to solve an integral equation ${\mathcal T} g=f$ for the following 
integral operator:
\begin{align}
\label{define_T}
({\mathcal T} g)(\theta):= {\mathbb E}_{\theta} g(\hat \theta) = \int_{T} g(t)P(\theta;dt), \theta \in T,
\end{align}
where 
\begin{align}
\label{define_Markov_kernel}
P(\theta;A) = {\mathbb P}_{\theta}\{\hat \theta\in A\}, A\subset T
\end{align}
is a Markov kernel on the parameter space $T$ (the distribution 
of estimator $\hat \theta$). Recall that, by the definition of Markov kernel,  
it is assumed that $T\ni\theta\mapsto P(\theta;A)$ is a Borel measurable function 
for all Borel subsets $A\subset T.$

Note that ${\mathcal T} f$ is well defined for all functions $f\in L_{\infty}(T)$
and, moreover, ${\mathcal T}: L_{\infty}(T)\mapsto L_{\infty}(T)$ is a 
contraction.  
Most often, we will deal with operator ${\mathcal T}$ acting on uniformly bounded Lipschitz functions
(or even on sufficiently smooth functions).

Finding an estimator of $f(\theta)$ with a small 
bias then reduces to an approximate solution of equation ${\mathcal T} g=f.$ 
If ${\mathcal B}:= {\mathcal T}-{\mathcal I}$ is a ``small operator" 
(which is the case when the estimator $\hat \theta$ is ``close" to $\theta$ with a high probability),
then the solution of this equation could be written (at least, formally) as the sum 
of Neumann series 
\begin{align*}
g=({\mathcal I}+{\mathcal B})^{-1}f = f-{\mathcal B}f+{\mathcal B}^2 f - {\mathcal B}^3 f + \dots 
\end{align*}
and one can try to use the following function $f_k(\theta)$ (with a properly chosen $k$),
\begin{align*}
f_k(\theta):= \sum_{j=0}^k (-1)^j ({\mathcal B}^j f)(\theta),
\end{align*}
as an approximate solution of equation ${\mathcal T}g=f.$ This yields an estimator 
$f_k(\hat \theta)$ with a reduced bias 
\begin{align*}
{\mathbb E}_{\theta} f_k(\hat \theta)- f(\theta) = (-1)^k ({\mathcal B}^{k+1} f)(\theta), \theta\in T.
\end{align*} 
Another way to look at this bias reduction procedure is to observe that the bias of the plug-in 
estimator $f(\hat \theta)$ is equal to 
\begin{align*}
{\mathbb E}_{\theta} f(\hat \theta)-f(\theta) = ({\mathcal T} f)(\theta)-f(\theta) = ({\mathcal B}f)(\theta), \theta \in
T.
\end{align*}
To reduce the bias of $f(\hat \theta),$ one could subtract from it the plug-in estimator of the function 
$({\mathcal B}f)(\theta)$ yielding the estimator $f_1(\hat \theta)= f(\hat \theta)- ({\mathcal B} f)(\hat \theta).$
The bias of $f_1(\hat \theta)$ is equal to $-({\mathcal B}^2 f)(\theta).$ To further reduce the bias, we have 
to add its plug-in estimator $({\mathcal B}^2 f)(\hat \theta)$ yielding the estimator 
$f_2(\hat \theta)= f(\hat \theta)- ({\mathcal B} f)(\hat \theta)+({\mathcal B}^2 f)(\hat \theta),$
and so on. 

This higher order bias reduction method has been studied in \cite{Koltchinskii_2017, Koltchinskii_2018,
Koltchinskii_Zhilova, Koltchinskii_Zhilova_19} in the case of various high-dimensional 
normal models and in \cite{Jiao} in the case of the classical binomial model. In particular,
the approach to the analysis of this method initiated in  \cite{Koltchinskii_2017, Koltchinskii_2018} and further developed in \cite{Koltchinskii_Zhilova_19}
is based on the derivation of integral representation formulas for functions $({\mathcal B}^k f)(\theta)$ in terms of so called 
smooth random homotopies. These formulas provide a way to obtain sharp bounds on the bias of estimator $f_k(\hat \theta)$
and to establish smoothness properties of functions $f_k$ needed to develop concentration inequalities for this estimator
(see Section \ref{sec:boostrap_chain} for more details).  However, the construction of random 
homotopies for a given estimator $\hat \theta$ relies on certain coupling techniques. In particular,
it is based on the existence of a smooth stochastic process $G(\theta), \theta \in \Theta$ with values in $\Theta$
such that $G(\theta)\overset{d}{=} \hat \theta (X^{(n)}), X^{(n)}\sim P_{\theta}.$ The bounds on the bias of estimator
$f_k(\hat \theta)$ obtained in \cite{Koltchinskii_Zhilova_19} rely on the existence of such a coupling and the H\"older norms of process $G$
are involved in these bounds. Such a coupling trivially exists in the case of random shift models \cite{Koltchinskii_Zhilova, Koltchinskii_Zhilova_2020} and it is easy to construct in the case of general Gaussian models \cite{Koltchinskii_Zhilova_19} as well as some other 
exponential transformation families. However, it is much harder to develop smooth random homotopies for MLE and other relevant estimators in the case of more general high-dimensional parametric models.     
A possible approach could rely on general coupling methods developed in the literature such as optimal transport maps and Moser's coupling (see, e.g., \cite{Villani}).
However, the bounds on H\"older norms for such coupling maps with explicit dependence on the dimension have not been developed in the literature and 
their development leads to difficult questions concerning smoothness of solutions of PDEs (in particular, Monge-Amp\`ere and Poisson type equations) in high dimensions. 
Another serious difficulty is the need to develop tight concentration bounds for estimators $f_k(\hat \theta)$ that are also not readily available for general high-dimensional models 
(with Gaussian, log-concave and some closely related models being exceptions). Due to these difficulties, the higher order bias reduction method described above has been so far fully studied only 
in the case of Gaussian models as well as some random shift models with Poincar\'e type noise \cite{Koltchinskii_Zhilova_19}. 

In this paper, we study the problem under an additional assumption that the estimator $\hat \theta$ admits sufficiently accurate normal approximation. 
More precisely, we assume that $\sqrt{n}(\hat \theta- \theta)$ can be approximated in distribution by a Gaussian r.v. $\xi(\theta)$ in $E.$ This assumption allows us to define an approximating Gaussian 
model, an ``estimator" $\tilde \theta=\theta + \frac{\xi(\theta)}{\sqrt{n}}$ of parameter $\theta$ for this model and the corresponding operators $\tilde {\mathcal T}, \tilde{\mathcal B}$ and functions $\tilde f_k, k\geq 0.$
We show that functions $\tilde f_k$ provide a reasonable approximation of functions $f_k$ and one can reduce the bounds on estimator $f_k(\hat \theta)$ of $f(\theta)$ to the bounds on estimator $\tilde f_k(\tilde \theta)$
in the corresponding approximating Gaussian model. This approach allows us to circumvent the difficulties with the direct analysis 
of estimator $f_k(\hat \theta)$ since both the technique of random homotopies and concentration inequalities are applicable to the approximating model. 
As a result, we prove \it ``reduction theorems"
\rm (stated in Section \ref{sec:Main results})
showing that the risk bounds and normal approximation properties established earlier in the Gaussian case hold also for general models, provided that the normal approximation of estimator $\hat \theta$ is sufficiently accurate.

\subsection{Smoothness classes and distances between random variables}

Let $F$ be a Banach space and let $U\subset E.$ For a function $g:U\mapsto F,$
denote 
\begin{align*}
\|g\|_{L_{\infty}(U)}:= \sup_{x\in U}\|g(x)\|,
\end{align*}
\begin{align*}
\|g\|_{{\rm Lip}(U)}:= \sup_{x,x'\in U, x\neq x'}\frac{\|g(x)-g(x')\|}{\|x-x'\|}
\end{align*}
and, for $\rho \in (0,1],$
\begin{align*}
\|g\|_{{\rm Lip}_{\rho}(U)}:= \sup_{x,x'\in U, x\neq x'}\frac{\|g(x)-g(x')\|}{\|x-x'\|^{\rho}}.
\end{align*}
We will now introduce H\"older spaces $C^{s}(U;F)$ of functions of smoothness $s>0$ from an open subset $U\subset E$ into a Banach space $F$ (most often, either $F={\mathbb R},$  or $F=E$). 
Given a function $g:U\mapsto F,$ let $g^{(j)}$ denote its Fr\'echet derivative of order 
$j$ (in particular, $g^{(0)}=g$). Note that, for all $x\in U,$ $g^{(j)}(x)$ is a symmetric bounded $j$-linear 
form (with values in $F$). For such forms $M[u_1,\dots, u_j], u_1,\dots, u_j\in E,$ we will 
use the operator norm 
\begin{align*}
\|M\|:= \sup_{\|u_1\|\leq 1,\dots, \|u_j\|\leq 1}\|M[u_1,\dots, u_j]\|
\end{align*}
and $g^{(j)}$ will be always viewed as a mapping from $U$ into the space of symmetric bounded 
$j$-linear forms equipped with the operator norm.  
Let $s=m+\rho,$ $m\geq 0, \rho \in (0,1].$ For an $m$-times Fr\'echet differentiable 
function $g$ from $U$ into $F,$ define 
\begin{align*}
&
\|g\|_{C^s(U;F)} := 
\max\Bigl(\|g\|_{L_{\infty}(U)}, 
\max_{0\leq j\leq m-1} \|g^{(j)}\|_{{\rm Lip}(U)},
\|g^{(m)}\|_{{\rm Lip}_{\rho}(U)}
\Bigr).
\end{align*}
The space $C^s(U,F)$ is then defined as the set of all $m$-times Fr\'echet differentiable 
functions $g$ from $U$ into $F$ such that $\|g\|_{C^s(U,F)}<\infty.$ When the space $F$
is clear from the context (in particular, when $F={\mathbb R}$), we will write simply 
$C^s(U)$ and $\|\cdot \|_{C^s(U)}$ instead of $C^s(U,F)$ and $\|\cdot \|_{C^s(U,F)}.$

\begin{remark}
\normalfont
The definition of the space $C^s(U)$ used here is not quite standard.
In particular, the space $C^1(U)$ consists of all uniformly bounded
Lipschitz functions in $U$ rather than continuously differentiable 
functions. 
Note also that, for a $j$ times Fr\'echet differentiable function $g,$ 
$
\|g^{(j)}\|_{L_{\infty}(U)}\leq 
\|g^{(j-1)}\|_{{\rm Lip}(U)},
$
with the equality holding when $U$ is convex
(which would lead to a more standard definition of H\"older norms). 
\end{remark}

We will also use the following notation. Let $s=m+\rho,$ $m\geq 0, \rho \in (0,1].$
For $l=0,\dots, m,$ denote 
\begin{align*}
&
\|g\|_{C^{l,s}(U;F)} := 
\max\Bigl( 
\max_{l\leq j\leq m-1} \|g^{(j)}\|_{{\rm Lip}(U)},
\|g^{(m)}\|_{{\rm Lip}_{\rho}(U)}
\Bigr).
\end{align*}
and let $C^{l,s}(U;F)$ be the set of all $m$-times Fr\'echet differentiable 
functions $g$ from $U$ into $F$ such that $\|g\|_{C^{l,s}(U,F)}<\infty.$
In particular, $\|\cdot\|_{C^{0,1}(U)}= \|\cdot\|_{{\rm Lip}(U)}.$

\begin{remark}
\normalfont	
Note that, by McShane-Whitney extension theorem, any Lipschitz 
function $g:U\mapsto {\mathbb R}$ 
could be extended to a Lipschitz function defined on the whole space $E$
with preservation of its Lipschitz norm $\|g\|_{{\rm Lip}(U)}$ (in fact, this theorem 
applies to general metric spaces, not just to Banach spaces). 
Moreover, any function $g\in C^{1}(U)$ (a uniformly bounded Lipschitz 
function) could be extended to the whole space $E$ with preservation of 
its $C^{1}$-norm.
In what follows, it will be convenient to assume that bounded Lipschitz functions (in particular, functions 
from the space $C^s(U)$ for $s\geq 1$) and Lipschitz functions (in particular, functions from 
the space $C^{0,s}(U)$ for $s\geq 1$)
are indeed extended 
to the whole space this way. Similarly, any function from space $C^s(U), U\subset E, s\in (0,1]$
could be extended to the whole space $E$ with preservation of its norm (again by the application of 
McShane-Whitney extension theorem to the metric space $(E,d),$ $d(x,y):=\|x-y\|^s, x,y\in E$).
Note that the problem of extension of smooth functions (from space $C^s(U)$ with $s>1$)
to the whole space with preservation of the norm is much more complicated and such extensions do not always exist in general Banach spaces.  
\end{remark}

We will need to quantify the accuracy of normal approximation for random variable 
$\sqrt{n}(\hat \theta-\theta)$ by 
$\xi(\theta)$  (as well as for other random variables), 
and, for this purpose, we will introduce below certain distances between distributions of random variables. 

Let $\eta_1, \eta_2$ be random variables defined on a probability space $(\Omega,\Sigma,{\mathbb P})$ with values in a measurable space $(S,{\mathcal A}),$ and 
let ${\mathcal F}$ be a set of measurable functions on $S.$
Define 
\begin{align*}
\Delta_{\mathcal F}(\eta_1,\eta_2) := \sup_{f\in {\mathcal F}}|{\mathbb E}f(\eta_1)-{\mathbb E}f(\eta_2)|.
\end{align*}

\begin{remark}
	\normalfont
Note that, in fact, $\Delta_{\mathcal F}(\eta_1,\eta_2)$ is a distance between the laws ${\mathcal L}(\eta_1), {\mathcal L}(\eta_2)$ of random variables $\eta_1,\eta_2$ (so, it does not matter whether 
$\eta_1, \eta_2$ are defined on the same probability space or not; however, it is always possible 
to assume that they are and it will be convenient for our purposes).
\end{remark}

Let now $\psi:{\mathbb R}\mapsto {\mathbb R}_+$ be an even convex function with $\psi(0)=0$
and such that $\psi$ is increasing in ${\mathbb R}_+.$  The Orlicz $\psi$-norm of real valued r.v. $\zeta$ is defined as 
\begin{align*} 
\|\zeta\|_{\psi}:= 
\inf\Bigl\{c>0: {\mathbb E}\psi\Bigl(\frac{|\zeta|}{c}\Bigr)\leq 1\Bigr\}.
\end{align*}
Denote $L_{\psi}({\mathbb P}):=\{\zeta: \|\zeta\|_{\psi}<+\infty\}.$
We will also write $\|\zeta\|_{L_{\psi}({\mathbb P})}=\|\zeta\|_{\psi}$
(to emphasize the dependence of the Orlicz norm on the underlying probability measure ${\mathbb P}$).
If $\psi(u):= |u|^p, u\in {\mathbb R}, p\geq 1,$ then $\|\cdot \|_{\psi}=\|\cdot \|_{L_p}$
and $L_{\psi}({\mathbb P})=L_p({\mathbb P}).$ Other common choices of $\psi$ are 
$\psi_1(u)=e^{|u|}-1$ (subexponential Orlicz norm) and $\psi_2(u)=e^{u^2}-1$ (subgaussian 
Orlicz norm). 

We will need another distance between random variables $\eta_1,\eta_2$ in a space $(S,{\mathcal A})$
defined as follows:
\begin{align*}
\Delta_{{\mathcal F},\psi}(\eta_1,\eta_2) := \sup_{f\in {\mathcal F}}|\|f(\eta_1)\|_{\psi}-\|f(\eta_2)\|_{\psi}|.
\end{align*}
If $(S,d)$ is a metric space, one can also define the following Wasserstein type distance:
\begin{align*}
W_{\psi}(\eta_1,\eta_2) := \inf\Bigl\{\|d(\eta_1', \eta_2')\|_{\psi}: \eta_1'\overset{d}{=} \eta_1, 
\eta_2'\overset{d}{=} \eta_2\Bigr\}, 
\end{align*}  
where the infimum is taken over all random variables $\eta_1', \eta_2'$ on $(\Omega,\Sigma,{\mathbb P})$ such that $ \eta_1'$ has the same distribution as $\eta_1$ and $\eta_2'$ has the same distribution 
as $\eta_2.$ If $\psi(u)=|u|^p$ this becomes a usual definition of the Wasserstein $p$-distance
$W_p.$ For this choice of $\psi,$ we also use the notation $\Delta_{{\mathcal F},p}$ instead of 
$\Delta_{{\mathcal F},\psi}.$

We will also use the notations $\Delta_{{\mathcal F},{\mathbb P}}(\eta_1,\eta_2),$
$\Delta_{{\mathcal F},\psi, {\mathbb P}}(\eta_1,\eta_2)$ and $W_{\psi,{\mathbb P}}(\eta_1,\eta_2)$
whenever it is needed to emphasize the dependence of these distances on ${\mathbb P}.$

Since, for $\eta_1'\overset{d}{=}\eta_1, \eta_2'\overset{d}{=}\eta_2,$ 
\begin{align*}
|\|f(\eta_1)\|_{\psi}-\|f(\eta_2)\|_{\psi}|= |\|f(\eta_1')\|_{\psi}-\|f(\eta_2')\|_{\psi}|
\leq \|f(\eta_1')-f(\eta_2')\|_{\psi},
\end{align*}
we could conclude that 
\begin{align}
\label{Delta_psi_W_psi_comp}
\Delta_{{\mathcal F},\psi}(\eta_1,\eta_2) \leq \sup_{f\in {\mathcal F}}W_{\psi}(f(\eta_1), f(\eta_2))
\end{align}
(for r.v. $\eta_1,\eta_2$ with values in an arbitrary measurable space $S$). If $(S,d)$
is a complete separable metric space and ${\mathcal F}$ is the set of all contractions (Lipschitz functions on $S$ with constant $1$), then, for all $f\in {\mathcal F},$ $|f(\eta_1)-f(\eta_2)|\leq d(\eta_1,\eta_2)$
implying that 
\begin{align}
\label{Delta_F_psi_W_psi}
\Delta_{{\mathcal F},\psi}(\eta_1,\eta_2) \leq \sup_{f\in {\mathcal F}}W_{\psi}(f(\eta_1), f(\eta_2))
\leq W_{\psi}(\eta_1,\eta_2).
\end{align} 
Note also that for $\psi(u)=|u|$ (the $L_1$-norm), we have  
\begin{align*}
\Delta_{{\mathcal F},1}(\eta_1,\eta_2)\leq W_1(\eta_1,\eta_2)= \Delta_{{\mathcal F}}(\eta_1,\eta_2),
\end{align*}
where ${\mathcal F}$ is the set of all real valued contractions on $S$
(follows from Kantorovich-Rubinstein duality).

Most often, we will deal with random variables in a Banach space $F$ (in particular, $F=E$ and 
$F={\mathbb R}$) and the set ${\mathcal F}$ will usually be a H\"older ball of certain smoothness, 
such as ${\mathcal F}=\{f: \|f\|_{C^s(E)}\leq 1\}$ for $s>0,$ or ${\mathcal F}:=\{f: \|f\|_{C^s(U)}\leq 1\},$
or ${\mathcal F}:=\{f: \|f\|_{C^{l,s}(U)}\leq 1\}$ for some $0\leq l<s$ and for $U\subset F.$ In particular, we will use the notations 
$$\Delta_s(\eta_1,\eta_2)= \Delta_{{\mathcal F}}(\eta_1, \eta_2)\ {\rm and}\ \Delta_{s,\psi}(\eta_1,\eta_2)= \Delta_{{\mathcal F},\psi}(\eta_1, \eta_2)$$ 
for ${\mathcal F}=\{f: \|f\|_{C^s(E)}\leq 1\}.$ 

Other distances that will be used in the future include:
\begin{itemize}
\item Kolmogorov's distance between random variables $\eta_1, \eta_2$ in ${\mathbb R}$ (more precisely, between their laws ${\mathcal L}(\eta_1), {\mathcal L}(\eta_2)$) defined as 
\begin{align*}
d_K(\eta_1,\eta_2):= \sup_{x\in {\mathbb R}} |{\mathbb P}\{\eta_1\leq x\}- {\mathbb P}\{\eta_2\leq x\}|= \Delta_{{\mathcal F}}(\eta_1,\eta_2),
\end{align*}
where ${\mathcal F}:=\{I_{(-\infty, x]}:x\in {\mathbb R}\}.$
\item For $s=k+\rho,\ k\geq 0, \rho\in (0,1]$ and random variables $\eta_1, \eta_2$ in a Banach space $E,$ let
\begin{align}
\label{Zolotarev}
&
\zeta_s(\eta_1,\eta_2):= 
\sup_{\|f^{(k)}\|_{{\rm Lip}_{\rho}(E)}\leq 1} 
|{\mathbb E}f(\eta_1)-{\mathbb E}f(\eta_2)|
= \Delta_{{\mathcal F}}(\eta_1,\eta_2),
\end{align}
where ${\mathcal F}:=\{f: \|f^{(k)}\|_{{\rm Lip}_{\rho}(E)}\leq 1\}.$
Note that, for $s=1,$ $\zeta_1(\eta_1,\eta_2)=W_1(\eta_1,\eta_2).$ 
\end{itemize}

Finally, in a statistical framework, we have to deal with a family of probability measures 
${\mathbb P}_{\theta}, \theta \in \Theta$ (that generates different distributions of the data)
and we will use uniform versions of the distances defined above:
\begin{align*}
\Delta_{{\mathcal F},\Theta}(\eta_1,\eta_2):= \sup_{\theta\in \Theta}\Delta_{{\mathcal F},{\mathbb P}_{\theta}}(\eta_1,\eta_2),
\end{align*}
\begin{align*}
\Delta_{{\mathcal F},\psi, \Theta}(\eta_1,\eta_2):= \sup_{\theta\in \Theta}\Delta_{{\mathcal F},\psi, {\mathbb P}_{\theta}}(\eta_1,\eta_2)
\end{align*}
and 
\begin{align*}
W_{\psi, \Theta}(\eta_1,\eta_2):=\sup_{\theta\in \Theta} W_{\psi,{\mathbb P}_{\theta}}(\eta_1,\eta_2).
\end{align*}
We will also use the distance 
\begin{align*}
\Delta_{{\mathcal F},\psi, \Theta}^{+}(\eta_1,\eta_2):= 
\Delta_{{\mathcal F},\Theta}(\eta_1,\eta_2)+
\Delta_{{\mathcal F},\psi, \Theta}(\eta_1,\eta_2).
\end{align*}

Throughout the paper, the following notations will be used. For real non-negative 
variables $A,B,$ $A\lesssim B$ means that there exists a universal constant $C>0$
such that $A\leq CB,$ $A\gtrsim B$ means that $B\lesssim A$ and $A\asymp B$
means that $A\lesssim B$ and $B\lesssim A.$ If constant $C$ in the above inequalities 
depends on additional parameters, we will provide the corresponding relationships with subscripts:
say, $A\lesssim_{s,\psi} B$ means that $A\leq C B$ with $C=C_{s,\psi}>0$ depending on $s$ and $\psi.$

\section{Main results}
\label{sec:Main results}

In this section, we study the error rates of estimator $f_k(\hat \theta)$ depending on the smoothness of functional $f$ and show that they coincide with the rates known to be optimal in the Gaussian case provided that normal approximation error for $\hat \theta$ is negligible.   

In \cite{Koltchinskii_Zhilova}, the following Gaussian shift model
\begin{align*}
X^{(n)}= \theta + \frac{\xi}{\sqrt{n}}, \theta \in E
\end{align*}
was studied, where $\xi$ is a Gaussian r.v. in $E$ with mean zero and covariance operator 
$\Sigma.$ 
It was assumed that $\theta$ is an unknown parameter and $\Sigma$ is known and the goal is to estimate $f(\theta)$ for a given smooth functional $f.$ The complexity of this estimation problem could be characterized by two parameters: the ``weak variance" of the noise $\xi,$
$\|\Sigma\|= \sup_{\|u\|\leq 1} {\mathbb E} \langle \xi, u\rangle^2,$ and its ``strong variance" ${\mathbb E}\|\xi\|^2 = {\mathbb E}\sup_{\|u\|\leq 1} \langle \xi,u\rangle^2.$ Note that, in the case of Euclidean space $E={\mathbb R}^d$ and $\xi\sim N(0, \sigma^2 I_d),$
$\|\Sigma\|=\sigma^2$ and ${\mathbb E}\|\xi\|^2=\sigma^2 d.$

The following result was proved.

\begin{theorem}
\label{Gaussian_shift_th}
Let $s>0.$ For $s\in (0,1],$ set $k:=0$ and for $s>1,$ let $s=k+1+\rho$ for some $k\geq 0$ and $\rho\in (0,1].$ Let $\hat \theta = \hat \theta(X^{(n)})=X^{(n)}.$
Then 
\begin{align*}
\sup_{\|f\|_{C^s(E)}\leq 1}\sup_{\theta \in E}\| f_k(\hat \theta)-f(\theta)\|_{L_2({\mathbb P}_{\theta})}
\lesssim_s \biggl(\frac{\|\Sigma\|^{1/2}}{n^{1/2}}\bigvee \biggl(\sqrt{\frac{{\mathbb E}\|\xi\|^2}{n}}\biggr)^s\biggr)
\bigwedge 1.
\end{align*}	
\end{theorem}	

Note that the term $\frac{\|\Sigma\|^{1/2}}{\sqrt{n}}$ of the error bound of Theorem \ref{Gaussian_shift_th} 
controls the concentration of estimator $f_k(\hat \theta)$ around its expectation whereas 
the term $\Bigl(\sqrt{\frac{{\mathbb E}\|\xi\|^2}{n}}\Bigr)^s$ controls the bias of this estimator. 
Moreover, it  was also shown in \cite{Koltchinskii_Zhilova} that, for $E={\mathbb R}^d$ equipped with 
the standard Euclidean norm and 
$\xi\sim N(0, \sigma^2 I_d),$
\begin{align*}
	&
	\nonumber
	\sup_{\|f\|_{C^s({\mathbb R}^d)}\leq 1}\inf_{T}\sup_{\theta\in {\mathbb R}^d}
	\|T(X^{(n)})-f(\theta)\|_{L_{2}({\mathbb P}_{\theta})}
	\asymp 
	\biggl(\frac{\|\Sigma\|^{1/2}}{n^{1/2}}
	\bigvee \biggl(\sqrt{\frac{{\mathbb E}\|\xi\|^2}{n}}\biggr)^s
	\biggr)\bigwedge 1,
\end{align*}	
where the infimum is taken over all estimators $T(X^{(n)}),$
implying the minimax optimality of the $L_2$ error rates in the case of Gaussian shift model 
in the Euclidean space $E={\mathbb R}^d.$

Note that the convergence rate is of the order $O(n^{-1/2})$ if 
$\|\Sigma\|\lesssim 1,$ ${\mathbb E}\|\xi\|^2\lesssim n^{\alpha}$ for $\alpha \in (0,1)$
and $s\geq \frac{1}{1-\alpha}$ and it is slower than $n^{-1/2}$ if $s<\frac{1}{1-\alpha}.$
For $s>\frac{1}{1-\alpha},$ it was proved in \cite{Koltchinskii_Zhilova} that 
$\sqrt{n}(f_k(\hat \theta)-f(\theta))$ could be approximated in distribution by $\sigma_f(\theta)Z, Z\sim N(0,1)$ as $n\to \infty,$ where $\sigma_f^2(\theta):= \langle \Sigma f'(\theta), f'(\theta)\rangle,$ 
and, moreover, it was shown that $f_k(\hat \theta)$ is an asymptotically efficient estimator.

We will try to extend some of these results to general models and general estimators 
$\hat \theta$ for which Gaussian approximation holds. 

\subsection{Bounds on $L_{\psi}$-errors and normal approximation of $f_k(\hat \theta)$}

To describe Gaussian approximation property for estimator $\hat \theta$ more precisely, let 
\begin{align*}
G(\theta) := \theta+ \frac{\xi(\theta)}{\sqrt{n}}, \theta \in T, 
\end{align*}
where $\xi :T \mapsto E$ is a Gaussian stochastic process. 
In what follows, $\tilde \theta := G(\theta), \theta \in T$ will be viewed as a Gaussian approximation 
of estimator $\hat \theta.$ 
In other words, the estimator $\hat \theta$ in the initial model is approximated 
by the ``estimator" $\tilde \theta$ in a Gaussian shift model with unknown parameter 
$\theta$ and small Gaussian noise $\frac{\xi (\theta)}{\sqrt{n}}.$ 
For simplicity, we also assume that ${\mathbb E} \xi(\theta)=0, \theta\in T$
and let $\Sigma(\theta)$ denote the covariance operator of random variable $\xi(\theta).$
As a typical example, consider the case when $E:={\mathbb R}^d$ and $\xi(\theta):= A(\theta)Z,
Z\sim N(0,I_d),$ where $A(\theta):{\mathbb R}^d\mapsto {\mathbb R}^d$ is a bounded 
linear operator. Even more specifically, when $X^{(n)}=(X_1,\dots, X_n),$ $X_1,\dots, X_n$
being i.i.d. $\sim P_{\theta}, \theta \in T\subset {\mathbb R}^d,$ one can think of 
the maximum likelihood estimator $\hat \theta$ and $A(\theta)=I(\theta)^{-1/2},$
where $I(\theta)$ is the Fisher information matrix (since in the case of regular statistical 
models $\sqrt{n}(\hat \theta-\theta)$ is close in distribution to $I(\theta)^{-1/2} Z$).

In the results stated below, the $L_{\psi}$-error of estimator $f_k(\hat \theta)$ and 
its normal approximation will be controlled uniformly in a subset $\Theta$ of parameter 
space $T.$ It will be assumed that $\xi(\theta)$ is bounded or even sufficiently smooth in 
a small neighborhood 
\begin{align*}
	\Theta_{\delta}:= \{\theta\in E: {\rm dist}(\theta;\Theta)<\delta\}\subset T
\end{align*}
of set $\Theta$ for some $\delta>0,$ and, moreover, that the normal approximation of $\hat \theta$
by $\tilde \theta,$ or of $\sqrt{n}(\hat \theta-\theta)$ by $\xi(\theta)$ holds in proper distances 
uniformly in $\theta \in \Theta_{\delta}.$ 
The behavior of the process $\xi(\theta)$ outside of $\Theta_{\delta}$ 
will be of no importance for us, and,
without loss of generality, we can and will set  
$\xi(\theta):=0, \theta\in E\setminus \Theta_{\delta}.$
With this definition, we still have 
that $\|\xi\|_{L_{\infty}(E)}=\|\xi\|_{L_{\infty}(\Theta_{\delta})}.$

In what follows, we will deal with loss functions $\psi:{\mathbb R}\mapsto {\mathbb R}_+.$
It will be assumed that $\psi$ is convex with $\psi(0)=0.$ Moreover, $\psi $ is even, increasing on ${\mathbb R}_+$ and 
satisfies the condition 
\begin{align*} 
	c' u \leq \psi (u)\leq c'' \psi_1(u), u\geq 0
\end{align*}
with some constants $c', c''>0,$
where $\psi_1(u)=e^u-1.$ Let $\Psi$ be the set of such loss functions. 
Given $\psi\in \Psi,$ denote 
\begin{align*}
	\tilde \psi (u):= \frac{1}{\psi^{-1}\Bigl(\frac{1}{u}\Bigr)}, u\geq 0.
\end{align*}
For instance, in the case of $\psi(u)=|u|^p, p\geq 1,$ we have $\tilde \psi(u)= u^{1/p}, u\geq 0,$
and in the case of $\psi(u)=\psi_1(u)=e^{|u|}-1,$ we have $\tilde \psi(u)=\frac{1}{\log(1+\frac{1}{u})}, u\geq 0.$

For $\psi\in \Psi,$ we will study Orlicz norm error rates $\|f_k(\hat \theta)-f(\theta)\|_{L_{\psi}({\mathbb P}_{\theta})}$ of estimator $f_k(\hat \theta)$ depending on the smoothness of functional $f.$ We will also study normal approximation of r.v. $\sqrt{n}(f_{k}(\hat \theta)-f(\theta)).$
The choice of $k$ depends on the degree of smoothness of functional $f.$
Namely, if $f$ is $C^s$-smooth with $s=k+1+\rho,$ $k\geq 0, \rho\in (0,1],$
we will use estimator $f_k(\hat \theta).$ Note that, for $k=0,$ we have $f_0=f$ and one can use 
a standard plug-in estimator $f(\hat \theta)$ for all $s\in (0,2].$ First, we will state the results in 
this simple case. 

Given $\Theta\subset T,$ denote 
\begin{align*}
{\frak v}_{\xi}(\Theta):= \sup_{\theta\in \Theta}{\mathbb E}\|\xi(\theta)\|^2.	
\end{align*}

\begin{theorem}
\label{simple_case_sleq2_a}	
Let $\Theta\subset T$ be an open subset and let $\psi\in \Psi.$
The following statements hold:\\
(i) For all $s\in (0,1],$ 
\begin{align}
\label{bd_for_sleq1}	
\sup_{\|f\|_{C^s(E)}\leq 1}\sup_{\theta\in \Theta}\|f(\hat \theta)-f(\theta)\|_{L_{\psi}({\mathbb P}_{\theta})}
\lesssim_{s,\psi} \biggl(\sqrt{\frac{{\frak v}_{\xi}(\Theta)}{n}}\Biggr)^s + 
\Delta_{{\mathcal H},\psi,\Theta} (\hat \theta,\tilde \theta),
\end{align}	
where ${\mathcal H}:=\{g: \|g\|_{C^s(E)}\leq 1\}.$
\\
(ii) Let $\delta>0$ be such that $\Theta_{\delta}\subset T.$ 
For $s=1+\rho$ with $\rho\in (0,1],$ there exists a constant $c_s\in (0,1)$
such that 
\begin{align}
\label{bd_for_s>1}	
	&
	\nonumber
	\sup_{\|f\|_{C^s(\Theta_{\delta})}\leq 1}\sup_{\theta\in \Theta}
	\|f(\hat \theta)-f(\theta)\|_{L_{\psi}({\mathbb P}_{\theta})}
	\\
	&
	\nonumber
	\lesssim_{s,\psi}
	\biggl[\frac{\|\Sigma\|^{1/2}_{L_{\infty}(\Theta)}}{n^{1/2}}+
	\biggl(\sqrt{\frac{{\frak v}_{\xi}(\Theta)}{n}}\biggr)^s 
	+\Delta_{{\mathcal H},\psi, \Theta_{\delta}}^{+}(\hat \theta,\tilde \theta)
	\\
	&
	\ \ \ \ \ \ \ \ 
	+
	\sup_{\theta\in \Theta}	\tilde \psi^{1/2} \Bigl({\mathbb P}\{\|\xi(\theta)\|\geq c_s\delta \sqrt{n}\}\Bigr)\biggr]\bigwedge 1,
\end{align}	
where ${\mathcal H}:=\{g: \|g\|_{C^s(\Theta_{c_s\delta})}\leq 1\}.$
\end{theorem}

\begin{remark}
\label{rem_on_psi}
	\normalfont	
	For $\Theta=T=E,$ bound \eqref{bd_for_s>1} of Theorem \ref{simple_case_sleq2_a} simplifies as follows:
	\begin{align*}
		&
		\nonumber
		\sup_{\|f\|_{C^s(E)}\leq 1}\sup_{\theta\in E}
		\|f(\hat \theta)-f(\theta)\|_{L_{\psi}({\mathbb P}_{\theta})}
		\lesssim_{s, \psi} 
		\biggl[\frac{\|\Sigma\|^{1/2}_{L_{\infty}(E)}}{n^{1/2}}
		+
		\biggl(\sqrt{\frac{{\frak v}_{\xi}(E)}{n}}\biggr)^s
		+\Delta_{{\mathcal H},\psi, E}^{+}(\hat \theta,\tilde \theta)\biggr]\bigwedge 1,
	\end{align*}		
	where ${\mathcal H}:=\{g: \|g\|_{C^s(E)}\leq 1\}.$
	
In the general case, there are additional terms in the bounds depending on tail probabilities
of $\|\xi(\theta)\|.$ Note that under the assumption that ${\frak v}_{\xi}(E)\leq c_1'\delta^2 n$
for small enough constant $c_1'>0,$ it easily follows from the Gaussian concentration inequality
that 
\begin{align*}
	{\mathbb P}\{\|\xi(\theta)\|\geq c_s\delta \sqrt{n}\}\leq \exp\biggl\{-\frac{c_1'' \delta^2 n}{\|\Sigma\|_{L_{\infty}(\Theta)}}\biggr\}, \theta\in \Theta.
\end{align*}
Since for $\psi \in \Psi,$ $\psi(u)\lesssim \psi_1(u), u\geq 0$ and $\tilde \psi_1(u)=\frac{1}{\log(1+\frac{1}{u})},$ it is easy to conclude that 
\begin{align*}
	\sup_{\theta\in \Theta}	\tilde \psi^{1/2} \Bigl({\mathbb P}\{\|\xi(\theta)\|\geq c_s\delta \sqrt{n}\}\Bigr)
	\lesssim \frac{1}{\delta}\frac{\|\Sigma\|_{L_{\infty}(\Theta)}^{1/2}}{n^{1/2}}.	
\end{align*}
Thus, in the worst case, the additional term in bound \eqref{bd_for_s>1}	
is of the same order (up to a factor $\frac{1}{\delta}$) as the term  $\frac{\|\Sigma\|_{L_{\infty}(\Theta)}^{1/2}}{n^{1/2}}$ present in the optimal 
bounds in the Gaussian case. For slower growing losses, this additional term becomes 
negligible. For instance, for the loss $\psi(u)=u^p, u>0, p\geq 1,$ it is dominated by 
$\exp\biggl\{-\frac{c''}{p}\frac{\delta^2 n}{\|\Sigma\|_{L_{\infty}(\Theta)}}\biggr\}$ for some 
constant $c''>0,$ so, it decays exponentially fast as $n\to\infty.$ 
Note that constants $c_1', c_1'', c''$ might depend on $s.$			
\end{remark}

The next result provides bounds on normal approximation of the error $f(\hat \theta)-f(\theta)$
for functionals $f$ of smoothness $s\in (1,2].$
Recall that for a Fr\'echet differentiable functional $f,$ 
\begin{align*}
	\sigma_{f}^2(\theta):= \langle \Sigma(\theta)f'(\theta),f'(\theta)\rangle. 
\end{align*}

\begin{theorem}
\label{simple_case_s<2_B}	
Let $s=1+\rho$ with $\rho\in (0,1]$ and let $\delta>0.$
Let $\Theta$ be a subset of $E$ such that $\Theta_{\delta}\subset T.$ 
Suppose, for some sufficiently 
small constant $c_1>0,$
$
{\frak v}_{\xi}(\Theta)\leq c_1 n.
$
Then,  for some constant $c_s\in (0,1),$ the 
following bounds hold.

(i) For all $\psi\in \Psi,$ 	
	\begin{align}
\label{bd_psi_s_in_1_2}		
	&
	\nonumber
	\sup_{\|f\|_{C^s(\Theta_{\delta})}\leq 1}\sup_{\theta\in \Theta}
	\Bigl|\|f(\hat \theta)-f(\theta)\|_{L_\psi({\mathbb P}_{\theta})}-n^{-1/2}\sigma_f(\theta)\|Z\|_{L_{\psi}({\mathbb P})}\Bigr|
	\\
	&
	\lesssim_{s, \psi}  
\biggl(\sqrt{\frac{{\frak v}_{\xi}(\Theta)}{n}}\biggr)^s 
+\Delta_{{\mathcal H},\psi, \Theta_{\delta}}^{+}(\hat \theta,\tilde \theta)
+\sup_{\theta\in \Theta}	\tilde \psi^{1/2} \Bigl({\mathbb P}\{\|\xi(\theta)\|\geq c_s\delta \sqrt{n}\}\Bigr),
\end{align}
where ${\mathcal H}:=\{g: \|g\|_{C^s(\Theta_{c_s \delta})}\leq 1\}.$

(ii) For all $s'\in [1,s],$ 
\footnote{Here and in what follows, $U_r:=\{x\in E: \|x\|< r\}.$}
\begin{align}
\label{bd_Delta_s'_s_in_1_2}	
	&
	\nonumber
	\sup_{\|f\|_{C^s(\Theta_{\delta})}\leq 1}
	\sup_{\theta\in \Theta}\Delta_{s'}(\sqrt{n}(f(\hat \theta)-f(\theta)), \sigma_{f}(\theta)Z)
	\\
	&
	\lesssim_{s} 
	\biggl[
	\sqrt{n}\biggl(\sqrt{\frac{{\frak v}_{\xi}(\Theta)}{n}}\biggr)^s 
	+\Delta_{\mathcal F, \Theta_{\delta}}(\sqrt{n}(\hat \theta -\theta), \xi(\theta)) 	
	+
	\sqrt{n}
	\sup_{\theta\in \Theta} {\mathbb P}^{1/4}\{\|\xi(\theta)\|\geq c_s\delta \sqrt{n}\}
	\biggr],
\end{align}
where ${\mathcal F}:= \{g: \|g\|_{C^{0,s'}(U_{c_s\delta\sqrt{n}})}\leq 1\}.$
\end{theorem}

\begin{remark}
\normalfont 	
For $\Theta=T=E,$ 
under condition ${\frak v}_{\xi}(E)\leq c_1 n$ for a small enough constant $c_1>0,$  	
the bounds of Theorem \ref{simple_case_s<2_B} simplify as follows:
\begin{align*}
	&
	\nonumber
	\sup_{\|f\|_{C^s(E)}\leq 1}\sup_{\theta\in E}
	\Bigl|\|f(\hat \theta)-f(\theta)\|_{L_\psi({\mathbb P}_{\theta})}-n^{-1/2}\sigma_f(\theta)\|Z\|_{L_{\psi}({\mathbb P})}\Bigr|
	\\
	&
	\nonumber
	\lesssim_{s, \psi} 
	\biggl(\sqrt{\frac{{\frak v}_{\xi}(E)}{n}}\biggr)^s 
	+\Delta_{{\mathcal H},\psi, E}^{+}(\hat \theta,\tilde \theta)
\end{align*}
with ${\mathcal H}:=\{h: \|h\|_{C^s(E)}\leq 1\},$ and 
\begin{align*}
	&
	\nonumber
	\sup_{\|f\|_{C^s(E)}\leq 1}\sup_{\theta \in E}\Delta_{s'}(\sqrt{n}(f(\hat \theta)-f(\theta)), \sigma_{f}(\theta)Z)
	\\
	&
	\lesssim_{s} 
	\biggl[
	\sqrt{n}\biggl(\sqrt{\frac{{\frak v}_{\xi}(E)}{n}}\biggr)^s 
	+\Delta_{\mathcal F, E}(\sqrt{n}(\hat \theta -\theta), \xi(\theta))
	\biggr],
\end{align*}
where ${\mathcal F}:= \{g: \|g\|_{C^{0,s'}(E)}\leq 1\}.$
\end{remark}

The problem becomes much more difficult in the case when $s=k+1+\rho>2$ 
($k\geq 1, \rho\in (0,1]$).
In this case, $f_k(\hat \theta)$ is no longer a standard plug-in estimator and a non-trivial 
analysis of its bias is needed (see Section \ref{sec:boostrap_chain}).
This analysis requires some smoothness assumptions on the Gaussian stochastic process $\xi(\theta).$  Namely, instead of quantity $\frak{v}_{\xi}(\Theta),$ we will use 
such quantities as
\begin{align*} 
{\frak d}_{\xi}(\Theta; s):= {\mathbb E}\|\xi\|_{C^s(\Theta)}^2.
\end{align*}
Note that, if ${\frak d}_{\xi}(\Theta;s)<\infty,$ then, for all $p\geq 1,$ 
\begin{align*}
	{\mathbb E}^{1/p}\|\xi\|^p_{C^s(\Theta)}\lesssim_p \sqrt{{\frak d}_{\xi}(\Theta;s)},
\end{align*}
which easily follows from Gaussian concentration. 
Note also that, if $\xi(\theta)= A(\theta) Z,$ where $Z$ is a given Gaussian vector 
in $E$ and $\Theta\ni \theta \mapsto A(\theta)\in L(E)$ is a $C^s$ function with  
values in the space $L(E)$ of bounded linear operators in $E,$ then 
\begin{align*}
	{\frak d}_{\xi}(\Theta;s) \leq \|A\|_{C^s(\Theta)}^2 {\mathbb E}\|Z\|^2.
\end{align*}
In particular, if $E={\mathbb R}^d$ (equipped with the Euclidean norm) and $Z\sim N(0,I_d),$
then 
\begin{align}
\label{bound_d_xi_d}
	{\frak d}_{\xi}(\Theta;s) \leq \|A\|_{C^s(\Theta)}^2 d.
\end{align}
In such cases, the conditions in term of ${\frak d}_{\xi}(\Theta;s)$ can be reduced to smoothness 
assumptions on the ``scaling" operator $A(\theta)$ (which is related to regularity properties of covariance $\Sigma(\theta)$ as a function of $\theta$).

If $\Theta=T=E,$ we will use the notation 
$
{\frak d}_{\xi}(s):= {\frak d}_{\xi}(E;s).
$
In what follows, such quantities will be used as complexity parameters 
in our problem.

We are now ready to state the main results of the paper.
The next theorem provides a bound on the $L_{\psi}$-error of estimator $f_k(\hat \theta)$
for $\psi\in \Psi.$ Recall that it is assumed that $\xi(\theta)=0$ outside of the neighborhood 
$\Theta_{\delta}$ specified in the theorems.

\begin{theorem}
	\label{risk_efficient_psi}
	Let $\Theta \subset T,$ let $\delta>0$
	and let $s=k+1+\rho$ with $k\geq 1$ and $\rho\in (0,1].$  
	Suppose that $\Theta_{\delta}\subset T.$ 
	Then, for all $\psi\in \Psi$ and for some constant $c_s\in (0,1),$
\begin{align}
\label{psi_error_bdd}
&
\nonumber
\sup_{\|f\|_{C^s(\Theta_{\delta})}\leq 1}\sup_{\theta\in \Theta}
\|f_k(\hat \theta)-f(\theta)\|_{L_{\psi}({\mathbb P}_{\theta})}
\\
&
\nonumber
\lesssim_s
\biggl[\frac{\|\Sigma\|^{1/2}_{L_{\infty}(E)}}{n^{1/2}}+
\biggl(\sqrt{\frac{{\frak d}_{\xi}(\Theta_{\delta};s-1)}{n}}\biggr)^s 
+\Delta_{{\mathcal H},\psi, \Theta_{\delta}}^{+}(\hat \theta,\tilde \theta)
\\
&
+
\sup_{\theta\in \Theta_{\delta}}\tilde \psi\Bigl({\mathbb P}_{\theta}\{\|\hat \theta-\theta\|\geq c_s\delta\}\Bigr) + 
\tilde \psi^{1/2}\Bigl({\mathbb P}\{\|\xi\|_{L_{\infty}(E)}\geq c_s\delta \sqrt{n}\}\Bigr)\biggr]\bigwedge 1,
\end{align}	
where ${\mathcal H}:=\{g: \|g\|_{C^s(\Theta_{c_s\delta})}\leq 1\}.$
\end{theorem}

Next we study normal approximation of estimator $f_{k}(\hat \theta).$

\begin{theorem}
	\label{Main_Th_AA}
	Let $s=k+1+\rho$ with $k\geq 1$ and $\rho\in (0,1],$ and let $\delta>0.$
	Let $\Theta$ be a subset of $T$ such that $\Theta_{\delta}\subset T.$ 
Suppose that, for some sufficiently small constant $c_1>0,$
	\begin{align*} 
		{\frak d}_{\xi}(\Theta_{\delta};s)\leq c_1 n.
	\end{align*}
Then, the following statements hold.

(i) For all $\psi\in \Psi$ and some constant $c_s\in (0,1),$		
\begin{align}
	\label{psi_error}
	&
	\nonumber
	\sup_{\|f\|_{C^s(\Theta_{\delta})}\leq 1}\sup_{\theta\in \Theta}
	\Bigl|\|f_k(\hat \theta)-f(\theta)\|_{L_\psi({\mathbb P}_{\theta})}-n^{-1/2}\sigma_f(\theta)\|Z\|_{L_{\psi}({\mathbb P})}\Bigr|
	\\
	&
	\nonumber
	\lesssim_{s, \psi}  
	\biggl(\sqrt{\frac{{\frak d}_{\xi}(\Theta_{\delta};s-1)}{n}}\biggr)^s 
	+
	\frac{\|\Sigma\|_{L_{\infty}(E)}^{1/2}}{n^{1/2}}\sqrt{\frac{{\frak d}_{\xi}(\Theta_{\delta};s-1)}{n}}
	+\Delta_{{\mathcal H},\psi, \Theta_{\delta}}^{+}(\hat \theta,\tilde \theta)
	\\
	& 
	+\sup_{\theta\in \Theta_{\delta}}\tilde \psi\Bigl({\mathbb P}_{\theta}\{\|\hat \theta-\theta\|\geq c_s\delta\}\Bigr) + 
	\tilde \psi^{1/2}\Bigl({\mathbb P}\{\|\xi\|_{L_{\infty}(E)}\geq c_s\delta \sqrt{n}\}\Bigr),
\end{align}
where ${\mathcal H}:=\{g:\|g\|_{C^s(\Theta_{c_s \delta})}\leq 1\}.$

(ii)
For all $s'\in [1,s]$ 
and some constant $c_s\in (0,1),$
	\begin{align}
		\label{bd_Delta_s'_sgeq2}
		&
		\nonumber
		\sup_{\|f\|_{C^s(\Theta_{\delta})}\leq 1}
		\sup_{\theta\in \Theta}\Delta_{s'}(\sqrt{n}(f_{k}(\hat \theta)-f(\theta)), \sigma_{f}(\theta)Z)
		\\
		&
		\nonumber
		\lesssim_{s} 
		\biggl[
		\sqrt{n}\biggl(\sqrt{\frac{{\frak d}_{\xi}(\Theta_{\delta};s-1)}{n}}\biggr)^s 
		+
		\|\Sigma\|_{L_{\infty}(E)}^{1/2}\sqrt{\frac{{\frak d}_{\xi}(\Theta_{\delta};s-1)}{n}}
		+
		\Delta_{\mathcal F, \Theta_{\delta}}(\sqrt{n}(\hat \theta -\theta), \xi(\theta)) 
		\\
		&
		+\sqrt{n}\sup_{\theta \in \Theta_{\delta}}{\mathbb P}_{\theta}\{\|\hat \theta-\theta\|\geq c_s\delta\}
		+ \sqrt{n}{\mathbb P}^{1/4}\{\|\xi\|_{L_{\infty}(E)}\geq c_s\delta \sqrt{n}\}
		\biggr],
	\end{align}
where ${\mathcal F}:= \{g: \|g\|_{C^{0,s'}(U_{c_s\delta\sqrt{n}})}\leq 1\}.$
\end{theorem}

\begin{remark}
	\normalfont	
	For $\Theta=T=E,$ the bounds of Theorem \ref{risk_efficient_psi} simplify as follows:
	\begin{align}
		\label{psi_error_bdd'''}
		&
		\nonumber
		\sup_{\|f\|_{C^s(E)}\leq 1}\sup_{\theta\in E}
		\|f_k(\hat \theta)-f(\theta)\|_{L_{\psi}({\mathbb P}_{\theta})}
		\\
		&
		\lesssim_{s, \psi} 
		\biggl[\frac{\|\Sigma\|^{1/2}_{L_{\infty}(E)}}{n^{1/2}}
		+
		\biggl(\sqrt{\frac{{\frak d}_{\xi}(s-1)}{n}}\biggr)^s 
		+\Delta_{{\mathcal H},\psi, E}^{+}(\hat \theta,\tilde \theta)\biggr]\wedge 1,
	\end{align}	
where ${\mathcal H}:=\{g: \|g\|_{C^s(E)}\leq 1\}.$

Assume that, for some sufficiently small constant $c_1>0,$ ${\frak d}_{\xi}(s)\leq c_1 n.$  
Then, for all $\psi\in \Psi,$ the following versions of bounds of Theorem \ref{Main_Th_AA}	 hold:
\begin{align}
	\label{psi_error'''}
	&
	\nonumber
	\sup_{\|f\|_{C^s(E)}\leq 1}\sup_{\theta\in E}
	\Bigl|\|f_k(\hat \theta)-f(\theta)\|_{L_\psi({\mathbb P}_{\theta})}-n^{-1/2}\sigma_f(\theta)\|Z\|_{L_{\psi}({\mathbb P})}\Bigr|
	\\
	&
	\lesssim_{s, \psi} 
	\biggl(\sqrt{\frac{{\frak d}_{\xi}(s-1)}{n}}\biggr)^s 
	+
	\frac{\|\Sigma\|_{L_{\infty}(E)}^{1/2}}{n^{1/2}} \sqrt{\frac{{\frak d}_{\xi}(s-1)}{n}}
	+\Delta_{{\mathcal H},\psi, E}^{+}(\hat \theta,\tilde \theta)
\end{align}
with ${\mathcal H}:=\{g: \|g\|_{C^s(E)}\leq 1\},$	
and, for all $s'\in [1,s],$ 
\begin{align*}
	&
	\nonumber
	\sup_{\|f\|_{C^s(E)}\leq 1}\sup_{\theta \in E}\Delta_{s'}(\sqrt{n}(f_{k}(\hat \theta)-f(\theta)), \sigma_{f}(\theta)Z)
	\\
	&
	\lesssim_{s} 
	\biggl[
	\sqrt{n}\biggl(\sqrt{\frac{{\frak d}_{\xi}(s-1)}{n}}\biggr)^s 
	+
	\|\Sigma\|_{L_{\infty}(E)}^{1/2} \sqrt{\frac{{\frak d}_{\xi}(s-1)}{n}}
	+\Delta_{\mathcal F, E}(\sqrt{n}(\hat \theta -\theta), \xi(\theta))
	\biggr],
\end{align*}
where ${\mathcal F}:= \{g: \|g\|_{C^{0,s'}(E)}\leq 1\}.$

In the general case, there are additional terms depending on tail probabilities of $\|\xi\|_{L_{\infty}(E)}$ and $\|\hat \theta-\theta\|.$
The term $\tilde \psi^{1/2}\Bigl({\mathbb P}\{\|\xi\|_{L_{\infty}(E)}\geq c_s\delta \sqrt{n}\}\Bigr)$ is negligible (smaller than $n^{-1/2}$) for losses $\psi$ that grow slower 
than sub-exponential loss $\psi_1$ (see Remark \ref{rem_on_psi}). For the term 
$$\sup_{\theta\in \Theta_{\delta}}\tilde \psi\Bigl({\mathbb P}_{\theta}\{\|\hat \theta-\theta\|\geq c_s\delta\}\Bigr)$$
to be of the order $O(n^{-1/2}),$ some conditions on the tail probabilities $$\sup_{\theta\in \Theta_{\delta}}{\mathbb P}_{\theta}\{\|\hat \theta-\theta\|\geq c_s\delta\},$$ ranging from polynomial decay 
in the case of $L_p$-losses $\psi(u)=u^p$ to exponential decay in the case of sub-exponential losses, are needed. 
In some cases, it is possible to reduce these conditions to the conditions on the tails of $\|\xi\|_{L_{\infty}(E)}$
using normal approximation (see Corollary \ref{Main_Th_AA_cor_cor}).
\end{remark}

\begin{remark}
\label{rem_on_all_bounds}
\normalfont
The bounds of theorems \ref{simple_case_sleq2_a}, \ref{simple_case_s<2_B},  \ref{risk_efficient_psi} and \ref{Main_Th_AA}
show that the estimator $f_k (\hat \theta)$
of $f(\theta)$ exhibits the same type of behavior as in the case of Gaussian shift model studied in \cite{Koltchinskii_Zhilova} (see also Theorem \ref{Gaussian_shift_th} at the beginning of this section 
and the discussion that follows) provided that normal approximation of $\hat \theta,$ quantified by such parameters as 
$$
\Delta_{{\mathcal H},\psi, \Theta_{\delta}}^{+}(\hat \theta,\tilde \theta)\ \ {\rm and}\ \ \Delta_{\mathcal F, \Theta_{\delta}}(\sqrt{n}(\hat \theta -\theta), \xi(\theta)),
$$ 
is sufficiently accurate. 

\begin{enumerate}
\item The {\it ``Gaussian parts"} of the bounds of these theorems, such as the part
$$
\frac{\|\Sigma\|^{1/2}_{L_{\infty}(E)}}{n^{1/2}}+
\biggl(\sqrt{\frac{{\frak d}_{\xi}(\Theta_{\delta};s-1)}{n}}\biggr)^s 
$$ 
of bound \eqref{psi_error_bdd}, are similar to the main part
$\frac{\|\Sigma\|^{1/2}}{n^{1/2}}\bigvee \biggl(\sqrt{\frac{{\mathbb E}\|\xi\|^2}{n}}\biggr)^s$
of the bound of Theorem \ref{Gaussian_shift_th}. 
They consist of two terms: the concentration term, such as $\frac{\|\Sigma\|^{1/2}_{L_{\infty}(E)}}{n^{1/2}},$
controlling the random error of estimator $f_k(\hat \theta)$ and the bias term, such as $\Bigl(\sqrt{\frac{{\frak d}_{\xi}(\Theta_{\delta};s-1)}{n}}\Bigr)^s,$ 
controlling the bias of the estimator. 
In fact, the Gaussian parts are exactly the same as in the case of Gaussian shift model when $\xi(\theta)$ does not depend on $\theta$ and it is possible 
to obtain Theorem \ref{Gaussian_shift_th} for Gaussian shift models as a corollary of our general results, see Corollary \ref{cor_Gauss_shift} below. 

\item In typical examples, such as $\xi(\theta)=A(\theta)Z, Z\sim N(0,I_d),$ complexity parameters $\frak{v}_{\xi}(\Theta)$ and ${\frak d}_{\xi}(\Theta_{\delta};s-1)$ could be easily controlled 
in terms of some dimension type parameter $d$ (see, e.g.,  bound \eqref{bound_d_xi_d}), and the Gaussian parts of the bounds are controlled by the expression 
$
\frac{1}{\sqrt{n}}+ \Bigl(\sqrt{\frac{d}{n}}\Bigr)^s.
$
If the normal approximation terms 
of bounds of theorems \ref{simple_case_sleq2_a}, \ref{simple_case_s<2_B},  \ref{risk_efficient_psi} and \ref{Main_Th_AA}
are negligible comparing with the Gaussian part, there is a phase transition from the classical $\frac{1}{\sqrt{n}}$
error rate when the smoothness $s$ of functional $f$ is sufficiently large to slower rates when the smoothness is not sufficient
(similarly to the case of Gaussian shift models \cite{Koltchinskii_Zhilova}).	
More precisely, if $d\leq n^{\alpha}$ for some $\alpha\in (0,1),$ then $\frac{1}{\sqrt{n}}$ error rate for estimators $f_k(\hat \theta)$ holds for all $s\geq \frac{1}{1-\alpha}$
and slower rates hold for $s<\frac{1}{1-\alpha}$ (which is known to be a sharp threshold in the case of Gaussian shift models). 
Moreover, if $s>\frac{1}{1-\alpha},$ then the bounds of theorems \ref{simple_case_s<2_B} and \ref{Main_Th_AA}
also imply normal approximation of estimator $f_k(\hat \theta).$ 
However, for Gaussian type bounds on estimator $f_k(\hat \theta)$ to hold in the whole range of values of $\alpha\in (0,1),$ 
the normal approximation of $\sqrt{n}(\hat \theta-\theta)$ by $\xi(\theta)$ should hold for $d=o(n)$
(see further discussion in Section \ref{sec: Examples_and_Applications}).

\item Finally, note that, for $s\in (0,2],$ there is no need in bias reduction to achieve the optimal (in the Gaussian case)
error rates and plug-in estimator $f(\hat \theta)$ could be used for this purpose (see theorems  \ref{simple_case_sleq2_a} and \ref{simple_case_s<2_B}).
For $s>2,$ the bias of the plug-in estimator is too large and estimators with reduced bias, such as $f_k(\hat \theta),$ are needed to achieve the optimal 
rate (see theorems \ref{risk_efficient_psi} and \ref{Main_Th_AA}).

\end{enumerate}
\end{remark}

It is not hard to obtain a generalization of results of \cite{Koltchinskii_Zhilova} to more general Gaussian shift 
models as a corollary of the results of the current paper.
Namely, suppose that $X^{(n)}$ satisfies the following Gaussian shift model
\begin{align*}
X^{(n)}= \theta + \frac{\xi(\theta)}{\sqrt{n}}, \theta \in E,
\end{align*}
where $\xi(\theta)$ is a Gaussian random variable in $E$ with mean $0$ and covariance 
operator $\Sigma(\theta), \theta \in E.$ In particular, this includes the Gaussian shift models 
studied in \cite{Koltchinskii_Zhilova} in which the noise $\xi(\theta)=\xi$ did not depend on $\theta.$
Let $\hat \theta = \hat \theta (X^{(n)})=X^{(n)}.$ 
The next corollary is immediate since $\hat \theta = \tilde \theta$ and $\sqrt{n}(\hat \theta -\theta)=\xi(\theta),$
implying that
\begin{align*}
\Delta_{\mathcal F, \Theta_{\delta}}(\sqrt{n}(\hat \theta -\theta), \xi(\theta))=\Delta_{{\mathcal H},\psi, E}^{+}(\hat \theta,\tilde \theta) 
= 0.
\end{align*}

\begin{corollary}
\label{cor_Gauss_shift}
Let $s=k+1+\rho$ with $k\geq 0$ and $\rho\in (0,1].$ For all $\psi\in \Psi,$ 
\begin{align*}
&
\nonumber
\sup_{\|f\|_{C^s(E)}\leq 1}\sup_{\theta\in E}
\|f_k(\hat \theta)-f(\theta)\|_{L_{\psi}({\mathbb P}_{\theta})}
\lesssim_{s, \psi} 
\biggl[\frac{\|\Sigma\|^{1/2}_{L_{\infty}(E)}}{n^{1/2}}
+
\biggl(\sqrt{\frac{{\frak d}_{\xi}(s-1)}{n}}\biggr)^s\biggr]\bigwedge 1.
\end{align*}
Moreover, for $k\geq 1$ under the assumption that ${\frak d}_{\xi}(s)\leq c_1 n$ for a small enough constant $c_1>0,$
\begin{align*}
	&
	\nonumber
	\sup_{\|f\|_{C^s(E)}\leq 1}\sup_{\theta\in E}
	\Bigl|\|f_k(\hat \theta)-f(\theta)\|_{L_\psi({\mathbb P}_{\theta})}-n^{-1/2}\sigma_f(\theta)\|Z\|_{L_{\psi}({\mathbb P})}\Bigr|
	\\
	&
	\lesssim_{s, \psi} 
	\biggl(\sqrt{\frac{{\frak d}_{\xi}(s-1)}{n}}\biggr)^s 
	+
	\frac{\|\Sigma\|_{L_{\infty}(E)}^{1/2}}{n^{1/2}}\sqrt{\frac{{\frak d}_{\xi}(s-1)}{n}}
\end{align*}	
and, for all $s'\in [1,s],$
\begin{align*}
&
\nonumber
\sup_{\|f\|_{C^s(E)}\leq 1}\sup_{\theta \in E}\Delta_{s'}(\sqrt{n}(f_{k}(\hat \theta)-f(\theta)), \sigma_{f}(\theta)Z)
\\
&
\lesssim_{s} 
\biggl[
\sqrt{n}\biggl(\sqrt{\frac{{\frak d}_{\xi}(s-1)}{n}}\biggr)^s 
+
\|\Sigma\|_{L_{\infty}(E)}^{1/2}\sqrt{\frac{{\frak d}_{\xi}(s-1)}{n}}
\biggr].
\end{align*}
For $k=0,$ the last two bounds hold without the terms involving $\|\Sigma\|_{L_{\infty}(E)}^{1/2}\sqrt{\frac{{\frak d}_{\xi}(s-1)}{n}}.$
\end{corollary}

In the case when the noise $\xi(\theta)=\xi$ does not depend on $\theta,$
we have 
\begin{align*}
{\frak d}_{\xi}(s-1)= {\mathbb E}\|\xi\|^2, 
\end{align*}
and the above bounds 
immediately imply the main results of paper \cite{Koltchinskii_Zhilova}.

The bounds of theorems \ref{risk_efficient_psi} and \ref{Main_Th_AA} show that the ``Gaussian 
error rates" would hold for other models provided that the additional terms related to the accuracy of normal approximation and to the tails of random variables $\|\hat \theta-\theta\|$ and $\|\xi\|_{L_{\infty}(E)}$ are negligible comparing with the Gaussian terms. To ensure this 
(and, in particular, to ensure that $\sqrt{n}$ convergence rate is attainable for estimator $f_k(\hat \theta)$ if $f$ is sufficiently smooth), one needs the conditions 
\begin{align*}
\Delta_{\mathcal F, \Theta_{\delta}}\Bigl(\sqrt{n}(\hat \theta -\theta),\xi(\theta)\Bigr)\to 0
\end{align*}
and 
\begin{align*}
\Delta_{{\mathcal H},\psi, \Theta_{\delta}}^{+}(\hat \theta,\tilde \theta)=o(n^{-1/2})
\end{align*}
as $n\to\infty.$ 

The following proposition provides useful upper bounds
on the distances $\Delta_{{\mathcal H}, \Theta_{\delta}}(\hat \theta,\tilde \theta),$
$\Delta_{{\mathcal H},\psi, \Theta_{\delta}}(\hat \theta,\tilde \theta)$ and
$\Delta_{{\mathcal H},\psi, \Theta_{\delta}}^{+}(\hat \theta,\tilde \theta).$

\begin{proposition}
\label{prop_Delta_H_Delta_F}	
Let $s\geq 1.$ For ${\mathcal H}:=\{g: \|g\|_{C^s(\Theta_{\delta})}\leq 1\},$ 
\begin{align*}
\Delta_{{\mathcal H}, \Theta_{\delta}}(\hat \theta,\tilde \theta)
\leq 
\frac{\Delta_{{\mathcal F}, \Theta_{\delta}}\Bigl(\sqrt{n}(\hat \theta-\theta), \xi(\theta)\Bigr)}{\sqrt{n}},
\end{align*}
\begin{align*}
\Delta_{{\mathcal H},\psi, \Theta_{\delta}}(\hat \theta,\tilde \theta)
\leq 
\frac{\Delta_{{\mathcal F},\psi, \Theta_{\delta}}\Bigl(\sqrt{n}(\hat \theta-\theta), \xi(\theta)\Bigr)}{\sqrt{n}},
\end{align*}
and 
\begin{align*}
\Delta_{{\mathcal H},\psi, \Theta_{\delta}}^{+}(\hat \theta,\tilde \theta)
\leq 
\frac{\Delta_{{\mathcal F},\psi, \Theta_{\delta}}^{+}\Bigl(\sqrt{n}(\hat \theta-\theta), \xi(\theta)\Bigr)}{\sqrt{n}},
\end{align*}
where 
${\mathcal F}:=\{g: \|g\|_{C^{0,s}(U_{\delta\sqrt{n}})}\leq 1\}.$
\end{proposition}

It follows that the condition $\Delta_{{\mathcal H},\psi, \Theta_{\delta}}^{+}(\hat \theta,\tilde \theta)=o(n^{-1/2})$ holds if 
$$
\Delta_{{\mathcal F},\psi, \Theta_{\delta}}^{+} \Bigl(\sqrt{n}(\hat \theta-\theta), \xi(\theta)\Bigr)
=o(1).
$$

Next we state corollaries of theorems \ref{risk_efficient_psi} and \ref{Main_Th_AA} (and, for $s\in (1,2],$ of theorems \ref{simple_case_sleq2_a} and \ref{simple_case_s<2_B})
in the case of quadratic loss $\psi(u)=u^2.$
In these corollaries, we will use Wasserstein distances $W_1,W_2$ to quantify the accuracy of normal approximation and to obtain a simpler form of the results. 

\begin{corollary}
	\label{Main_Th_AA_cor_cor}
	Let $s=k+1+\rho$ with $k\geq 1$ and $\rho\in (0,1],$ and let $\delta>0.$
	Let $\Theta$ be a subset of $E$ such that $\Theta_{\delta}\subset T.$ 
	Suppose that, for some sufficiently small constant $c_1>0,$
	\begin{align*} 
	{\frak d}_{\xi}(\Theta_{\delta};s)\leq c_1 \delta^2 n.
	\end{align*}
	Then, for all $s'\in [1,s]$ and for some constant $c_2>0,$
	\begin{align*}
	&
	\sup_{\|f\|_{C^s(\Theta_{\delta})}\leq 1}
	\sup_{\theta\in \Theta}\Delta_{s'}(\sqrt{n}(f_{k}(\hat \theta)-f(\theta)), \sigma_{f}(\theta)Z)
	\\
	&
	\lesssim_{s,\delta} 
	\biggl[
	\sqrt{n}\biggl(\sqrt{\frac{{\frak d}_{\xi}(\Theta_{\delta};s-1)}{n}}\biggr)^s 
	+
	\|\Sigma\|_{L_{\infty}(E)}^{1/2}\sqrt{\frac{{\frak d}_{\xi}(\Theta_{\delta};s-1)}{n}}
	\\
	&
	\ \ \ \ \ \ \ \ \ +W_{1,\Theta_{\delta}}(\sqrt{n}(\hat \theta -\theta), \xi(\theta)) 
	+ \sqrt{n}\exp\biggl\{-\frac{c_2\delta^2 n}{\|\Sigma\|_{L_{\infty}(E)}}\biggr\}
	\biggr].
	\end{align*}
\end{corollary}

\begin{corollary}
	\label{risk_efficient_psi_cor_cor}
	Let $\Theta \subset T,$ let $\delta>0$
	and let $s=k+1+\rho$ with $k\geq 0$ and $\rho\in (0,1].$  
	Suppose that $\Theta_{\delta}\subset T.$ 
	Then, for some constant $c_2>0,$
	\begin{align}
	\label{psi_error_bdd_W_2}
	&
	\nonumber
	\sup_{\|f\|_{C^s(\Theta_{\delta})}\leq 1}\sup_{\theta\in \Theta}
	\|f_k(\hat \theta)-f(\theta)\|_{L_2({\mathbb P}_{\theta})}
	\\
	&
	\nonumber
	\lesssim_{s,\delta}
	\biggl[\frac{\|\Sigma\|^{1/2}_{L_{\infty}(E)}}{n^{1/2}}+
	\biggl(\sqrt{\frac{{\frak d}_{\xi}(\Theta_{\delta};s-1)}{n}}\biggr)^s 
	\\
	&
	\ \ \ \ \ \ \ \ +\frac{W_{2,\Theta_{\delta}}(\sqrt{n}(\hat \theta -\theta), \xi(\theta))}{\sqrt{n}}
	+\exp\biggl\{-\frac{c_2\delta^2 n}{\|\Sigma\|_{L_{\infty}(E)}}\biggr\}
	\biggr]\bigwedge 1.
	\end{align}	
	Moreover, if 	
	for some sufficiently 
	small constant $c_1>0,$
	\begin{align*}
	{\frak d}_{\xi}(\Theta_{\delta};s)\leq c_1\delta^2 n,
	\end{align*}
	then 	
	\begin{align}
	\label{psi_error_W_2}
	&
	\nonumber
	\sup_{\|f\|_{C^s(\Theta_{\delta})}}\sup_{\theta\in \Theta}
	\Bigl|\|f_k(\hat \theta)-f(\theta)\|_{L_2({\mathbb P}_{\theta})}-n^{-1/2}\sigma_f(\theta)\Bigr|
	\\
	&
	\nonumber
	\lesssim_{s, \delta}  
	\biggl(\sqrt{\frac{{\frak d}_{\xi}(\Theta_{\delta};s-1)}{n}}\biggr)^s 
	+
	\frac{\|\Sigma\|_{L_{\infty}(E)}^{1/2}}{n^{1/2}}\sqrt{\frac{{\frak d}_{\xi}(\Theta_{\delta};s-1)}{n}}
	\\
	& 
	\ \ \ \ \ \ \ +\frac{W_{2,\Theta_{\delta}}(\sqrt{n}(\hat \theta -\theta), \xi(\theta))}{\sqrt{n}}
	+\exp\biggl\{-\frac{c_2\delta^2 n}{\|\Sigma\|_{L_{\infty}(E)}}\biggr\}.
	\end{align}
\end{corollary}

\begin{remark}
\normalfont
Bounds of corollaries \ref{Main_Th_AA_cor_cor} and \ref{risk_efficient_psi_cor_cor} also holds for $k=0.$ In this case, 
the terms involving $\|\Sigma\|_{L_{\infty}(E)}^{1/2}\sqrt{\frac{{\frak d}_{\xi}(\Theta_{\delta};s-1)}{n}}$ could be dropped.
\end{remark}

As a simple consequence, we get the following result that shows asymptotic normality 
of estimator $f_k(\hat \theta)$ with $\sqrt{n}$ rate and provides an exact limit of its mean squared error if normal approximation holds and the functional $f$ is sufficiently smooth.   

\begin{corollary}
	\label{Cor_Main_Th_AA}
	Let $\Theta = \Theta_n \subset T,$ let $\delta>0$
	and let $s=k+1+\rho$ with $k\geq 0$ and $\rho\in (0,1].$  
	Suppose that $\Theta_{\delta}\subset T$ and, for some $\alpha\in (0,1),$
	\begin{align*} 
	{\frak d}_{\xi}(\Theta_{\delta};s)\lesssim n^{\alpha}.
	\end{align*}
	Suppose also that $\|\Sigma\|_{L_{\infty}(E)}\lesssim 1.$
	Assume that $s>\frac{1}{1-\alpha}.$ Finally, suppose that
	\begin{align}
	\label{Delta_to_zero}
	W_{2,\Theta_{\delta}}\Bigl(\sqrt{n}(\hat \theta -\theta),\xi(\theta)\Bigr)\to 0\ 
	{\rm as}\ n\to\infty.
	\end{align}
Then
\begin{align}
\label{risk_eff_asymp_claim}
\sup_{\|f\|_{C^s(\Theta_{\delta})}\leq 1}\sup_{\theta\in \Theta}
\Bigl|n{\mathbb E}_{\theta}(f_k(\hat \theta)-f(\theta))^2 - \sigma_f^2(\theta)\Bigr|\to 0\ {\rm as}\ n\to\infty,
\end{align}
and, for all $\sigma_0>0,$ 
	\begin{align*}
	\sup_{\|f\|_{C^s(\Theta_{\delta})}\leq 1}
	\sup_{\theta\in \Theta, \sigma_f(\theta)\geq \sigma_0} d_K\Bigl(\frac{\sqrt{n}(f_{k}(\hat \theta)-f(\theta))}{\sigma_{f}(\theta)}, Z\Bigr)
	\to 0\ {\rm as}\ n\to\infty,
	\end{align*}
	where $Z\sim N(0,1).$\footnote{Of course, it is assumed here and in Theorem \ref{Main_Th_AA} that $\hat \theta=\hat \theta(X^{(n)}), X^{(n)}\sim P_{\theta}^{(n)}.$}
\end{corollary}

\begin{remark}
\normalfont 
Let $\Theta=T=E$ and assume that, for some sufficiently small constant $c_1>0,$ ${\frak d}_{\xi}(s)\leq c_1 n.$  
Then, for all $s'\in [1,s],$ the following version of the bound of Corollary \ref{Main_Th_AA_cor_cor} holds:
\begin{align}
\label{Delta_s'_T=E}
	&
	\nonumber
	\sup_{\|f\|_{C^s(E)}\leq 1}\sup_{\theta \in E}\Delta_{s'}(\sqrt{n}(f_{k}(\hat \theta)-f(\theta)), \sigma_{f}(\theta)Z)
	\\
	&
	\lesssim_{s} 
	\biggl[
	\sqrt{n}\biggl(\sqrt{\frac{{\frak d}_{\xi}(s-1)}{n}}\biggr)^s 
	+
	\|\Sigma\|_{L_{\infty}(E)}^{1/2}\sqrt{\frac{{\frak d}_{\xi}(s-1)}{n}}
	+W_{1,E}(\sqrt{n}(\hat \theta -\theta), \xi(\theta))
	\biggr].
\end{align}	
The bounds of Corollary \ref{risk_efficient_psi_cor_cor} simplify as follows:
\begin{align}
	\label{psi_error_bdd'''}
	&
	\nonumber
	\sup_{\|f\|_{C^s(E)}\leq 1}\sup_{\theta\in E}
	\|f_k(\hat \theta)-f(\theta)\|_{L_2({\mathbb P}_{\theta})}
	\\
	&
	\lesssim_{s} 
	\biggl[\frac{\|\Sigma\|^{1/2}_{L_{\infty}(E)}}{n^{1/2}}
	+
	\biggl(\sqrt{\frac{{\frak d}_{\xi}(s-1)}{n}}\biggr)^s 
	+\frac{W_{2,E}(\sqrt{n}(\hat \theta -\theta), \xi(\theta))}{\sqrt{n}}\biggr]\wedge 1,
\end{align}		
and, under the condition 
${\frak d}_{\xi}(s)\leq c_1 n$ for a small enough constant $c_1>0,$  	
\begin{align}
	\label{psi_error'''}
	&
	\nonumber
	\sup_{\|f\|_{C^s(E)}\leq 1}\sup_{\theta\in E}
	\Bigl|\|f_k(\hat \theta)-f(\theta)\|_{L_2({\mathbb P}_{\theta})}-n^{-1/2}\sigma_f(\theta)\Bigr|
	\\
	&
	\lesssim_{s} 
	\biggl(\sqrt{\frac{{\frak d}_{\xi}(s-1)}{n}}\biggr)^s 
	+
	\frac{\|\Sigma\|_{L_{\infty}(E)}^{1/2}}{n^{1/2}}\sqrt{\frac{{\frak d}_{\xi}(s-1)}{n}} 
	+\frac{W_{2,E}(\sqrt{n}(\hat \theta -\theta), \xi(\theta))}{\sqrt{n}}.
\end{align}
\end{remark}

\subsection{Examples and applications: estimation of functionals and normal approximation in high-dimensional spaces}
\label{sec: Examples_and_Applications}

To apply the results of Section \ref{sec:Main results} to concrete statistical models, 
one needs to use sharp bounds on the accuracy of normal approximation over classes 
of smooth functions for typical statistical estimators (such as maximum likelihood estimators) in 
a high-dimensional setting. Ideally, in the case of a $d$-dimensional parameter $\theta,$
bounds on such distances as $\Delta_{\mathcal F, \Theta_{\delta}}(\sqrt{n}(\hat \theta -\theta), \xi(\theta))$
with ${\mathcal F}:= \{g: \|g\|_{C^{0,s'}(U_{\delta\sqrt{n}})}\leq 1\}$ of the order 
$\sqrt{\frac{d}{n}},$ or $\Delta_{{\mathcal H}, \psi, \Theta_{\delta}}(\hat \theta, \tilde \theta)$
with ${\mathcal H}:=\{h: \|h\|_{C^s(\Theta_{\delta})}\leq 1\}$ of the order $n^{-1/2}\sqrt{\frac{d}{n}}$
are needed to ensure that the normal approximation 
holds for $d=o(n).$ This would allow us to deduce from Theorem \ref{risk_efficient_psi} and Theorem \ref{Main_Th_AA} 
 the results known to be optimal in the Gaussian case.
Unfortunately, such bounds are, in our view, underdeveloped in the literature,
not only in the case of general classes of estimators for high-dimensional models, 
such as MLE (see, e.g., \cite{Anastasiou_0, Anastasiou}), but even in the case of classical central limit theorems (CLT) in high-dimensional spaces (see, e.g., \cite{Portnoy} where there are counterexamples 
showing that CLT could fail for some reasonable distributions in ${\mathbb R}^d$ unless $d^2=o(n)$).
The main difficulties involved in these problems are purely probabilistic: identifying classes of distributions 
in high-dimensional spaces with a reasonably good dependence of the normal approximation bounds on the dimension. The importance of these problems in high-dimensional statistics 
goes far beyond their applications to functional estimation discussed in the current paper. 
In this section, we will provide a very brief review of some approaches to high-dimensional 
CLT (including, very recent ones) and discuss several applications to the problem of functional 
estimation. A more detailed development of this approach is beyond the scope of the paper.

\subsubsection{High-dimensional CLT}

The rates of convergence in CLT in ${\mathbb R}^d$ and in infinite-dimensional 
spaces have been studied for over fifty years (see \cite{Bhattacharya}, \cite{Paulauskas}, \cite{Senatov} and references 
therein) with a goal to obtain the bounds on the accuracy of normal approximation in various distances
in the spaces of probability distributions often represented by sup-norms over classes of sets 
(for instance, convex sets), or classes of functions (for instance, Lipschitz functions).

The distances $\zeta_s$ defined by \eqref{Zolotarev} are particularly useful for our purposes.
Such distances occur very naturally in connection to the Lindeberg's proof of CLT, they are used as a tool in bounding other distances (the sup-norms over convex sets, bounded 
Lipschitz distance, etc) and they were advocated in \cite{Zolotarev}. In particular, the following fact is 
straightforward: if $X_1,\dots, X_n$ are i.i.d. r.v. in ${\mathbb R}^d$ (equipped with the Euclidean norm) with mean zero 
and identity covariance and $Z$ is a standard normal r.v. in ${\mathbb R}^d,$ then   
\begin{align*}
\zeta_3\Bigl(\frac{X_1+\dots+X_n}{\sqrt{n}}, Z\Bigr)\lesssim \frac{{\mathbb E}\|X\|^3}{\sqrt{n}}.
\end{align*}
Typically, ${\mathbb E} \|X\|^3$ could be of the order $d^{3/2},$ yielding 
the bound on $\zeta_3$-distance of the order $\frac{d^{3/2}}{\sqrt{n}}.$ Thus, normal 
approximation holds when $d=o(n^{1/3}).$ This is not good enough for our purposes 
since an interesting regime in functional estimation problem is $d\sim n^{\alpha}$ for $\alpha\geq 1/2,$ which 
leads to non-trivial bias reduction problems. However, in some special cases, in particular, in the 
case of random vectors with independent components, one can improve bounds on $\zeta_s$-distance
rather substantially. The following fact is very simple and well known (see \cite{Zolotarev} for similar 
statements).

\begin{proposition}
\label{componentwise}
Let $Y=(Y_1,\dots, Y_d),$ $Y'=(Y_1', \dots, Y_d')$ be two random vectors with independent components.
Then 
\begin{align*}
\zeta_{s}(Y,Y') \leq \sum_{j=1}^d \zeta_{s}(Y_j, Y_j').
\end{align*}
\end{proposition}
 
As a consequence, in the case when r.v. $X=(X^{(1)}, \dots, X^{(d)})$ has independent components,
\begin{align*}
\zeta_3\Bigl(\frac{X_1+\dots+X_n}{\sqrt{n}}, Z\Bigr)\lesssim \max_{1\leq j\leq d}{\mathbb E}|X^{(j)}|^3
\frac{d}{\sqrt{n}}.
\end{align*}
and if, in addition, ${\mathbb E}(X^{(j)})^3=0,$ then it is easy to see that    
\begin{align*}
\zeta_4\Bigl(\frac{X_1+\dots+X_n}{\sqrt{n}}, Z\Bigr)\lesssim \max_{1\leq j\leq d}{\mathbb E}|X^{(j)}|^4
\frac{d}{n}.
\end{align*} 
The last bound is of the order $\frac{d}{n},$ which is already sufficient for our purposes.   
 
In Subsection \ref{sec:ind_comp} below, we use this very simple approach to study 
estimation of smooth functionals for some statistical models with independent components.  
 
In the recent years, there has been a lot of interest in studying normal approximation 
bounds in high-dimensional CLT in optimal transport distances (in particular, Wasserstein 
type distances).  
A recent result in \cite{Eldan}, provides the following bound on the Wasserstein 
$W_2$-distance in normal approximation: assuming that $\|X\|\leq \beta$ a.s.,
\begin{align*} 
W_2\Bigl(\frac{X_1+\dots+X_n}{\sqrt{n}}, Z\Bigr) \lesssim \frac{\beta \sqrt{d \log n}}{\sqrt{n}}.
\end{align*}
Thus, for convergence of $W_2$-distance to $0,$ this bound requires the condition $d=o\Bigl(\sqrt{\frac{n}{\log n}}\Bigr)$ in typical situations when $\beta \sim \sqrt{d}.$ This is again too restrictive for our purposes.

In recent papers \cite{Fathi, Fathi_1}, another approach to high-dimensional normal approximation
has been developed.  It is based on the technique of Stein kernels and it applies to 
probability distributions in ${\mathbb R}^d$ with bounded Poincar\'e constants,
in particular, to some log-concave distributions (see also \cite{Houdre} for more general results).

A probability measure $\mu$ on ${\mathbb R}^d$ is said to satisfy Poincar\'e inequality 
iff there exists a constant $C>0$ such that for all locally Lipschitz functions 
$g:{\mathbb R}^d\mapsto {\mathbb R}$ and for $X\sim \mu,$
\begin{align*}
{\rm Var}_{\mu} (g(X)) \leq C {\mathbb E}_{\mu} \|(\nabla g)(X)\|^2.
\end{align*}
Let $C_{P}(\mu)$ denote the infimum of all constants $C>0$ for which 
the inequality holds. It is called the Poincar\'e constant of probability 
measure $\mu.$ 

A probability measure (distribution) $\mu$ on ${\mathbb R}^d$ with density $p$ is called {\it log-concave} if $p$ is a log-concave function, that is, 
$\log p$ is concave. Among the examples of log-concave distributions are Gaussian measures and uniform distributions in convex bodies 
of ${\mathbb R}^d.$ It is known that log-concave distributions satisfy Poincar\'e inequality.

\begin{remark}
\label{properties_P_const}
\normalfont
The following facts are well known: 

\begin{enumerate}
\item 
For a standard Gaussian measure $\mu$ on ${\mathbb R}^d,$ $C_P(\mu)=1.$
Moreover, if $\mu(dx)=e^{-V(x)}dx$ with $V:{\mathbb R}^d \mapsto {\mathbb R}$
such that $V^{''}(x)\succeq C^{-1}$ for a symmetric positively definite matrix 
$C,$ then 
$
C_P(\mu) \leq \|C\|,
$
and, if $B$ is a symmetric positively definite matrix 
and 
$$\mu(dx)=\exp\Bigl\{-\frac{1}{2}\langle B^{-1}x,x\rangle-V(x)\Bigr\}dx,$$
where $V$ is a convex function on ${\mathbb R}^d,$ then 
$
C_P(\mu)\leq \|B\|
$
(see \cite{Bobkov}).

\item There are also ways to control the value of Poincar\'e constant under certain perturbations of 
probability measure.
For instance, if $\mu, \nu$ are two probability measures and $\mu $ is absolutely continuous with respect 
to $\nu$ with the density $\frac{d\mu}{d\nu}$ bounded from above by a constant $A>0$ and bounded 
from below by a constant $a>0,$ then 
$$
C_P(\mu)\leq \frac{A}{a} C_P(\nu).
$$
Also, if $\mu, \nu$ are log-concave measures on ${\mathbb R}^d$ and, for some $\eps\in (0,1),$ 
\begin{align*}
d_{TV}(\mu,\nu):= \sup_{A\subset {\mathbb R}^d}|\mu(A)-\nu(A)|\leq 1-\eps,
\end{align*}
then $C_P(\mu) \lesssim_{\eps} C_P(\nu)$ (see \cite{Milman}).

\item Let $\mu$ be an arbitrary log-concave distribution with covariance $\Sigma.$ 
According to the Kannan-Lov\`asz-Simonovits (KLS) conjecture,
$C_P(\mu)\lesssim \|\Sigma\|.$ Although this conjecture 
still remains open,  it was recently shown in \cite{Yuansi_Chen} 
(building upon earlier results of \cite{Eldan, Lee_Vempala})
that for some constant $c>0$
\begin{align*}
C_P(\mu)\leq d^{c(\frac{\log \log d}{\log d})^{1/2}} \|\Sigma\|.
\end{align*}
\end{enumerate}
\end{remark}

It was proved in \cite{Fathi} that, if $X_1,\dots, X_n$ are i.i.d. mean zero random variables 
with identity covariance sampled from a distribution $\mu$ on ${\mathbb R}^d$ such that $C_P(\mu)<\infty,$
then 
\begin{align}
\label{bd_Fathi_1}
W_2\Bigl(\frac{X_1+\dots+X_n}{\sqrt{n}}, Z\Bigr) \leq \sqrt{C_P(\mu)-1}\sqrt{\frac{d}{n}},
\end{align}
where $Z\sim N(0;I_d).$ Thus, the convergence in high-dimensional CLT in the $W_2$-distance 
holds provided that $d=o(n)$ for all the distributions with bounded Poincar\'e constants.  
If distribution $\mu$ is log-concave, then it follows from the bound on Poincar\'e constant 
proved in \cite{Yuansi_Chen} (see Remark \ref{properties_P_const}) that the CLT holds 
provided that $d\leq n^{1-\delta}$ for an arbitrary $\delta>0.$

This approach will be used in Subsection \ref{Poincare_constant} below to study 
smooth functional estimation for some classes of log-concave and related models.

\begin{remark}
\normalfont	
Another interesting approach to high-dimensional normal approximation was initiated 	
in \cite{Chernozhukov}. In this paper, the authors were trying to overcome the ``curse of dimensionality" in CLT
by sacrificing the convergence rate with respect to $n.$ Namely, they proved the bound on the 
accuracy of normal approximation in sup-norm over the class of hyperrectangles of the order 
$O\Bigl(\Bigl(\frac{\log^7 (dn)}{n}\Bigr)^{1/6}\Bigr),$ implying that normal approximation holds 
provided that $\log^7 d=o(n).$ More recently, this result was improved in \cite{Kuchib, Koike, Chernozhukov_2020, Chernozhukov_2022}. 
In particular, it was shown in \cite{Chernozhukov_2022} that the accuracy of normal approximation over hyperrectangles is of the order 
$O\Bigl(\frac{\log^{3/2}d}{\sqrt{n}}\log n\Bigr),$ which is optimal up to a $\log n$ factor. 
Thus, the normal approximation holds when $\log^3 d=o\Bigl(\frac{n}{\log^2 n}\Bigr).$ In principle, the results 
of this type could be adapted for our purposes in the case when $E$ is the space ${\mathbb R}^d$
equipped with the $\ell_{\infty}$-norm. However, in this case ${\mathbb E}\|\xi\|_{\ell_{\infty}}^2$
would be typically of the order $\log d$ and this would be also a typical size of such parameters
as  ${\frak d}_{\xi}(s)$ involved in our bounds. 
Thus, the Gaussian part of the error bounds in functional estimation (see Remark \ref{rem_on_all_bounds})
would be of the order $\frac{1}{\sqrt{n}}+ \Bigl(\sqrt{\frac{\log d}{n}}\Bigr)^s.$ 
If $\log d=o(n^{\alpha})$ for some $\alpha<1/3,$ the classical rate $n^{-1/2}$ dominates the bias term $\Bigl(\sqrt{\frac{\log d}{n}}\Bigr)^s$ 
for all $s\geq 3/2$ and there is no improvement of the rate when the degree of smoothness $s$ is above $3/2.$
Moreover, for such low values of $\log d,$ the bias reduction is not required and the optimal rates would be 
attained for plug-in estimators.  
One would need to have normal approximation in high-dimensional CLT in the distances relevant in our paper for $\log d=o(n)$ 
to take the full advantage of the bias reduction method in the whole range of smoothness of the functionals.
However, such normal approximation results do not seem to be 
available in the current literature.  
\end{remark}

\subsubsection{Independent components}
\label{sec:ind_comp}

We will start this section with an application of  Corollary \ref{Main_Th_AA_cor_cor}
and Corollary \ref{risk_efficient_psi_cor_cor} to statistical models with many independent components. 

Let $X^{(n)}= (X_1^{(n)}, \dots, X_d^{(n)})$ be an observation with values in the space $S^{(n)}:=S_1^{(n)}\times \dots \times S_d^{(n)},$
where $(S_j^{(n)}, {\mathcal A}_j^{(n)}), j=1,\dots, d$ are measurable spaces and $S^{(n)}$ is equipped with the 
product $\sigma$-algebra ${\mathcal A}^{(n)}:= {\mathcal A}_1^{(n)}\times \dots \times {\mathcal A}_d^{(n)}.$
We will assume that the components $X_1^{(n)}, \dots, X_d^{(n)}$ of $X^{(n)}$ are independent r.v. and 
$X_j^{(n)}\sim P_{\theta_j}^{(n)}$ with parameter $\theta_j$ taking values in a Banach space $E_j,$ $j=1,\dots, d.$ 
Let $E:=E_1\times \dots \times E_d$ be equipped with a standard structure of linear space (the direct sum of linear spaces $E_1,\dots, E_d$)
and with the norm 
$
\|x\| = \Bigl(\sum_{j=1}^d \|x_j\|^2\Bigr)^{1/2},
$
$x=(x_1,\dots, x_d)\in E.$ Then, clearly, $X^{(n)}\sim P_{\theta}^{(n)}, \theta\in E,$ where 
$P_{\theta}^{(n)}:= P_{\theta_1}^{(n)}\times \dots \times P_{\theta_d}^{(n)},$ $\theta=(\theta_1,\dots, \theta_d)\in E.$
In the problems we have in mind, $\{P_{\theta_j}^{(n)}:\theta_j \in E_j\}, j=1,\dots, d$ are low dimensional models and the complexity 
of combined model $\{P_{\theta}^{(n)}: \theta=(\theta_1,\dots, \theta_d)\in E\}$ depends only on the number $d$ of independent components. 

Let $\hat \theta_j= \hat \theta_j(X_j^{(n)})$ be estimators of parameters $\theta_j, j=1,\dots, d$ and let $\hat \theta := (\hat \theta_1,\dots, \hat \theta_d)$ be the estimator of $\theta.$
Assume 
that $\sqrt{n}(\hat \theta_j-\theta_j)$ could be approximated in distribution by a centered Gaussian r.v. $\xi_j(\theta_j)$
with values in $E_j$ and with covariance operator $\Sigma_j(\theta_j).$ Since $\hat \theta_j, j=1,\dots, d$ are independent r.v., we assume that $\xi_j, j=1,\dots, d$
are also independent and $\xi(\theta):=(\xi_1(\theta_1),\dots, \xi_d(\theta_d)), \theta=(\theta_1,\dots, \theta_d)\in E$ can be used to approximate 
$\sqrt{n}(\hat \theta-\theta)$ in distribution. The following formula holds for the covariance operator $\Sigma(\theta)$ of $\xi(\theta):$ 
 \begin{align}
\label{cov_Sigma_Sigma_j}
&
\nonumber
\langle\Sigma(\theta)u,v\rangle = \sum_{j=1}^d \langle \Sigma_j(\theta_j) u_j,v_j\rangle,
\\
&
u=(u_1,\dots ,u_d), v=(v_1,\dots, v_d)\in E^{\ast}=E_1^{\ast}\times \dots\times E_d^{\ast}.
\end{align}
Moreover, we will view $\theta_j\mapsto \xi_j(\theta_j)$ 
as a stochastic process and use the following ``complexity"
characteristic of $\xi:$
\begin{align*}
{\frak q}_{\xi}(s):= 
\begin{cases}
\frac{1}{d}\sum_{i=1}^d {\mathbb E}\|\xi_i\|_{C^s(E_i)}^2& {\rm for}\ s\in (0,1]\\
\frac{1}{d}\sum_{i=1}^d {\mathbb E}\|\xi_i\|_{C^1(E_i)}^2 + \frac{\log(2d)}{d}\max_{1\leq i\leq d} {\mathbb E}  \|\xi_i\|_{C^{1,s}(E_i)}^2 & {\rm for}\ s>1. 
\end{cases}
\end{align*}

Based on estimator $\hat \theta := (\hat \theta_1,\dots, \hat \theta_d),$ define operators ${\mathcal T}, {\mathcal B}$
and functions $f_k.$ Note that 
\begin{align*}
\sigma_f^2(\theta)=\langle \Sigma(\theta) f'(\theta), f'(\theta)\rangle = \sum_{j=1}^d \langle \Sigma_j(\theta_j) f'_{\theta_j}(\theta),f'_{\theta_j}(\theta)\rangle,
\end{align*}
where $f'_{\theta_j}(\theta)\in E_j^{\ast}$ denotes the partial Fr\'echet derivative of $f(\theta)=f(\theta_1,\dots, \theta_d)$ w.r.t. $\theta_j.$

The following result holds.

\begin{corollary}
\label{indep_hat_theta_j}
Suppose that 
$
{\frak q}_{\xi}(s)\lesssim 1
$
and, for some sufficiently small constant $c_1>0,$ $d\leq c_1 n.$ 
Let $s=k+1+\rho$ with $k\geq 1$ and $\rho\in (0,1].$ Then, for all $s'\in [1,s],$ 
\begin{align}
\label{ind_comp_AAA}
	&
	\nonumber
	\sup_{\|f\|_{C^s(E)}\leq 1}\sup_{\theta \in E}\Delta_{s'}(\sqrt{n}(f_{k}(\hat \theta)-f(\theta)), \sigma_{f}(\theta)Z)
	\\
	&
	\nonumber
	\lesssim_{s} 
	\biggl[
	\sqrt{n}{\frak q}_{\xi}^{s/2}(s-1) \biggl(\sqrt{\frac{d}{n}}\biggr)^s 
	+
	\max_{1\leq j\leq d}\|\Sigma_j\|_{L_{\infty}(E_j)}^{1/2}{\frak q}_{\xi}^{1/2}(s-1)\sqrt{\frac{d}{n}}
	\\
	&
	\ \ \ \ \ \ \ +\Bigl(\sum_{j=1}^d W_{2,E_j}^2(\sqrt{n}(\hat \theta_j -\theta_j), \xi_j(\theta_j))\Bigr)^{1/2}
	\biggr]
\end{align}	
and 
\begin{align}
\label{ind_comp_BBB}
&	
\nonumber
	\sup_{\|f\|_{C^s(E)}\leq 1}\sup_{\theta\in E}
	\Bigl|\|f_k(\hat \theta)-f(\theta)\|_{L_2({\mathbb P}_{\theta})}-n^{-1/2}\sigma_f(\theta)\Bigr|
	\\
	&
	\nonumber
	\lesssim_{s} 
	{\frak q}_{\xi}^{s/2}(s-1)\ \biggl(\sqrt{\frac{d}{n}}\biggr)^s 
	+
	\frac{\max_{1\leq j\leq d}\|\Sigma_j\|_{L_{\infty}(E_j)}^{1/2}}{n^{1/2}}{\frak q}_{\xi}^{1/2}(s-1) \sqrt{\frac{d}{n}}
	\\
	&
	\ \ \ \ \ +\frac{1}{\sqrt{n}}\Bigl(\sum_{j=1}^d W_{2,E_j}^2(\sqrt{n}(\hat \theta_j -\theta_j), \xi_j(\theta_j))\Bigr)^{1/2}. 
\end{align}
In particular, bound \eqref{ind_comp_BBB} implies that 
\begin{align}
\label{ind_comp_CCC}
\nonumber
\sup_{\|f\|_{C^s(E)}\leq 1}\sup_{\theta\in E}\|f_k(\hat \theta)-f(\theta)\|_{L_2({\mathbb P}_{\theta})}
&
\lesssim_s \frac{\max_{1\leq j\leq d}\|\Sigma_j\|_{L_{\infty}(E_j)}^{1/2}}{n^{1/2}} + 
{\frak q}_{\xi}^{s/2}(s-1)\ \biggl(\sqrt{\frac{d}{n}}\biggr)^s 
\\
&
+ \frac{1}{\sqrt{n}}\Bigl(\sum_{j=1}^d W_{2,E_j}^2(\sqrt{n}(\hat \theta_j -\theta_j), \xi_j(\theta_j))\Bigr)^{1/2}.
\end{align}
\end{corollary}

\begin{remark}
\normalfont
The bounds also hold for $k=0.$ In this case, the terms 
$$
\max_{1\leq j\leq d}\|\Sigma_j\|_{L_{\infty}(E_j)}^{1/2}{\frak q}_{\xi}^{1/2}(s-1)\sqrt{\frac{d}{n}} 
$$
of \eqref{ind_comp_AAA} and 
$$
\frac{\max_{1\leq j\leq d}\|\Sigma_j\|_{L_{\infty}(E_j)}^{1/2}}{n^{1/2}}{\frak q}_{\xi}^{1/2}(s-1) \sqrt{\frac{d}{n}}
$$
of \eqref{ind_comp_BBB} could be dropped.
\end{remark}

\begin{remark}
\normalfont
Suppose ${\frak q}_{\xi}(s)\lesssim 1.$ In particular, for $s>1,$ this holds if $\max_{1\leq i\leq d}{\mathbb E}\|\xi_i\|_{C^1(E_i)}^2\lesssim 1$
and $\max_{1\leq i\leq d} {\mathbb E}  \|\xi_i\|_{C^{1,s}(E_i)}^2\lesssim \frac{d}{\log d}.$ Suppose, in addition,  that  
$$\max_{1\leq j\leq d}\|\Sigma_j\|_{L_{\infty}(E_j)}^{1/2}\lesssim 1.$$ 
If the models $\{P_{\theta_j}: \theta_j\in E_j\}$ are low-dimensional and sufficiently 
regular, the assumptions above hold for maximum likelihood estimators $\hat \theta_j$
of $\theta_j, j=1,\dots, d.$ In fact, in this case, we would have $\max_{1\leq i\leq d}{\mathbb E}\|\xi_i\|_{C^s(E_i)}^2\lesssim 1$
(if the Fisher information matrices $I_j(\theta_j)$ of low dimensional models are sufficiently smooth). 
If, in addition,  the following 
normal approximation bound holds for the estimators $\hat \theta_j$ of the components $\theta_j$
\begin{align}
\label{W_2_rate_estimator}
\max_{1\leq j\leq d} W_{2,E_j}(\sqrt{n}(\hat \theta_j -\theta_j), \xi_j(\theta_j)) \lesssim n^{-1/2},
\end{align}
then we have 
\begin{align*}
\Bigl(\sum_{j=1}^d W_{2,E_j}^2(\sqrt{n}(\hat \theta_j -\theta_j), \xi_j(\theta_j))\Bigr)^{1/2}\lesssim \sqrt{\frac{d}{n}},
\end{align*}
which guarantees normal approximation of $\sqrt{n}(\hat \theta-\theta)$ by $\xi(\theta)$ for $d=o(n).$
In this case, \eqref{ind_comp_AAA} implies 
\begin{align*}
\nonumber
	\sup_{\|f\|_{C^s(E)}\leq 1}\sup_{\theta \in E}\Delta_{s'}(\sqrt{n}(f_{k}(\hat \theta)-f(\theta)), \sigma_{f}(\theta)Z)
	\lesssim_{s} \sqrt{n}\biggr(\sqrt{\frac{d}{n}}\biggr)^s +\sqrt{\frac{d}{n}},
\end{align*}
\eqref{ind_comp_BBB} implies 
\begin{align*}
	\sup_{\|f\|_{C^s(E)}\leq 1}\sup_{\theta\in E}
	\Bigl|\|f_k(\hat \theta)-f(\theta)\|_{L_2({\mathbb P}_{\theta})}-n^{-1/2}\sigma_f(\theta)\Bigr|
\lesssim_s \biggl(\sqrt{\frac{d}{n}}\biggr)^s + \frac{1}{\sqrt{n}} \sqrt{\frac{d}{n}}
\end{align*}
and \eqref{ind_comp_CCC} implies that 
\begin{align*}
\sup_{\|f\|_{C^s(E)}\leq 1}\sup_{\theta\in E}\|f_k(\hat \theta)-f(\theta)\|_{L_2({\mathbb P}_{\theta})}
\lesssim_s \frac{1}{\sqrt{n}}+ \biggl(\sqrt{\frac{d}{n}}\biggr)^s.
\end{align*}
If $d\leq n^{\alpha}$ for some $\alpha\in (0,1)$ and $s>\frac{1}{1-\alpha},$ the above bounds imply the asymptotic normality of estimator 
$f_k(\hat \theta)$ with $\sqrt{n}$ rate as well as the convergence of $\sqrt{n}\|f_k(\hat \theta)-f(\theta)\|_{L_2({\mathbb P}_{\theta})}$ to $\sigma_f(\theta).$
Note also that, if $d\leq n^{\alpha}$ for some $\alpha \in (0,1),$ then it is sufficient for asymptotic normality of $f_k(\hat \theta)$ and for convergence of its normalized 
risk to $\sigma_f(\theta)$ to have normal approximation error in  \eqref{W_2_rate_estimator} of the order $o(n^{-\alpha/2})$ instead of $n^{-1/2}.$
\end{remark}

\begin{remark}
\normalfont
In the low-dimensional case, bounds of the order $n^{-1/2}$ on the accuracy of normal approximation of MLE and more general $M$-estimators in Kolmogorov's distance (Berry-Esseen type bounds)
could be found in \cite{Pfanzagl, Bentkus, Pinelis} and in Wasserstein's $W_1$-distance in \cite{Anastasiou}.
We are not aware of similar published results for Wassertein's $W_2$-distance. However, it is possible to adapt the approach of these papers 
in combination with known bounds on the accuracy of normal approximation in CLT (see, e.g.,  \cite{Rio, Eldan}) to obtain bounds for the $W_2$-distance 
suitable in the framework of Corollary \ref{indep_hat_theta_j}.
\end{remark}

We now turn to some other examples of statistical models with independent components.  
Assume that 
\begin{align*}
X= \theta + A(\theta)\sum_{j=1}^d \eta_j x_j, \theta \in E,
\end{align*}
where $A(\theta):E\mapsto E$ is a bounded linear operator,
$\{\eta_j\}$ are independent r.v. with\footnote{One can even assume that the distribution 
of $\eta_j$ depends on $\theta.$} 
\begin{align*}
{\mathbb E}\eta_j=0,\ {\mathbb E}\eta_j^2=1, j=1,\dots, d
\end{align*}
and $x_j\in E, j=1,\dots, d.$ Define 
\begin{align*}
\xi(\theta):= A(\theta)\sum_{j=1}^d \zeta_j x_j, \theta \in E,
\end{align*}  
where $\{\zeta_j\}$ are i.i.d. standard normal r.v. 
Note that 
\begin{align*}
\Sigma(\theta)= \sum_{j=1}^d A(\theta)x_j \otimes A(\theta)x_j
\end{align*}
is the covariance operator of both $X$ and $\xi(\theta).$

Denote 
\begin{align*}
{\frak d}_d:= {\frak d}_d(x_1,\dots, x_d):={\mathbb E}\Bigl\|\sum_{j=1}^d \zeta_j x_j\Bigr\|^2.
\end{align*}

Given a sample $X_1,\dots, X_n$ of i.i.d. copies of $X,$ let 
\begin{align*}
\hat \theta := \bar X = \frac{X_1+\dots+X_n}{n}. 
\end{align*}

\begin{proposition}
\label{prop_ind_comp_00}	
Let $s=k+1+\rho$ with $k\geq 1$ and $\rho \in (0,1].$
Let 
\begin{align*}
\beta_{4}:= \max_{1\leq j\leq d}{\mathbb E}|\eta_j|^{4}<\infty. 
\end{align*}
Suppose also that $\|A\|_{C^s(E)}\lesssim 1$ and ${\frak d}_d\leq c_1 n$ 
for a small enough constant $c_1>0.$
Then the following bounds hold:
\begin{align*}
&
\sup_{\|f\|_{C^s(E)}\leq 1}\sup_{\theta \in E}\Delta_{1}(\sqrt{n}(f_{k}(\hat \theta)-f(\theta)), \sigma_{f}(\theta)Z)
\\
&
\lesssim_{s} 
\biggl[
\sqrt{n}\biggl(\|A\|_{C^s(E)}\sqrt{\frac{{\frak d}_{d}}{n}}\biggr)^s 
+
\|\Sigma\|_{L_{\infty}(E)}^{1/2}\|A\|_{C^s(E)}\sqrt{\frac{{\frak d}_{d}}{n}}
+\beta_4^{1/2} \|\Sigma\|_{L_{\infty}(E)}^{1/2}\sqrt{\frac{d}{n}}
\biggr]
\end{align*}
and 
\begin{align*}
&
\sup_{\|f\|_{C^s(E)}\leq 1}\sup_{\theta \in E}
\Bigl|\|f_k(\hat \theta)-f(\theta)\|_{L_2({\mathbb P}_{\theta})}
-n^{-1/2}\sigma_f(\theta)\Bigr|
\\
&
\lesssim_{s} 
\biggl[
\biggl(\|A\|_{C^s(E)}\sqrt{\frac{{\frak d}_{d}}{n}}\biggr)^s 
+
\frac{\|\Sigma\|_{L_{\infty}(E)}^{1/2}}{\sqrt{n}}\|A\|_{C^s(E)}\sqrt{\frac{{\frak d}_{d}}{n}}
+\frac{\beta_4^{1/2} \|\Sigma\|_{L_{\infty}(E)}^{1/2}}{\sqrt{n}}\sqrt{\frac{d}{n}}
\biggr].
\end{align*}	
\end{proposition}	

\begin{remark}
\normalfont
For $k=0,$ the bounds hold without the terms $\|\Sigma\|_{L_{\infty}(E)}^{1/2}\|A\|_{C^s(E)}\sqrt{\frac{{\frak d}_{d}}{n}}$
and $\frac{\|\Sigma\|_{L_{\infty}(E)}^{1/2}}{\sqrt{n}}\|A\|_{C^s(E)}\sqrt{\frac{{\frak d}_{d}}{n}}.$
\end{remark}

\begin{remark}
\normalfont	
Under a stronger assumption that r.v. $\eta_j$ are sub-exponential and
\begin{align*}
\max_{1\leq j\leq d}\|\eta_j\|_{L_{\psi_1}({\mathbb P})}\lesssim 1,
\end{align*}	
it is possible to prove a bound on the $L_{\psi}$-risk $\|f_k(\hat \theta)-f(\theta)\|_{L_\psi({\mathbb P}_{\theta})}$ of estimator $f_k(\hat \theta)$ for any loss $\psi \in \Psi$ such that the function $\psi (\sqrt{u}), u\geq 0$ is convex. Namely, in this case, the following bound holds:
\begin{align*}
	&
	\sup_{\|f\|_{C^s(E)}\leq 1}\sup_{\theta \in E}
	\Bigl|\|f_k(\hat \theta)-f(\theta)\|_{L_\psi({\mathbb P}_{\theta})}
	-n^{-1/2}\sigma_f(\theta)\|Z\|_{L_{\psi}({\mathbb P})}\Bigr|
	\\
	&
	\lesssim_{s} 
	\biggl[
	\biggl(\|A\|_{C^s(E)}\sqrt{\frac{{\frak d}_{d}}{n}}\biggr)^s 
	+
	\frac{\|\Sigma\|_{L_{\infty}(E)}^{1/2}}{\sqrt{n}}\|A\|_{C^s(E)}\sqrt{\frac{{\frak d}_{d}}{n}}
	+\frac{\beta_{\psi, \eta}\|\Sigma\|_{L_{\infty}(E)}^{1/2}}{\sqrt{n}}\sqrt{\frac{d}{n}}
	\biggr],
\end{align*}
where the constant $\beta_{\psi, \eta}$ depends only on $\max_{1\leq j\leq d}\|\eta_j\|_{L_{\psi_1}({\mathbb P})}.$ The proof relies on Theorem 2.1  in \cite{Rio}
providing rates of convergence in CLT in ${\mathbb R}$ in Wasserstein $W_{\psi_1}$-distance. 	
\end{remark}	

In the next proposition, it will be additionally required that $s\geq 3,$ but it might provide
a better bound in the cases when the norm of space $E$ is not Euclidean (say, the operator norm 
for matrices) and vectors $x_j$ have small norms (see the example below).

\begin{proposition}
\label{prop_ind_comp}
Let $s=k+1+\rho \geq 3,$ $k\geq 1,$ $\rho \in (0,1].$ 
Let 
\begin{align*}
\beta_{3}:= \max_{1\leq j\leq d}{\mathbb E}|\eta_j|^{3}<\infty 
\end{align*}
and 
\begin{align*}
U:= \max_{1\leq j\leq d}\|x_j\|.
\end{align*}
Suppose also that $\|A\|_{C^s(E)}\lesssim 1$ and ${\frak d}_d\leq c_1 n$ 
for a small enough constant $c_1>0.$
Then the following bounds hold:
\begin{align*}
&
\sup_{\|f\|_{C^s(E)}\leq 1}\sup_{\theta \in E}\Delta_{3}(\sqrt{n}(f_{k}(\hat \theta)-f(\theta)), \sigma_{f}(\theta)Z)
\\
&
\lesssim_{s} 
\biggl[
\sqrt{n}\biggl(\|A\|_{C^s(E)}\sqrt{\frac{{\frak d}_{d}}{n}}\biggr)^s 
+
\|\Sigma\|_{L_{\infty}(E)}^{1/2}\|A\|_{C^s(E)}\sqrt{\frac{{\frak d}_{d}}{n}}
+\frac{\|A\|_{L_{\infty}(E)}^{3}\beta_{3} U^{3} d}{\sqrt{n}}
\biggr]
\end{align*}
and 
\begin{align*}
&
\sup_{\|f\|_{C^s(E)}\leq 1}\sup_{\theta \in E}
\Bigl|\|f_k(\hat \theta)-f(\theta)\|_{L_2({\mathbb P}_{\theta})}
-n^{-1/2}\sigma_f(\theta)\Bigr|
\\
&
\lesssim_{s} 
\biggl[
\biggl(\|A\|_{C^s(E)}\sqrt{\frac{{\frak d}_{d}}{n}}\biggr)^s 
+
\frac{\|\Sigma\|_{L_{\infty}(E)}^{1/2}}{\sqrt{n}}\|A\|_{C^s(E)}\sqrt{\frac{{\frak d}_{d}}{n}}
\\
&
\ \ \ \ \ \ +
\frac{\|A\|_{L_{\infty}(E)}^{3/2}\beta_{3}^{1/2} U^{3/2}\sqrt{d}}{n} 
+\frac{\|A\|_{L_{\infty}(E)}^{3}\beta_{3} U^{3}d}{n^2}
\biggr].
\end{align*}	
\end{proposition}

As a more specific example, consider a matrix problem in which 
$E={\mathbb H}_m$ is the space of $m\times m$ Hermitian matrices 
equipped with the Hilbert--Schmidt inner product. Let $d=m=2^l$ and let $\{E_1, \dots, E_{m^2}\}$
be {\it the Pauli basis} (often used in quantum compressed sensing). It is defined as follows:
let
\begin{equation*}
 \sigma_0:=\left(\begin{array}{cc}1&0\\0&1 \end{array}\right),\quad \sigma_1:=\left(\begin{array}{cc}0&1\\1&0 \end{array}\right),
  \quad \sigma_2:=\left(\begin{array}{cc}0&i\\-i&0 \end{array}\right),\quad \sigma_3:=\left(\begin{array}{cc}1&0\\0&-1 \end{array}\right).
 \end{equation*}
 The matrices $\sigma_1, \sigma_2, \sigma_3$ (often denoted $\sigma_x, \sigma_y, \sigma_z$)
 are called the {\it Pauli matrices}. The matrices $W_i=\frac{1}{\sqrt{2}}\sigma_i,\ i=0,1,2, 3$
 form an orthonormal basis of the space ${\mathbb H}_2$ (the Pauli basis).  
 The Pauli basis of the space ${\mathbb H}_m,$ $m=2^l$ is defined by tensorizing the Pauli basis of ${\mathbb H}_2:$ it consists of $m^2=4^l$ tensor products $W_{i_1}\otimes\ldots\otimes W_{i_l}, (i_1,\ldots,i_l)
\in \left\{0,1,2,3\right\}^l.$

The goal is to estimate a target matrix $\theta\in {\mathbb H}_m$
(for instance, a density matrix of a quantum system) based on 
i.i.d. copies $X_1,\dots, X_n$ of r.v. $X,$ 
\begin{align*}
X= \theta + A(\theta)\sum_{j=1}^{m^2} \eta_j E_j,
\end{align*}
where $A(\theta):{\mathbb H}_m\mapsto {\mathbb H}_m$ is a linear operator. 
We can now equip ${\mathbb H}_m$ either with its Hilbert--Schmidt norm $\|\cdot\|_2,$ or with its 
operator norm (that will be denoted by $\|\cdot\|$).

In the first case, $\|E_j\|_2=1$ and $U=1.$ We also have,
${\frak d}_d =d=m^2.$ The following corollary of Proposition \ref{prop_ind_comp_00}
is immediate. 

\begin{corollary}
\label{cor_ind_comp_1}
Let $s=k+1+\rho,$ $k\geq 0,$ $\rho \in (0,1].$ 
Suppose $\|A\|_{C^s({\mathbb H}_m;\|\cdot\|_2)}\lessim 1$ and $\beta_{4}\lesssim 1.$
Suppose also that $d=m^2 \lesssim n^{\alpha}$ for some $\alpha\in (0,1)$
and $s>\frac{1}{1-\alpha}.$
Then, 
\begin{align*}
\sup_{\|f\|_{C^s({\mathbb H}_m;\|\cdot\|_2)}\leq 1}
\sup_{\theta\in {\mathbb H}_m}\Bigl|n{\mathbb E}(f_k(\hat \theta)-f(\theta))^2-\sigma_f^2(\theta) \Bigr|\to 0
\end{align*}
and, for all $\sigma_0>0,$ 
\begin{align*}
\sup_{\|f\|_{C^s({\mathbb H}_m;\|\cdot\|_2)}\leq 1}
\sup_{\theta\in {\mathbb H}_m, \sigma_f(\theta)\geq \sigma_0} d_K\Bigl(\frac{\sqrt{n}(f_{k}(\hat \theta)-f(\theta))}{\sigma_{f}(\theta)}, Z\Bigr)\to 0
\end{align*}
as $n\to\infty.$
\end{corollary}

In the second case, we have $\|E_j\| \leq m^{-1/2}$ and we can take $U=m^{-1/2}.$
Also, by standard bounds for Gaussian matrices, it is easy to see that 
${\frak d}_d \lesssim m= \sqrt{d}.$ Proposition \ref{prop_ind_comp} implies the following corollary.

\begin{corollary}
\label{cor_ind_comp_2}
Let $s=k+1+\rho \geq 3,$ $k\geq 1,$ $\rho \in (0,1].$ 
Suppose $\|A\|_{C^s({\mathbb H}_m;\|\cdot\|)}\lessim 1$ and $\beta_{3}\lesssim 1.$
Suppose also that $m \lesssim n^{\alpha}$ for some $\alpha\in (0,1)$
and $s>\frac{1}{1-\alpha}.$
Then 
\begin{align*}
\sup_{\|f\|_{C^s({\mathbb H}_m;\|\cdot\|)}\leq 1}
\sup_{\theta\in {\mathbb H}_m}\Bigl|n{\mathbb E}(f_k(\hat \theta)-f(\theta))^2-\sigma_f^2(\theta)\Bigr|\to 0
\end{align*}
and, for all $\sigma_0>0,$ 
\begin{align*}
\sup_{\|f\|_{C^s({\mathbb H}_m;\|\cdot\|)}\leq 1}
\sup_{\theta\in {\mathbb H}_m, \sigma_f(\theta)\geq \sigma_0} d_K\Bigl(\frac{\sqrt{n}(f_{k}(\hat \theta)-f(\theta))}{\sigma_{f}(\theta)}, Z\Bigr)
\to 0
\end{align*}
as $n\to\infty.$
\end{corollary}

Note that the $C^s$-norms depend on the underlying norm of $E$ and, in a high-dimensional 
setting, norms  $\|\cdot\|_{C^s({\mathbb H}_m;\|\cdot\|_2)}$ and 
$\|\cdot\|_{C^s({\mathbb H}_m;\|\cdot\|)}$ could be very different.

\subsubsection{Poincar\'e constants and log-concave models}
\label{Poincare_constant}

Let $E={\mathbb R}^d$ be equipped with the Euclidean norm and let $X\sim P_{\theta}, \theta \in T, T\subset {\mathbb R}^d$ be a statistical model with the sample space ${\mathbb R}^d.$
As before, we assume that $T$ is an open subset. Also assume that ${\mathbb E}_{\theta}\|X\|^2<\infty, \theta \in T$
and let 
\begin{align*}
&
\Psi(\theta):= {\mathbb E}_{\theta} X, 
\\
&
\Sigma(\theta):= {\mathbb E}_{\theta}(X-\Psi(\theta))\otimes 
(X-\Psi(\theta)),
\theta \in \Theta.
\end{align*} 
Moreover, let us assume that 
$\Psi : T\mapsto \Psi(T)$ is a homeomorphism between open sets $T$ and $\Psi(T).$ This 
assumption would allow us to re-parametrize our model by setting 
$\vartheta:= \Psi(\theta)={\mathbb E}_{\theta} X, \theta \in T$ and using parameter $\vartheta\in \Psi(T)$
instead of $\theta.$ For this new parameter, we simply have ${\mathbb E}_{\vartheta} X=\vartheta, \vartheta \in \Psi(T).$ 

Given i.i.d. observations $X_1,\dots, X_n$ of $X,$
let 
\begin{align*}
\bar X := \frac{X_1+\dots + X_n}{n},
\end{align*}
\begin{align*}
\hat \theta = \hat \theta(X_1,\dots, X_n) = 
\begin{cases}
\Psi^{-1}(\bar X) & {\rm if}\ \bar X\in \Psi(T)\\ 
\theta_0& {\rm if}\ \bar X\not\in \Psi(T),
\end{cases}
\end{align*}
where $\theta_0\in T$ is an arbitrary point, 
and 
\begin{align*}
\hat \vartheta = \hat \vartheta(X_1,\dots, X_n) = 
\begin{cases}
\bar X & {\rm if}\ \bar X\in \Psi(T)\\ 
\Psi(\theta_0)& {\rm if}\ \bar X\not\in \Psi(T)
\end{cases}
= \Psi (\hat \theta).
\end{align*}
It is easy to check that ${\mathcal T}(f\circ \Psi^{-1})=({\mathcal T}f)\circ \Psi^{-1},$
${\mathcal B}(f\circ \Psi^{-1})= ({\mathcal B}f)\circ \Psi^{-1}$ and $(f\circ \Psi^{-1})_k= f_k\circ \Psi^{-1},$
where, with a little abuse of notation, we keep the same letters ${\mathcal T}$ and ${\mathcal B}$
to denote the operators based on estimator $\hat \vartheta.$
This allows us to reduce the problem of estimation of functional $f(\theta)$
to the problem of estimation of functional $(f\circ \Psi^{-1})(\vartheta)$ under its proper 
smoothness 
and to use for this purpose the estimator 
$$
f_k (\hat \theta)= (f_k \circ \Psi^{-1})(\hat \vartheta)= (f\circ \Psi^{-1})_k(\hat \vartheta).
$$
Of course, one can expect that $\sqrt{n}(\hat\vartheta-\vartheta)$ could be approximated 
by Gaussian random variable $\xi(\theta)$ with mean zero and covariance operator $\Sigma(\theta)$
(for $\vartheta=\Psi(\theta)$).

We will assume that $P_{\theta}$ satisfies Poincar\'e inequality,
so, $C_P(P_{\theta})<\infty.$
Let
\begin{align*}
\sigma^2_{f\circ \Psi^{-1}}(\vartheta)= \langle \Sigma(\Psi^{-1}(\vartheta)) (f\circ \Psi^{-1})'(\vartheta), 
(f\circ \Psi^{-1})'(\vartheta)\rangle.
\end{align*}

\begin{proposition}
\label{Cor_Main_Th_DD}
Let $d=d_n$ and $\Theta = \Theta_n \subset {\mathbb R}^d$ with ${\rm Diam}(\Theta)\lesssim n^{A}$
for some $A>0.$ 
Let $\delta>0$
and let $s=k+1+\rho$ with $k\geq 0$ and $\rho\in (0,1].$  
Suppose that $\Theta_{\delta}\subset T$ and 
\begin{align}
\label{cond_on_Sigma}
\|\Sigma\|_{C^s(\Theta_{\delta})}\lesssim 1
\ {\rm and}\  
\|\Sigma^{-1}\|_{L_{\infty}(\Theta_{\delta})}\lesssim 1.
\end{align}
Suppose that, for some $\alpha\in (0,1),$
$ 
d\lesssim n^{\alpha}
$
and assume that $s>\frac{1}{1-\alpha}.$ 
Finally, suppose that 
\begin{align}
\label{cond_on_C_P}
\sup_{\theta \in \Theta_{\delta}}C_P(P_{\theta}) = o(n^{1-\alpha})
\ {\rm as}\ n\to\infty.
\end{align}
Let $\theta_0$ in the definition of $\hat \theta$ be a point from $\Theta.$
Then
\begin{align}
\label{asymp_variance_XYZ}
\sup_{\|f\circ \Psi^{-1}\|_{C^s((\Psi(\Theta))_{\delta})}\leq 1}
\sup_{\theta\in \Theta} 
\Bigl|n{\mathbb E}_{\theta}(f_k(\hat \theta)-f(\theta))^2- 
\sigma_{f\circ \Psi^{-1}}^2(\Psi(\theta))\Bigr|\to 0
\end{align}
and, for all $\sigma_0>0,$ 
\begin{align}
\label{asymp_norm_XYZ}
\sup_{\|f\circ \Psi^{-1}\|_{C^s((\Psi(\Theta))_{\delta})}\leq 1}
\sup_{\theta\in \Theta, \sigma_{f\circ \Psi^{-1}}(\Psi(\theta))\geq \sigma_0} 
d_K\Bigl(\frac{\sqrt{n}(f_{k}(\hat \theta)-f(\theta))}{\sigma_{f\circ \Psi^{-1}}(\Psi(\theta))}, Z\Bigr)\to 0
\end{align}
as $n\to\infty.$
\end{proposition}

\begin{remark}
\normalfont
Suppose that, for small $\delta>0,$ $\Psi $ is a $C^s$-diffeomorphism between $\Theta_{\delta}$ and 
$\Psi(\Theta_{\delta})$ (with bounded $C^s$-norms of $\Psi$ and $\Psi^{-1}$). Then, for a small enough
$\delta>0,$ there exists $\delta'>0$ such that $\Psi^{-1}((\Psi(\Theta))_{\delta'})\subset \Theta_{\delta}$
and the first supremum in \eqref{asymp_norm_XYZ} and \eqref{asymp_variance_XYZ}
could be taken over the set $\|f\|_{C^s(\Theta_{\delta})}\lesssim 1.$
\end{remark}

\begin{remark}
\normalfont
The properties of Poincar\'e constants discussed in Remark \ref{properties_P_const}
provide a way to check condition \eqref{cond_on_C_P}.
In particular, the claim of the corollary obviously holds in the Gaussian case. 
Moreover, if $P_{\theta}$ is absolutely continuous with respect 
to a measure $\nu_{\theta}$ for which $C_P(\nu_{\theta})$ is controlled by a numerical constant 
(for instance, a Gaussian measure)
and the densities $\frac{dP_{\theta}}{d\nu_{\theta}}$ 
are bounded from above by a constant $A>0$ and bounded from below by a constant $a>0,$ then 
$
C_P(P_{\theta})\lesssim 1
$
and condition \eqref{cond_on_C_P}  holds. Thus, the claim of Proposition \ref{Cor_Main_Th_DD}
also holds (under the rest of its conditions).
\end{remark}

\begin{remark}
\normalfont
Suppose that measures $P_{\theta}, \theta \in \Theta_{\delta}$ are log-concave. It follows from a recent result 
of \cite{Yuansi_Chen} (see Remark \ref{properties_P_const}) that $\sup_{\theta \in \Theta_{\delta}}C_P(P_{\theta})\lesssim_{\nu} d^{\nu}$
for an arbitrary $\nu>0.$ Thus, in this case, condition \eqref{cond_on_C_P} holds for all $\alpha\in (0,1)$ (and $d\lesssim n^{\alpha}$)
and so are the claims of  Proposition \ref{Cor_Main_Th_DD}.
\end{remark}

In the simplest case, $X=\theta+ \eta,$ where $\eta$ is a mean zero noise sampled 
from some distribution $\mu_{\theta}$ in ${\mathbb R}^d,$ depending on the parameter 
$\theta.$ 
In this case, $\vartheta=\Psi(\theta)=\theta$ and it is easy to state a simplified version of Proposition \ref{Cor_Main_Th_DD}. 
If the distribution $\mu_{\theta}=\mu$ of the noise does not depend on $\theta$ and $C_P(\mu)<\infty,$ similar problems 
were studied in a recent paper \cite{Koltchinskii_Zhilova_2020}. The approach was based on a more direct analysis of estimator 
$f_k(\hat \theta)$ in the case of such Poincar\'e random shift models without using normal approximation. However, this approach could not 
be extended to more general models with distribution $\mu_{\theta}$ of the noise depending on $\theta$ since, in this case, the construction 
of random homotopies between estimator $\bar X$ and parameter $\vartheta$ leads to rather challenging coupling problems (see also the discussion 
in Section \ref{sec:bias_reduce}).

A slightly more complicated example, is an exponential 
family\footnote{All the facts about exponential families used below could be found, for instance, in \cite{Chencov}}
\begin{align}
\label{exponential_family_def}
P_{\theta}(dx) = \frac{1}{Z(\theta)} \exp\{\langle\theta, x\rangle\}h(x)dx, \theta\in T,
\end{align}
where $h: {\mathbb R}^d\mapsto [0,+\infty)$ is a Borel measurable function and 
\begin{align*}
Z(\theta):= \int_{{\mathbb R}^d}\exp\{\langle\theta, x\rangle\}h(x)dx<\infty, \theta\in T.
\end{align*}
Note that the set $\{\theta\in {\mathbb R}^d: Z(\theta)<+\infty\}$ is convex and $T$ is a subset of this set. 
Assume that $T$ is convex, too. It is well known that $T\ni\theta \mapsto \log Z(\theta)$ is a 
strictly convex smooth function and 
\begin{align*}
\vartheta=\Psi(\theta)={\mathbb E}_{\theta} X= (\nabla \log Z)(\theta), \theta\in T.
\end{align*}

Moreover, $\Psi = \nabla \log Z$ is a strictly monotone vector field on $T$ (as a gradient 
of a strictly convex smooth function) and, therefore, it is a one-to-one mapping from $T$
onto $\Psi(T)$ (as before, it is also assumed to be a homeomorphism). Following the terminology 
of \cite{Chencov} (which is not quite standard), $\theta$ is called the {\it canonical parameter} of the exponential family 
and $\vartheta$ is called its {\it natural parameter}.

Note also that 
\begin{align*}
(\log Z)^{\prime \prime}(\theta) = \Psi^{\prime}(\theta) = \Sigma(\theta)
\end{align*}
is the covariance of $X.$ It is also the Fisher information matrix $I(\theta)$ for this model with respect 
to the canonical parameter $\theta$ and the inverse Fisher information matrix ${\mathcal I}^{-1}(\vartheta)$
with 
respect to the natural parameter $\vartheta=\Psi(\theta).$ 
Let now $X_1,\dots, X_n$ be i.i.d. $\sim P_{\theta}, \theta \in T.$ If $\bar X\in \Psi(T),$ then 
$\hat \theta=\Psi^{-1}(\bar X)$ is the unique maximum likelihood estimator for this exponential model.

We will call exponential family \eqref{exponential_family_def} log-concave iff the function $h$ is log-concave. 
Clearly, in this case the distributions $P_{\theta}, \theta \in T$ are log-concave. Proposition \eqref{Cor_Main_Th_DD}
and the above discussion yield the following corollary.

\begin{corollary}
\label{cor_log_conv_exp}
Let $d=d_n$ and let $P_{\theta}, \theta \in T=T_n\subset {\mathbb R}^d$ be a log-concave exponential family. 
Let $\Theta = \Theta_n \subset T$ with ${\rm Diam}(\Theta)\lesssim n^{A}$ for some $A>0.$ Let $\delta>0$
and let $s=k+1+\rho$ with $k\geq 0$ and $\rho\in (0,1].$  
Suppose that $\Theta_{\delta}\subset T$ and conditions \eqref{cond_on_Sigma}
hold. Suppose that, for some $\alpha\in (0,1),$
$ 
d\lesssim n^{\alpha}
$
and assume that $s>\frac{1}{1-\alpha}.$ 
Let $\theta_0$ in the definition of $\hat \theta$ be a point from $\Theta.$ Then asymptotic relationships \eqref{asymp_variance_XYZ} and 
\eqref{asymp_norm_XYZ} hold for estimator $f_k(\hat \theta)$ of $f(\theta).$ 
\end{corollary}

\begin{remark}
\normalfont
Note that, in the case of exponential model, the limit variance $\sigma_{f\circ \Psi^{-1}}(\Psi(\theta))$ in Proposition \ref{Cor_Main_Th_DD}
is equal to $\langle {\mathcal I}^{-1}(\vartheta)(f\circ \Psi^{-1})'(\vartheta),(f\circ \Psi^{-1})'(\vartheta)\rangle$
with $\vartheta=\Psi(\theta).$ It is possible to prove local minimax lower bounds showing  
optimality of this variance and the asymptotic efficiency of estimator of $f_k(\hat \theta)$
(for instance, using Van Trees inequality \cite{GillLevit}, see \cite{Koltchinskii_2017}, \cite{Koltchinskii_Zhilova}
for similar results).
\end{remark}

\begin{remark}
\normalfont
The result of Corollary \ref{cor_log_conv_exp} also holds under more general assumption that function $h$ in the definition of exponential model 
\eqref{exponential_family_def} satisfies the condition $c^{-1} g(x)\leq h(x)\leq c g(x), x\in {\mathbb R}^d$ for a non-negative log-concave function $g$
and for a constant $c\geq 1.$
\end{remark}

\begin{remark}
\normalfont
It was shown in \cite{Portnoy_1}, Theorem 3.1 that, under some moment assumptions on $d$-dimensional exponential families with MLE $\hat \theta,$
$\hat \theta-\theta$ could be approximated by a sample mean with accuracy $O_{{\mathbb P}}(\frac{d}{n}).$ Together with a high-dimensional 
CLT proved in \cite{Portnoy}, this implies that normal approximation of $\sqrt{n}(\hat \theta-\theta)$ holds if $d=o(\sqrt{n}).$
It was also shown in \cite{Portnoy_1}, Proposition 3.1 that, if $d$ is larger than $\sqrt{n},$ the normal approximation of 
$\sqrt{n}(\hat \theta-\theta)$ could fail even for linear functionals. Thus, additional conditions on exponential family 
(for instance, shape constraints such as log-concavity) are needed to justify normal approximation for MLE when $d\geq \sqrt{n}$
(which is an interesting regime for functional estimation requiring the bias reduction).
\end{remark}

\subsection{Outline of the proofs: bootstrap chains and random homotopies}
\label{sec:boostrap_chain}

Let $\hat \theta^{(k)}, k\geq 0$ be the Markov chain in the space $T$ with transition 
probability kernel $P(\theta, A), \theta \in T, A\subset T,$ defined by \eqref{define_Markov_kernel},
and with $\hat \theta^{(0)}=\theta.$ For this chain, $\hat \theta^{(1)}$ has the same distribution 
as $\hat \theta;$ conditionally on $\hat \theta^{(1)},$ $\hat \theta^{(2)}$ is sampled from the distribution $P(\hat \theta^{(1)}; \cdot);$ conditionally on $\hat \theta^{(2)},$ $\hat \theta^{(3)}$ is sampled from the distribution $P(\hat \theta^{(2)}, \cdot),$ etc. Thus, the Markov chain $\hat \theta^{(k)}, k\geq 0$ is constructed by an iterative 
application of parametric bootstrap to the estimator $\hat \theta$ and it was called in \cite{Koltchinskii_2017}
{\it the bootstrap chain} of this estimator. 
Bootstrap chains are involved in representations of functionals
${\mathcal B}^k f, k\geq 1$ needed to control the bias of estimator $f_k(\hat \theta).$ Namely (see \cite{Koltchinskii_2017, Koltchinskii_2018, Koltchinskii_Zhilova, Koltchinskii_Zhilova_19}), 
\begin{align*}
({\mathcal B}^k f)(\theta)= {\mathbb E}_{\theta}\sum_{j=0}^{k} (-1)^{k-j} {k\choose j} f(\hat \theta^{(j)}),
\end{align*}
which is the expectation of the $k$-th order difference of function $f$ along the sample path of the 
bootstrap chain. It is well known that for a $k$ times continuously differentiable function $f$ in the real 
line, its $k$-th order difference 
\begin{align*}
\Delta_h^k f(x) = \sum_{j=0}^k (-1)^{k-j} {k\choose j} f(x+jh) = f^{(k)}(x)h^k + o(h^k)\ {\rm as}\ h\to 0.  
\end{align*}
If, for a small $\delta>0,$ 
$
\sup_{\theta \in T}{\mathbb P}_{\theta}\{\|\hat \theta-\theta\|\geq \delta\}
$
is also small, we would have that $\|\hat \theta^{(j+1)}-\hat \theta^{(j)}\|<\delta$
with a high probability. In this case, one could expect that, for a $k$ times continuously 
differentiable function $f:E\mapsto {\mathbb R},$ $({\mathcal B}^k f)(\theta)$ is of the order $\delta^{k},$
and, if $f$ is $k+1$ times continuously differentiable function, then the bias of estimator 
$f_k(\hat \theta)$ of $f(\theta)$
\begin{align*}
{\mathbb E}_{\theta} f_k(\hat \theta)-f(\theta)= (-1)^k ({\mathcal B}^{k+1}f)(\theta)= 
O(\delta^{k+1}).
\end{align*}   
This heuristic was justified in \cite{Koltchinskii_Zhilova_19} (with some ideas developed earlier in \cite{Koltchinskii_2017, Koltchinskii_2018, Koltchinskii_Zhilova}) using representations of bootstrap 
chains as superpositions of so called {\it random homotopies}. 

A random homotopy between parameter $\theta$ and its estimator $\hat \theta$ is 
an a.s. continuous stochastic process $H: T\times [0,1]\times\Omega \mapsto T$
defined on a probability space $(\Omega, {\mathcal F}, {\mathbb P})$ such that, 
for all $\theta \in \Theta,$ 
\begin{gather*}
H(\theta;0) \coloneqq \theta, \  H(\theta;1) \overset{d}{=} \thetah,\ {\rm where}\ \thetah \sim P(\theta;\cdot).
\end{gather*}
In addition, random homotopy $H(\theta, t), \theta\in T, t\in [0,1]$ will be assumed to be sufficiently smooth. 
In other words, random homotopy is a coupling that provides a smooth path between parameter $\theta$
and a random variable in the parameter space with the same distribution as the estimator $\hat \theta.$
Given i.i.d. copies $H_1, H_2, \dots,$ one can define their superpositions $G_k:= H_k\bullet \dots \bullet H_1$
as follows: 
\begin{align*}
G_k(\theta; t_1,\dots, t_k):= H_k(G_{k-1}(\theta; t_1,\dots, t_{k-1}), t_k), (t_1,\dots, t_k)\in [0,1]^k
\end{align*}
with $G_0 \equiv \theta.$ One can also define a Markov chain $\tilde \theta^{(k)}:= G_k(\theta;1,\dots, 1)$
with $\tilde \theta^{(0)}=\theta$ and show that 
\begin{align*}
(\hat \theta^{(k)}: k\geq 0) \overset{d}{=} (\tilde \theta^{(k)}: k\geq 0),
\end{align*}
see Lemma 4.1 in \cite{Koltchinskii_Zhilova_19}. Moreover, it is also shown in the same 
lemma that 
\begin{align*}
\hat \theta_l \overset{d}{=} G_k(\theta;t_1,\dots, t_k), (t_1,\dots, t_k)\in \{0,1\}^k, \sum_{i=1}^k t_i = l.
\end{align*}
Using these facts, it is easy to derive the following representation of $({\mathcal B}^{k} f)(\theta)$
\begin{align*}
({\mathcal B}^k f)(\theta) = {\mathbb E} \Delta^{(1)}\dots \Delta^{(k)}f(G_k(\theta;t_1,\dots, t_k)),
\end{align*}
where 
\begin{align*}
\Delta^{(i)}\varphi (t_1,\dots, t_k) = \varphi (t_1,\dots, t_i,\dots, t_k)_{|t_i=1}-  
\varphi (t_1,\dots, t_i,\dots, t_k)_{|t_i=0}, i=1,\dots, k.
\end{align*}
Under proper smoothness assumptions on $f$ and on random homotopies, this yields 
the following formula:
\begin{align*}
({\mathcal B}^k f)(\theta) = \int_{0}^1\dots \int_{0}^1 {\mathbb E} \frac{\partial^k f(G_k(\theta;t_1,\dots, t_k))}{\partial t_1\dots \partial t_k}
dt_1\dots dt_k,
\end{align*}
This approach and other analytic techniques developed in \cite{Koltchinskii_Zhilova_19} 
led to the bounds on the H\"older norms of functions ${\mathcal B}^{j} f$ and $f_k$
as well as the bounds on the bias of estimator $f_k(\hat \theta)$ of $f(\theta).$

For a function $V: T\times [0,1]\mapsto F$
with values in a Banach space $F$ and such that $V(\cdot, t)\in C^s(T), t\in [0,1],$ denote 
\begin{align*}
\|V\|_{C^s(T\times [0,1])}^{\sim}:= \sup_{t\in [0,1]} \|V(\cdot, t)\|_{C^s(T)}
\end{align*} 
and 
\begin{align*}
\|V\|_{C^{0,s}(T\times [0,1])}^{\sim}:= \sup_{t\in [0,1]} \|V(\cdot, t)\|_{C^{0,s}(T)}
\end{align*}

We will summarize some facts proved in \cite{Koltchinskii_Zhilova_19}
(see, in particular, Theorem 3.1, Theorem 3.2, Proposition 7.1).

\begin{proposition}
\label{prop_summary}
Let $s=k+1+\rho,$ $k\geq 1, \rho\in (0,1].$
Assume that $H(\theta;t)$ is $k+1$ times continuously differentiable in $T\times [0,1]$ and 
let $\dot H(\theta; t):= \frac{d}{dt}H(\theta;t).$ Then, the following statements hold: 
\begin{enumerate}
\item If 
\begin{align}
\label{assume_glavnij}
{\mathbb E} (\|H\|_{C^{0,s-1}(T\times [0,1])}^{\sim})^{s-1} \|\dot H\|_{C^{s-1}(T\times [0,1])}^{\sim}<+\infty,
\end{align}
then 
\begin{align*}
\|{\mathcal B}\|_{C^s(T)\mapsto C^{s-1}(T)} \leq 4(k+1)^{k+2}{\mathbb E} (\|H\|_{C^{0,s-1}(T\times [0,1])}^{\sim})^{s-1} \|\dot H\|_{C^{s-1}(T\times [0,1])}^{\sim}.
\end{align*}
\item Moreover, under the same assumption, for some constant $D_s$ and for all $j=1,\dots, k,$
\begin{align*}
\|{\mathcal B}^j\|_{C^s(T)\mapsto C^{1+\rho}(T)} \leq D_s\Bigl({\mathbb E} (\|H\|_{C^{0,s-1}(T\times [0,1])}^{\sim})^{s-1} \|\dot H\|_{C^{s-1}(T\times [0,1])}^{\sim}\Bigr)^j.
\end{align*}
\item If 
\begin{align*}
D_s {\mathbb E} (\|H\|_{C^{0,s-1}(T\times [0,1])}^{\sim})^{s-1} \|\dot H\|_{C^{s-1}(T\times [0,1])}^{\sim}
\leq 1/2,
\end{align*}
then 
\begin{align*}
\|f_k\|_{C^{1+\rho}(T)}\leq 2 \|f\|_{C^s(T)}. 
\end{align*}
\item If assumption \eqref{assume_glavnij} holds, 
then for all $\theta\in T,$
\begin{align*}
|{\mathbb E}_{\theta} f_k(\hat \theta)-f(\theta)| \lesssim_s 
\|f\|_{C^s(T)} &\Bigl({\mathbb E} (\|H\|_{C^{0,s-1}(T\times [0,1])}^{\sim})^{s-1} \|\dot H\|_{C^{s-1}(T\times [0,1])}^{\sim}\Bigr)^k 
\\
&
\times \Bigl(\Bigl\|{\mathbb E}\int_0^1 \dot H(\theta;t) dt\Bigr\|+ {\mathbb E}\|\dot H\|_{L_{\infty}(T\times [0,1])}^{1+\rho}\Bigr).
\end{align*}
\end{enumerate}
\end{proposition}

These facts provide a way to control the bias of estimator $f_k(\hat \theta)$ and, using the smoothness 
of function $f_k,$ to study the concentration of $f_k(\hat \theta)$ around its expectation (in the case 
of normal models where Gaussian concentration could be used). 
However, both the construction of random homotopies and the development 
of concentration bounds for more general statistical models than Gaussian are 
challenging problems. 

In this paper, we try to get around this difficulty by using 
the normal approximation of estimator $\hat \theta$ and reducing the problem 
to the Gaussian case. More precisely, instead of developing random homotopies 
directly for the estimator $\hat \theta,$ we use a very simple random homotopy 
\begin{align*}
H(\theta;t) := \theta + \frac{t\xi(\theta)}{\sqrt{n}} 
\end{align*}
for the ``estimator" $\tilde \theta= G(\theta)=\theta+\frac{\xi(\theta)}{\sqrt{n}},$ or,
more precisely, for a slightly modified ``estimator" $\tilde \theta_{\delta},$ 
defined as follows:
$$
\tilde \theta_{\delta} := G_{\delta}(\theta) := \theta+ \frac{\xi_{\delta}(\theta)}{\sqrt{n}} \in \Theta_{\delta}, 
$$
where
$$
\xi_{\delta}(\theta):= \xi(\theta) I\Bigl(\|\xi\|_{L_{\infty}(E)}< \delta \sqrt{n}\Bigr), \theta \in E.
$$
This would allow us to prove our results under smoothness assumptions on the process 
$\xi$ and functional $f$ locally in a neighborhood of $\Theta.$
For these estimators, we construct the corresponding
bootstrap chains $\tilde \theta^{(k)}$ and $\tilde \theta^{(k)}_{\delta}$ and 
show that these chains approximate in distribution the bootstrap chain 
$\hat \theta^{(k)}$ of the initial estimator $\hat \theta$ 
(see Theorem \ref{hat^k_tilde^k_hat_tilde_AAA} in 
Section \ref{Gaussian_approx}). 
We also approximate operator ${\mathcal T}$ by the operators $\tilde {\mathcal T}$ and $\tilde {\mathcal T}_{\delta}:$
\begin{align*}
&
(\tilde {\mathcal T} f)(\theta) := {\mathbb E}_{\theta} f(\tilde \theta)
= {\mathbb E} f(G(\theta)),
\\
&
(\tilde {\mathcal T}_{\delta} f)(\theta) := {\mathbb E}_{\theta} f(\tilde \theta_{\delta})
= {\mathbb E} f(G_{\delta}(\theta)),\ \theta\in E, f\in {\rm Lip}(E)
\end{align*}
and define $\tilde {\mathcal B}:= \tilde {\mathcal T}-{\mathcal I},$ 
$\tilde {\mathcal B}_{\delta}:= \tilde {\mathcal T}_{\delta}-{\mathcal I}.$
This allows us to 
approximate the function $f_k$ by similar 
functions $\tilde f, \tilde f_{\delta,k}$ defined as 
follows
\begin{align*}
\tilde f_k(\theta):= \sum_{j=0}^k (-1)^j (\tilde {\mathcal B}^{j}f)(\theta),
\tilde f_{\delta,k}(\theta):= \sum_{j=0}^k (-1)^j (\tilde {\mathcal B}_{\delta}^{j}f)(\theta)
\end{align*}
(see Theorem \ref{uniform_bd}). 
Finally, in Section \ref{Gaussian_conc}, 
we use Gaussian concentration (more precisely, Maurey-Pisier type inequalities)
to control the norm 
\begin{align*}
\sup_{\theta\in \Theta}
\Bigl\|\tilde f_{\delta,k}(\tilde \theta_{\delta})-f(\theta)-n^{-1/2}\langle f'(\theta), \xi(\theta)\rangle
\Bigr\|_{L_{\psi}({\mathbb P})}
\end{align*}
(see Theorem \ref{conc_main}). 
We combine all these pieces together in Section \ref{proof_main} to complete the proofs.

\begin{remark}
\normalfont
It easily follows from the proofs of the 
main results that they also hold for estimators $\tilde f_k(\hat \theta)$
and $\tilde f_{\delta, k}(\hat \theta),$ based on the functionals related 
to the Gaussian approximation of estimator $\hat \theta.$ 
\end{remark}

\section{Gaussian approximation of bootstrap chains}
\label{Gaussian_approx}

The main goal of this section is to use the Gaussian approximation of estimator $\hat \theta$ in order to develop 
certain Markov chains approximating the bootstrap chain $\hat \theta^{(k)}, k\geq 0.$

We will need some additional definitions and notations. 
For a given $\theta\in T,$ let 
\begin{align*}
\hat \theta_{\delta}:= 
\begin{cases}
\hat \theta &\ {\rm if}\ \|\hat \theta-\theta\|<\delta\\
\theta &\ {\rm if}\  \|\hat \theta-\theta\|\geq \delta
\end{cases}
\end{align*} 
and let $\hat \theta_{\delta}^{(k)}, k\geq 0$ be the corresponding 
bootstrap chain. More precisely, $\hat \theta_{\delta}^{(k)}, k\geq 0$
is a Markov chain with $\hat \theta_{\delta}^{(0)}=\theta$ and with 
the transition kernel 
\begin{align*}
P_{\delta}(\theta;A) := {\mathbb P}_{\theta}\{\hat \theta_{\delta}\in A\}=
{\mathbb P}_{\theta}\{\hat \theta \in A, \|\hat \theta-\theta\|<\delta\}
+ I_A(\theta){\mathbb P}_{\theta}\{\|\hat \theta-\theta\|\geq \delta\}, 
\end{align*}	
where $A\subset T, \theta\in T.$
Note that, for all $k\geq 0,$ $\hat \theta_{\delta}^{(k)}\in \Theta_{k\delta}.$
Similarly to \eqref{define_T}, define operator ${\mathcal T}_{\delta}$ as follows:
\begin{align*}
({\mathcal T}_{\delta} g)(\theta):= {\mathbb E}_{\theta} g(\hat \theta_{\delta}), \theta \in T.
\end{align*}
Recall that $\tilde \theta= G(\theta)=\theta+\frac{\xi(\theta)}{\sqrt{n}}$ and 
$
\tilde \theta_{\delta} = G_{\delta}(\theta) := \theta+ \frac{\xi_{\delta}(\theta)}{\sqrt{n}} \in \Theta_{\delta}. 
$
Note that 
\begin{align}
\label{T^k}
({\mathcal T}^k f)(\theta)= {\mathbb E}_{\theta} f(\hat \theta^{(k)}),\ \ 
({\mathcal T}_{\delta}^k f)(\theta)= {\mathbb E}_{\theta} f(\hat \theta_{\delta}^{(k)})
\end{align}
and 
\begin{align}
\label{T_delta^k}
(\tilde {\mathcal T}^k f)(\theta)= {\mathbb E}_{\theta} f(\tilde \theta^{(k)}),
\ \
(\tilde {\mathcal T}_{\delta}^k f)(\theta)= {\mathbb E}_{\theta} f(\tilde \theta_{\delta}^{(k)}), k\geq 0.
\end{align}

We will now try to approximate the bootstrap chain $\hat \theta^{(k)}, k\geq 0$ by the Markov chains $\tilde \theta^{(k)}, k\geq 0$ and $\tilde \theta_{\delta}^{(k)}, k\geq 0$ (the bootstrap chains 
of $\tilde \theta$ and $\tilde \theta_{\delta}$) defined as follows:
\begin{align*}
&
\tilde \theta^{(0)}:= \theta,
\\
&
\tilde \theta^{(k)}:= F_{k}(\theta), k\geq 1,
\end{align*}
where $F_{k}:= G_{k}\circ \dots \circ G_{1}, k\geq 1,$
\begin{align*}
G_j(\theta):=\theta + \frac{\xi_j(\theta)}{\sqrt{n}},
\end{align*}
$\xi_j, j\geq 1$ being i.i.d. copies of $\xi;$
\begin{align*}
&
\tilde \theta_{\delta}^{(0)}:= \theta,
\\
&
\tilde \theta^{(k)}_{\delta}:= F_{k,\delta}(\theta), k\geq 1,
\end{align*}
where $F_{k,\delta}:= G_{k,\delta}\circ \dots \circ G_{1,\delta}, k\geq 1,$
\begin{align*}
&
G_{j,\delta}(\theta):= \theta + \frac{\xi_{j,\delta}(\theta)}{\sqrt{n}}, 
\\
&
\xi_{j,\delta}(\theta):= \xi_j(\theta)I\Bigl(\|\xi_j\|_{L_{\infty}(E)}<\delta \sqrt{n}\Bigr), \theta \in E.
\end{align*}
It follows from the above definitions that, for all $\theta\in \Theta$ and all $k\geq 0,$ 
$\tilde \theta_{\delta}^{(k)}\in \Theta_{k\delta}.$ 

For r.v. $\eta_1\sim \mu_{\theta}, \eta_2\sim \nu_{\theta}, \theta\in \Theta$ with values in $E,$ denote 
\begin{align*}
\Delta_{s,\delta}(\eta_1,\eta_2) 
&
= \Delta_{s,\delta,\Theta}(\eta_1,\eta_2):= \sup_{\theta\in \Theta} \sup_{\|f\|_{C^s(\Theta_{\delta})}\leq 1}
\Bigl|{\mathbb E}_{\theta}f(\eta_1)-{\mathbb E}_{\theta}f(\eta_2)\Bigr| 
\\
&
= \Delta_{{\mathcal F},\Theta}(\eta_1,\eta_2),
\end{align*}
where ${\mathcal F}:= \{f:\|f\|_{C^s(\Theta_{\delta})}\leq 1\}.$

The following theorem provides bounds on approximation of Markov chain $\hat \theta^{(k)}$ by Markov chains $\tilde \theta^{(k)}$ and $\tilde \theta_{\delta}^{(k)}.$ 

\begin{theorem}
	\label{hat^k_tilde^k_hat_tilde_AAA}
	For all $s\geq 1, \delta>0, k\geq 1$ such that $\Theta_{k\delta}\subset T,$
	\begin{align}
	\label{Delta_s,kdelta}
	&
	\Delta_{s,k\delta} (\hat \theta^{(k)}, \tilde \theta_{\delta}^{(k)})
	\lesssim_{s,k} 
	\biggl(1+ \frac{{\mathbb E}\|\xi\|_{C^s(\Theta_{(k-1)\delta})}^s}{n^{s/2}}\biggr)^{k-1}
	\Bigl[\Delta_{s,\delta}(\hat \theta, \tilde \theta)+ {\frak Q}_n(\Theta_{(k-1)\delta},\delta)
	\Bigr]
	\end{align} 
	and 
	\begin{align*}
	&
	\nonumber
	\Delta_{s,k\delta} (\hat \theta^{(k)}, \tilde \theta^{(k)})
	\lesssim_{s,k} 
	\biggl(1+ \frac{{\mathbb E}\|\xi\|_{C^s(\Theta_{(k-1)\delta})}^s}{n^{s/2}}\biggr)^{k-1}
	\Bigl[\Delta_{s,\delta}(\hat \theta, \tilde \theta)
	+{\frak Q}_n(\Theta_{(k-1)\delta},\delta)\Bigr],
	\end{align*}
	where 
	\begin{align*}
	\frak{Q}_n(\Theta, \delta):=\sup_{\theta\in \Theta}{\mathbb P}\{\|\hat \theta-\theta\|\geq \delta\}
	+{\mathbb P}\{\|\xi\|_{L_{\infty}(E)}\geq \delta \sqrt{n}\}.
	\end{align*}
	In particular, if $\Theta=T=E,$ then 
	\begin{align*}
	&
	\Delta_{s} (\hat \theta^{(k)}, \tilde \theta^{(k)})
	\lesssim_{s,k} 
	\biggl(1+ \frac{{\mathbb E}\|\xi\|_{C^s(E)}^s}{n^{s/2}}\biggr)^{k-1}\Delta_{s}(\hat \theta, \tilde \theta).
	\end{align*} 
\end{theorem}

In addition, the following result providing a bound on approximation error 
of the function $f_k$ by function the $\tilde f_{\delta,k}$ holds. 

\begin{theorem}
	\label{uniform_bd}
	For all $s=k+1+\rho, k\geq 1, \rho\in (0,1]$ and $\delta>0$ such that $\Theta_{k\delta}\subset T,$
	\begin{align*}
	&
	\|f_k-\tilde f_{\delta,k}\|_{L_{\infty}(\Theta)}
	\\
	&
	\lesssim_{s,k} 
	\|f\|_{C^s(\Theta_{k\delta})}\biggl(1+ \frac{{\mathbb E}\|\xi\|_{C^s(\Theta_{(k-1)\delta})}^s}{n^{s/2}}\biggr)^{k-1}
	\Bigl[\Delta_{s,\delta}(\hat \theta, \tilde \theta)+\frak{Q}_n(\Theta_{(k-1)\delta},\delta)
	\Bigr].
	\end{align*}
\end{theorem}

\subsection{Preliminary facts related to Fa\`a di Bruno calculus}

We will start with several technical facts (in particular, bounds on the norms 
of operators $\tilde {\mathcal T}_{\delta}$ and $\tilde {\mathcal B}_{\delta}:=\tilde {\mathcal T}_{\delta}-{\mathcal I}$) that are needed both in this and the the following sections. 
These facts are based on Fa\`a di Bruno type calculus developed in \cite{Koltchinskii_Zhilova_19} and they are just modifications of the 
results already presented there (see, in particular, Section 5 in \cite{Koltchinskii_Zhilova_19};
see also Proposition \ref{prop_summary} in the current paper).

The next lemma provides a bound on H\"older $C^s$-norm of superposition 
of two $C^s$ functions. 

\begin{lemma}
	\label{superposition_h_g}
	Let $E,F$ be Banach spaces and $U\subset E, V\subset F$ be open sets.
	Suppose $g: U\mapsto F,$ $g(U)\subset V\subset F,$ $f: V\mapsto {\mathbb R}$
	and  $g\in C^{s}(U), f\in C^s(V)$ for some $s\geq 1.$ Then 
	\begin{align*}
	\|f\circ g\|_{C^s(U)} \lesssim_s \|f\|_{C^{s}(V)} (1 \vee \|g\|_{C^{0,s}(U)}^s).
	\end{align*}
\end{lemma}

\begin{remark}
\label{proofs_superposition}	
\normalfont
The proof of Lemma \ref{superposition_h_g} is based on a version of Fa\`a di Bruno's 
formula briefly discussed below (see Section 2 in \cite{Koltchinskii_Zhilova_19} for the details on tensor 
product notations). 
Let $I_k:=\{1,\dots, k\}.$
For $I=\{i_1,\dots, i_l\}\subset I_k,$ $1\leq i_1<\dots<i_l\leq k$ and $h_i\in E, i=1,\dots, k,$
denote $h_I:= h_{i_1}\otimes \dots \otimes h_{i_l}.$ For $j=1,\dots, k,$
let 
\begin{align*}
{\mathcal K}_j := \{(k_1,\dots, k_j): k_i\geq 1, i=1,\dots, j, k_1\geq \dots \geq k_j, \sum_{i=1}^j k_i=k\}.
\end{align*}
Let ${\mathcal D}_{I_{k}, k_1,\dots, k_j}$ be the set of all partitions $(\Delta_1,\dots, \Delta_j)$
of set $I_k$ into disjoint subsets such that ${\rm card}(\Delta_i)=k_i, i=1,\dots, j.$ 
Finally, let 
\begin{align*}
h_{k_1,\dots, k_j}:= 
\sum_{(\Delta_1,\dots, \Delta_j)\in {\mathcal D}_{I_k, k_1,\dots, k_j}} 
h_{\Delta_1}\otimes \dots \otimes h_{\Delta_j}
\end{align*} 
and 
\begin{align*}
(D^{k_1,\dots, k_j} g)(x):= (D^{k_1}g)(x)\otimes \dots \otimes (D^{k_j}g)(x).
\end{align*}
Then 
\begin{align*}
D^k (f\circ g)(x)[h_1\otimes \dots \otimes h_k] = 
\sum_{j=1}^k \sum_{(k_1,\dots, k_j)\in {\mathcal K}_j}
(D^j f)(g(x))[(D^{k_1,\dots, k_j} g)(x)[h_{k_1,\dots, k_j}]].
\end{align*}
The proof of this formula is similar to the proof of Theorem 5.3
in \cite{Koltchinskii_Zhilova_19}.
The formula provides a way to control Lipschitz norms of the derivatives 
$D^k (f\circ g)$ for all $k<s$ similarly to the proof of Proposition 5.1
in \cite{Koltchinskii_Zhilova_19}, which leads to the bound of Lemma \ref{superposition_h_g}.
The proof of Lemma \ref{lemma_g_varphi} in Section \ref{proof_main} below is based on a similar argument. 
\end{remark}

In what follows, we need several lemmas to control the norms of operators 
$\tilde {\mathcal T}_{\delta},$ $\tilde {\mathcal B}_{\delta}$ and their powers. 

\begin{lemma}
\label{bd_T_delta}
For all $s\geq 1, \delta>0,$
\begin{align*}
\|\tilde {\mathcal T}_{\delta}\|_{C^s(\Theta_{\delta})\mapsto C^s(\Theta)}
\lesssim_s 1+ \frac{{\mathbb E}\|\xi\|_{C^s(\Theta)}^s}{n^{s/2}}.
\end{align*}
\end{lemma}

Bound of Lemma \ref{bd_T_delta}
easily follows from Lemma  \ref{superposition_h_g} with $U=\Theta,$ $V=\Theta_{\delta},$
$f\in C^s(\Theta_{\delta})$ and $g(\theta)=G_{\delta}(\theta)=\theta+\frac{\xi_{\delta}(\theta)}{\sqrt{n}},\theta\in \Theta.$
Since $\|\xi_\delta\|_{L_{\infty}(E)}<\delta\sqrt{n},$ we have $G_{\delta}(\Theta)\subset \Theta_{\delta},$ which allows one to 
apply  Lemma  \ref{superposition_h_g}.

Lemma \ref{bd_T_delta} implies the following:

\begin{lemma}
\label{f_delta_k_first}
For all $s\geq 1, \delta>0, k\geq 1,$
\begin{align*}
&
\|\tilde f_{\delta,k}\|_{C^s(\Theta)}
\lesssim_{s,k} \biggl(1+ \frac{{\mathbb E}\|\xi\|_{C^s(\Theta_{(k-1)\delta})}^s}{n^{s/2}}\biggr)^{k}
\|f\|_{C^s(\Theta_{k\delta})}.
\end{align*}
\end{lemma}

\begin{proof}
Note that bound of Lemma \ref{bd_T_delta} implies that 
\begin{align*}
\|\tilde {\mathcal B}_{\delta}\|_{C^s(\Theta_{\delta})\mapsto C^s(\Theta)}
\lesssim_s 1+ \frac{{\mathbb E}\|\xi\|_{C^s(\Theta)}^s}{n^{s/2}}.
\end{align*}
Therefore,
\begin{align*}
&
\|\tilde f_{\delta,k}\|_{C^s(\Theta)}
\leq \sum_{j=0}^k \|\tilde {\mathcal B}_{\delta}^j f\|_{C^s(\Theta_{j\delta})}
\leq \sum_{j=0}^k \|\tilde {\mathcal B}_{\delta}^j\|_{C^s(\Theta_{j\delta})\mapsto C^s(\Theta)}
\|f\|_{C^s(\Theta_{j\delta})}
\\
&
\leq 
\sum_{j=0}^k 
\prod_{i=1}^j \|\tilde {\mathcal B}_{\delta}\|_{C^s(\Theta_{i\delta})\mapsto C^s(\Theta_{(i-1)\delta})}\|f\|_{C^s(\Theta_{k\delta})}
\lesssim \sum_{j=0}^k \biggl(1+ \frac{{\mathbb E}\|\xi\|_{C^s(\Theta_{(k-1)\delta})}^s}{n^{s/2}}\biggr)^j  
\|f\|_{C^s(\Theta_{k\delta})}
\\
&
\lesssim_{s,k} \biggl(1+ \frac{{\mathbb E}\|\xi\|_{C^s(\Theta_{(k-1)\delta})}^s}{n^{s/2}}\biggr)^{k}\|f\|_{C^s(\Theta_{k\delta})}.
\end{align*}
\end{proof}

We will also use the following lemma that provides a bound 
on operator $\tilde {\mathcal B}_{\delta} =\tilde {\mathcal T}_{\delta}-{\mathcal I}.$
It is a slight modification of Proposition 5.1 in \cite{Koltchinskii_Zhilova_19}
(see also Proposition \ref{prop_summary}).

\begin{lemma}
\label{f_delta_k_second}
For all $s\geq 2,$ 
\begin{align*}
\|\tilde {\mathcal B}_{\delta}\|_{C^s(\Theta_{\delta})\mapsto C^{s-1}(\Theta)}
\leq 4 s^{s+1} {\mathbb E} \biggl(1+ \frac{\|\xi\|_{C^{s-1}(\Theta)}}{\sqrt{n}}\biggr)^{s-1}
\frac{\|\xi\|_{C^{s-1}(\Theta)}}{\sqrt{n}}.
\end{align*}
\end{lemma}

This bound implies that 
\begin{align*}
\|\tilde {\mathcal B}_{\delta}\|_{C^s(\Theta_{i\delta})\mapsto C^{s-1}(\Theta_{(i-1)\delta})}
\leq 4 s^{s+1} {\mathbb E} \biggl(1+ \frac{\|\xi\|_{C^{s-1}(\Theta_{(i-1)\delta})}}{\sqrt{n}}\biggr)^{s-1}
\frac{\|\xi\|_{C^{s-1}(\Theta_{(i-1)\delta})}}{\sqrt{n}}.
\end{align*}
Iterating the last bound $j$ times for some $j<s-1$ yields
\begin{align*}
&
\|\tilde {\mathcal B}_{\delta}^j\|_{C^s(\Theta_{j\delta})\mapsto C^{s-j}(\Theta)}
\leq 
\prod_{i=1}^j \|\tilde {\mathcal B}_{\delta}\|_{C^{s-(j-i)}(\Theta_{i\delta})\mapsto C^{s-(j-i)-1}(\Theta_{(i-1)\delta})}
\\
&
\leq 
\prod_{i=1}^j   \biggl(4 (s-(j-i))^{s-(j-i)+1} 
{\mathbb E} \biggl(1+ \frac{\|\xi\|_{C^{s-(j-i)-1}(\Theta_{(i-1)\delta})}}{\sqrt{n}}\biggr)^{s-1}
\frac{\|\xi\|_{C^{s-(j-i)-1}(\Theta_{(i-1)\delta})}}{\sqrt{n}}\biggr)
\\
&
\leq \biggl(4 s^{s+1} {\mathbb E} \biggl(1+ \frac{\|\xi\|_{C^{s-1}(\Theta_{(j-1)\delta})}}{\sqrt{n}}\biggr)^{s-1}
\frac{\|\xi\|_{C^{s-1}(\Theta_{(j-1)\delta})}}{\sqrt{n}}\biggr)^j.
\end{align*}
Thus, we have that for all $j=1,\dots, k,$
\begin{align}
\label{bd_B_delta^j}
\|\tilde {\mathcal B}_{\delta}^j\|_{C^s(\Theta_{k\delta})\mapsto C^{s-j}(\Theta)}
\leq \biggl(4 s^{s+1} {\mathbb E} \biggl(1+ \frac{\|\xi\|_{C^{s-1}(\Theta_{(k-1)\delta})}}{\sqrt{n}}\biggr)^{s-1}
\frac{\|\xi\|_{C^{s-1}(\Theta_{(k-1)\delta})}}{\sqrt{n}}\biggr)^j.
\end{align}

\begin{lemma}
\label{bd_on_tilde_f_kdelta_B}
Let $s=k+1+\rho$ with $k\geq 1$ and $\rho\in (0,1].$
Assume that 
\begin{align*}
4 s^{s+1} {\mathbb E} \biggl(1+ \frac{\|\xi\|_{C^{s-1}(\Theta_{(k-1)\delta})}}{\sqrt{n}}\biggr)^{s-1}
\frac{\|\xi\|_{C^{s-1}(\Theta_{(k-1)\delta})}}{\sqrt{n}}\leq 1/2.
\end{align*}
Then 
\begin{align}
\label{bd_on_tilde_f_kdelta_B_one}
&
\|\tilde f_{\delta,k}\|_{C^{1+\rho}(\Theta)}
\leq 2\|f\|_{C^s(\Theta_{k\delta})}
\end{align}
and 
\begin{align}
\label{bd_on_tilde_f_kdelta_B_two}
\|\tilde f_{\delta,k}-f\|_{C^{1+\rho}(\Theta)}
\lesssim_s 
\|f\|_{C^s(\Theta_{k\delta})} {\mathbb E} \biggl(1+ \frac{\|\xi\|_{C^{s-1}(\Theta_{(k-1)\delta})}}{\sqrt{n}}\biggr)^{s-1}
\frac{\|\xi\|_{C^{s-1}(\Theta_{(k-1)\delta})}}{\sqrt{n}}.
\end{align}
\end{lemma}

\begin{proof}
Indeed, to prove \eqref{bd_on_tilde_f_kdelta_B_one}, note that 
\begin{align*}
&
\|\tilde f_{\delta,k}\|_{C^{1+\rho}(\Theta)}=
\|\tilde f_{\delta,k}\|_{C^{s-k}(\Theta)}
\leq \sum_{j=0}^k \|\tilde {\mathcal B}_{\delta}^j f\|_{C^s(\Theta_{j\delta})}
\leq \sum_{j=0}^k \|\tilde {\mathcal B}_{\delta}^j\|_{C^s(\Theta_{j\delta})\mapsto C^{s-j}(\Theta)}
\|f\|_{C^s(\Theta_{j\delta})}
\\
&
\leq \sum_{j=0}^k \biggl(4 s^{s+1} {\mathbb E} \biggl(1+ \frac{\|\xi\|_{C^{s-1}(\Theta_{(k-1)\delta})}}{\sqrt{n}}\biggr)^{s-1}
\frac{\|\xi\|_{C^{s-1}(\Theta_{(k-1)\delta})}}{\sqrt{n}}\biggr)^j  \|f\|_{C^s(\Theta_{k\delta})}
\\
&
\leq \sum_{j=0}^k 2^{-j} \|f\|_{C^s(\Theta_{k\delta})} \leq 2\|f\|_{C^s(\Theta_{k\delta})}
\end{align*}
and bound \eqref{bd_on_tilde_f_kdelta_B_two} is proved similarly by controlling the sum from $j=1$ to $k.$
\end{proof}

\subsection{Proofs of approximation bounds for bootstrap chains}

We are now ready to provide the proofs of theorems \ref{hat^k_tilde^k_hat_tilde_AAA}
and \ref{uniform_bd}.

\begin{proof}
We will start with the following simple lemma.
\begin{lemma}
\label{bd_T-T_delta}
For all $s\geq 1, \delta>0,$
\begin{align*}
&
|\Delta_{s,\delta}(\hat \theta,\tilde \theta)- \Delta_{s,\delta}(\hat \theta_{\delta},\tilde \theta_{\delta})|
\leq 2\frak{Q}_n(\Theta,\delta).
\end{align*}
In addition, 
\begin{align*}
&
|\Delta_{s,\delta}(\hat \theta,\tilde \theta)- \Delta_{s,\delta}(\hat \theta,\tilde \theta_{\delta})|
\leq 
2{\mathbb P}\{\|\xi\|_{L_{\infty}(E)}\geq \delta \sqrt{n}\}\leq 2\frak{Q}_n(\Theta,\delta).
\end{align*}
\end{lemma}

\begin{proof}
Note that 
\begin{align*}
&
|\Delta_{s,\delta}(\hat \theta,\tilde \theta)- \Delta_{s,\delta}(\hat \theta_{\delta},\tilde \theta_{\delta})|
\\
&
=
\Bigl|\sup_{\theta \in \Theta}\sup_{\|f\|_{C^s(\Theta_{\delta})}\leq 1} 
\Bigl| {\mathbb E}_{\theta} f(\hat \theta)- {\mathbb E}_{\theta} f(\tilde \theta)\Bigr|
- \sup_{\theta \in \Theta}\sup_{\|f\|_{C^s(\Theta_{\delta})}\leq 1}
\Bigl| {\mathbb E}_{\theta} f(\hat \theta_{\delta})- {\mathbb E}_{\theta} f(\tilde \theta_{\delta})\Bigr|
\Bigr|
\\
&
\leq 
\sup_{\theta \in \Theta}\sup_{\|f\|_{C^s(\Theta_{\delta})}\leq 1} 
\Bigl| {\mathbb E}_{\theta} f(\hat \theta)- {\mathbb E}_{\theta} f(\hat \theta_{\delta})\Bigr|
+
\sup_{\theta \in \Theta}\sup_{\|f\|_{C^s(\Theta_{\delta})}\leq 1} 
\Bigl| {\mathbb E}_{\theta} f(\tilde \theta)- {\mathbb E}_{\theta} f(\tilde \theta_{\delta})\Bigr|.
\end{align*}
Since $\hat \theta = \hat \theta_{\delta}$ on the event $\{\|\hat \theta-\theta\|< \delta\},$
we get
\begin{align*}
\sup_{\theta\in \Theta}\Bigl| {\mathbb E}_{\theta} f(\hat \theta)- {\mathbb E}_{\theta} f(\hat \theta_{\delta})\Bigr|
\leq 2\|f\|_{L_{\infty}(E)} \sup_{\theta\in \Theta}{\mathbb P}_{\theta}
\{\|\hat \theta-\theta\|\geq \delta\}.
\end{align*}
Similarly,
\begin{align*}
&
\sup_{\theta\in \Theta}\Bigl| {\mathbb E}_{\theta} f(\tilde \theta)- {\mathbb E}_{\theta} f(\tilde \theta_{\delta})\Bigr|
\leq 2\|f\|_{L_{\infty}(E)} 
\sup_{\theta\in \Theta}{\mathbb P}\{\|\xi(\theta)\|\geq \delta\sqrt{n})\},
\end{align*}
which easily implies the claim.
\end{proof}

\begin{lemma}
\label{lm_delta_hat_tilde}
For all $\delta>0, k\geq 1,$
\begin{align*}
\sup_{\theta \in \Theta}|{\mathbb E}_{\theta} f(\hat \theta^{(k)})- 
{\mathbb E}_{\theta} f(\hat \theta_{\delta}^{(k)})|
\leq 2k \|f\|_{L_{\infty}(E)}
\sup_{\theta\in \Theta_{(k-1)\delta}}{\mathbb P}_{\theta}\{\|\hat \theta-\theta\|\geq \delta\}
\end{align*}
and 
\begin{align*}
\sup_{\theta \in \Theta}|{\mathbb E}_{\theta} f(\tilde\theta^{(k)})- 
{\mathbb E}_{\theta} f(\tilde \theta_{\delta}^{(k)})|
\leq 2k \|f\|_{L_{\infty}(E)}
{\mathbb P}\{\|\xi\|_{L_{\infty}(E)}\geq \delta \sqrt{n}\}.
\end{align*}
\end{lemma}

\begin{proof}
Note that the chains $\hat \theta^{(j)}, j=0,\dots, k$ and $\hat \theta_{\delta}^{(j)}, j=0,\dots, k$
coincide on the event 
\begin{align*}
A_k:= \Bigl\{\|\hat \theta^{(j)}-\hat \theta^{(j-1)}\| < \delta, j=1,\dots, k\Bigr\}.
\end{align*}
Therefore, 
\begin{align*}
&
{\mathbb E}_{\theta} f(\hat \theta^{(k)})- 
{\mathbb E}_{\theta} f(\hat \theta_{\delta}^{(k)})
= 
{\mathbb E}_{\theta} (f(\hat \theta^{(k)})- f(\hat \theta_{\delta}^{(k)}))I_{A_k^c}
\end{align*}
and 
\begin{align*}
|{\mathbb E}_{\theta} f(\hat \theta^{(k)})- 
{\mathbb E}_{\theta} f(\hat \theta_{\delta}^{(k)})| \leq 
2\|f\|_{L_{\infty}(E)} {\mathbb P}_{\theta}(A_k^c).
\end{align*}
Note also that
\begin{align*}
{\mathbb P}_{\theta}(A_k^c) = \sum_{j=1}^k {\mathbb P}_{\theta}(B_j),
\end{align*}
where 
\begin{align*}
B_j := \{\|\hat \theta^{(i)}-\hat \theta^{(i-1)}\|< \delta, i=1,\dots, j-1, \|\hat \theta^{(j)}-\hat \theta^{(j-1)}\|\geq \delta\}.
\end{align*} 
On the event $B_j,$ $\|\hat \theta^{(i)}-\theta\|< \delta i, i=1,\dots, j-1,$ implying that, for all $\theta\in \Theta,$ 
$\hat \theta^{(i)}\in \Theta_{i \delta}, i=1,\dots, j-1.$ Therefore, 
\begin{align*}
\sup_{\theta\in \Theta} {\mathbb P}_{\theta}(B_j) 
&\leq 
\sup_{\theta\in \Theta} {\mathbb E}_{\theta}I(\|\hat \theta^{(j)}-\hat \theta^{(j-1)}\| \geq \delta)
I(\hat \theta^{(j-1)}\in \Theta_{(j-1)\delta})
\\
&
=\sup_{\theta\in \Theta} {\mathbb E}_{\theta}{\mathbb E}\Bigl(I(\|\hat \theta^{(j)}-\hat \theta^{(j-1)}\| \geq \delta)
I(\hat \theta^{(j-1)}\in \Theta_{(j-1)\delta})|\hat \theta^{(j-1)}\Bigr)
\\
&
=\sup_{\theta\in \Theta} {\mathbb E}_{\theta} {\mathbb P}_{\hat \theta^{(j-1)}}\{\|\hat \theta^{(j)}-\hat \theta^{(j-1)}\|\geq \delta\}I(\hat \theta^{(j-1)}\in \Theta_{(j-1)\delta})
\\
&
\leq \sup_{\theta \in \Theta_{(j-1)\delta}}{\mathbb P}_{\theta}\{\|\hat \theta-\theta\|\geq \delta\},
\end{align*}
which implies 
\begin{align*}
\sup_{\theta \in \Theta}{\mathbb P}_{\theta}(A_k^c)
\leq \sum_{j=1}^k 
\sup_{\theta \in \Theta_{(j-1)\delta}}{\mathbb P}_{\theta}\{\|\hat \theta-\theta\|\geq \delta\}
\leq k 
\sup_{\theta \in \Theta_{(k-1)\delta}}{\mathbb P}_{\theta}\{\|\hat \theta-\theta\|\geq \delta\}.
\end{align*} 
As a result, we get 
\begin{align*}
&
\sup_{\theta \in \Theta}|{\mathbb E}_{\theta} f(\hat \theta^{(k)})- 
{\mathbb E}_{\theta} f(\hat \theta_{\delta}^{(k)})|
\leq 
2 k\|f\|_{L_{\infty}(E)}
\sup_{\theta \in \Theta_{(k-1)\delta}}{\mathbb P}_{\theta}\{\|\hat \theta-\theta\|\geq \delta\}, 
\end{align*}
and the first claim follows. The proof of the second claim is similar.
\end{proof}

Our next goal is to bound the distances $\Delta_{s,k\delta} (\hat \theta^{(k)}, \tilde \theta^{(k)})$
and $\Delta_{s,k\delta} (\hat \theta_{\delta}^{(k)}, \tilde \theta_{\delta}^{(k)})$ in terms 
of $\Delta_{s,\delta}(\hat \theta, \tilde \theta),$ which could be reduced to
bounding the difference $({\mathcal T}_{\delta}^k-\tilde {\mathcal T}_{\delta}^k)f.$
Denote 
\begin{align*}
&
L_{k,s}(\delta):=\prod_{i=1}^{k} \|\tilde {\mathcal T}_{\delta}\|_{C^s(\Theta_{(i+1)\delta})\mapsto C^s(\Theta_{i\delta})}, k\geq 1,
\end{align*}
with $L_{0,s}(\delta):=1$ for $k=0.$

\begin{proposition}
\label{hat^k_tilde^k_hat_tilde}
For all $\delta>0, s\geq 1$ and for $k\geq 1,$
\begin{align*}
\Delta_{s,k\delta} (\hat \theta_{\delta}^{(k)}, \tilde \theta_{\delta}^{(k)})
\leq 
\sum_{j=0}^{k-1} L_{j,s}(\delta)
\Bigl[\Delta_{s,\delta}(\hat \theta, \tilde \theta)+ 2\frak{Q}_n(\Theta,\delta)\Bigr],
\end{align*} 
\begin{align*}
&
\Delta_{s,k\delta} (\hat \theta^{(k)}, \tilde \theta^{(k)})
\leq 
\sum_{j=0}^{k-1} L_{j,s}(\delta)\Bigl[\Delta_{s,\delta}(\hat \theta, \tilde \theta)+ 2\frak{Q}_n(\Theta,\delta)\Bigr]
+ 4k \frak{Q}_n(\Theta_{(k-1)\delta},\delta)
\end{align*}
and 
\begin{align*}
&
\Delta_{s,k\delta} (\hat \theta^{(k)}, \tilde \theta_{\delta}^{(k)})
\leq 
\sum_{j=0}^{k-1} L_{j,s}(\delta)\Bigl[\Delta_{s,\delta}(\hat \theta, \tilde \theta)+ 2\frak{Q}_n(\Theta,\delta)\Bigr]
+ 2k \frak{Q}_n(\Theta_{(k-1)\delta},\delta).
\end{align*}
\end{proposition}

\begin{proof}
Note that 
\begin{align}
\label{odin}
{\mathcal T}_{\delta}^k - \tilde {\mathcal T}_{\delta}^k = {\mathcal T}_{\delta}
({\mathcal T}_{\delta}^{k-1}-\tilde {\mathcal T}_{\delta}^{k-1}) + ({\mathcal T}_{\delta}-\tilde {\mathcal T}_{\delta}) \tilde {\mathcal T}_{\delta}^{k-1}.
\end{align}
Since operator ${\mathcal T}_{\delta}: L_{\infty}(\Theta)\mapsto L_{\infty}(\Theta)$ is a contraction,
we have 
\begin{align}
\label{dva}
\|{\mathcal T}_{\delta}({\mathcal T}_{\delta}^{k-1}-\tilde {\mathcal T}_{\delta}^{k-1})f\|_{L_{\infty}(\Theta)}
\leq \|({\mathcal T}_{\delta}^{k-1}-\tilde {\mathcal T}_{\delta}^{k-1})f\|_{L_{\infty}(\Theta)}, f\in L_{\infty}(E), k>1.
\end{align}
The following bound is also straightforward:
\begin{align}
\label{tri}
\Bigl\|({\mathcal T}_{\delta}-\tilde {\mathcal T}_{\delta}) \tilde {\mathcal T}_{\delta}^{k-1}\Bigr\|_{C^s(\Theta_{\delta k})\mapsto L_{\infty}(\Theta)}
\leq 
L_{k-1,s}(\delta)\Bigl\|{\mathcal T}_{\delta}-\tilde {\mathcal T}_{\delta}\Bigr\|_{C^s(\Theta_{\delta})\mapsto L_{\infty}(\Theta)}.
\end{align}
Using \eqref{T^k}, \eqref{T_delta^k}, \eqref{odin}, \eqref{dva} and \eqref{tri}, 
we get 
\begin{align}
\label{bd_k_k-1}
\nonumber
&
\sup_{\theta\in \Theta} \Bigl|{\mathbb E}_{\theta} f(\hat \theta_{\delta}^{(k)})- {\mathbb E}_{\theta} 
f(\tilde \theta_{\delta}^{(k)})\Bigr|
\\
&
\leq
\sup_{\theta\in \Theta} 
\Bigl|{\mathbb E}_{\theta} f(\hat \theta_{\delta}^{(k-1)})- {\mathbb E}_{\theta} f(\tilde \theta_{\delta}^{(k-1)})\Bigr|
+
L_{k-1,s}(\delta)\|f\|_{C^s(\Theta_{k\delta})}\Bigl\|{\mathcal T}_{\delta}-\tilde {\mathcal T}_{\delta}\Bigr\|_{C^s(\Theta_{\delta})\mapsto L_{\infty}(\Theta)}.
\end{align}
By Lemma \ref{bd_T-T_delta},
\begin{align*}
\Bigl\|{\mathcal T}_{\delta}-\tilde {\mathcal T}_{\delta}\Bigr\|_{C^s(\Theta_{\delta})\mapsto L_{\infty}(\Theta)}
= \Delta_{s,\delta}(\hat \theta_{\delta},\tilde \theta_{\delta})
\leq  \Delta_{s,\delta}(\hat \theta,\tilde \theta)+2\frak{Q}_n(\Theta,\delta).
\end{align*}
It then follows from \eqref{bd_k_k-1} that 
\begin{align*}
&
\sup_{\theta\in \Theta} \Bigl|{\mathbb E}_{\theta} f(\hat \theta_{\delta}^{(k)})- {\mathbb E}_{\theta} f(\tilde \theta_{\delta}^{(k)})\Bigr|
\leq
\sup_{\theta\in \Theta} 
\Bigl|{\mathbb E}_{\theta} f(\hat \theta_{\delta}^{(k-1)})- {\mathbb E}_{\theta} f(\tilde \theta_{\delta}^{(k-1)})\Bigr|
\\
&
+
L_{k-1,s}(\delta)\|f\|_{C^s(\Theta_{k\delta})}
\Bigl[\Delta_{s,\delta}(\hat \theta, \tilde \theta)+
2\frak{Q}_n(\Theta,\delta)\Bigr].
\end{align*}
By induction, this implies that 
\begin{align*}
&
\sup_{\theta\in \Theta} \Bigl|{\mathbb E}_{\theta} f(\hat \theta_{\delta}^{(k)})- {\mathbb E}_{\theta} f(\tilde \theta_{\delta}^{(k)})\Bigr|
\leq
\sum_{j=0}^{k-1} L_{j,s}(\delta)
\|f\|_{C^s(\Theta_{k\delta})}
\Bigl[\Delta_{s,\delta}(\hat \theta, \tilde \theta)+ 2\frak{Q}_n(\Theta,\delta)\Bigr],
\end{align*}
and the first bound follows.
It remains to combine it with the bounds of Lemma \ref{lm_delta_hat_tilde}
to complete the proof. 
\end{proof}

To control the norms of operator $\tilde {\mathcal T}_{\delta},$
we use Lemma \ref{bd_T_delta}. 
It implies that, for all $i=1,\dots, k$ 
\begin{align*}
\|\tilde {\mathcal T}_{\delta}\|_{C^s(\Theta_{(i+1)\delta})\mapsto C^s(\Theta_{i\delta})}
\lesssim_s
1+ \frac{{\mathbb E}\|\xi\|_{C^s(\Theta_{i\delta})}^s}{n^{s/2}} 
\end{align*}
and 
\begin{align*}
L_{k,s}(\delta) \lesssim_{s,k} 
\biggl(1+ \frac{{\mathbb E}\|\xi\|_{C^s(\Theta_{(k-1)\delta})}^s}{n^{s/2}}\biggr)^{k}.
\end{align*}
Substituting the last bound in the bounds of 
Proposition \ref{hat^k_tilde^k_hat_tilde} yields the claim of Theorem \ref{hat^k_tilde^k_hat_tilde_AAA}.

Bound \eqref{Delta_s,kdelta} of Theorem \ref{hat^k_tilde^k_hat_tilde_AAA} implies that 
\begin{align*}
&
\max_{0\leq j\leq k}\Bigl\|{\mathcal T}^j f-\tilde{\mathcal T}_{\delta}^j f\Bigr\|_{L_{\infty}(\Theta)}
\\
&
\lesssim_{s,k} 
\|f\|_{C^s(\Theta_{k\delta})}\biggl(1+ \frac{{\mathbb E}\|\xi\|_{C^s(\Theta_{(k-1)\delta})}^s}{n^{s/2}}\biggr)^{k-1}
\Bigl[\Delta_{s,\delta}(\hat \theta, \tilde \theta)+ \frak{Q}_n(\Theta_{(k-1)\delta},\delta)
\Bigr].
\end{align*}
Since 
\begin{align*}
{\mathcal B}^k f = ({\mathcal T}-{\mathcal I})^k f= \sum_{j=0}^k (-1)^{k-j} {k\choose j} {\mathcal T}^j f
\end{align*}
and
\begin{align*}
\tilde {\mathcal B}_{\delta}^k f = (\tilde{\mathcal T}_{\delta}-{\mathcal I})^k f=\sum_{j=0}^k (-1)^{k-j} {k\choose j} \tilde{\mathcal T}_{\delta}^j f,
\end{align*}
we easily get 
\begin{align*}
&
\max_{0\leq j\leq k}\Bigl\|{\mathcal B}^j f-\tilde{\mathcal B}_{\delta}^j f\Bigr\|_{L_{\infty}(\Theta)}
\\
&
\lesssim_{s,k} 
\|f\|_{C^s(\Theta_{k\delta})}\biggl(1+ \frac{{\mathbb E}\|\xi\|_{C^s(\Theta_{(k-1)\delta})}^s}{n^{s/2}}\biggr)^{k-1}
\Bigl[\Delta_{s,\delta}(\hat \theta, \tilde \theta)+ 
\frak{Q}_n(\Theta_{(k-1)\delta},\delta)\Bigr].
\end{align*}
Recalling that 
\begin{align*}
f_k(\theta)= \sum_{j=0}^k (-1)^j ({\mathcal B}^j f)(\theta)\ {\rm and}\ 
\tilde f_{\delta,k}(\theta)= \sum_{j=0}^k (-1)^j (\tilde {\mathcal B}_{\delta}^j f)(\theta),
\end{align*}
we can conclude that the bound of Theorem \ref{uniform_bd} also holds.
\end{proof}

\section{Concentration bounds for the approximating Gaussian model}
\label{Gaussian_conc}

In this section, we study concentration properties of ``estimator" $\tilde f_{\delta,k}(\tilde \theta_{\delta})$
in the approximating Gaussian model. 

Recall that $\psi\in \Psi$  is a convex function 
$\psi:{\mathbb R}\mapsto {\mathbb R}_+$
with $\psi(0)=0.$ 
It is also even, increasing on ${\mathbb R}_+,$
and satisfies the condition $c'u\leq \psi(u)\leq \psi_1(c''u), u\geq 0$ for some constants $c',c''>0,$ where $\psi_1(u):=e^{u}-1, u\geq 0.$  
Also recall the notation
$
\tilde \psi(u):= \frac{1}{\psi^{-1}(\frac{1}{u})}, u\geq 0.
$

Our main goal is to prove the following theorem. 

\begin{theorem}
\label{conc_main}
Let $s=k+1+\rho,$ $k\geq 1, \rho\in (0,1].$ Let $\delta>0$ and $f\in C^s (\Theta_{(k+3)\delta}).$
Suppose that, for a sufficiently small constant $c_1>0,$
\begin{align*}
{\frak d}_{\xi}(\Theta_{(k+2)\delta};s-1)\leq c_1 n.
\end{align*}
Then, for all $\psi\in \Psi,$ the following bound holds:
\begin{align*}
&
\sup_{\theta\in \Theta}
\Bigl\|\tilde f_{\delta,k}(\tilde \theta_{\delta})-f(\theta)-n^{-1/2}\langle f'(\theta), \xi(\theta)\rangle
\Bigr\|_{L_\psi({\mathbb P})}
\\
&
\lesssim_{s, \psi}  
\|f\|_{C^s(\Theta_{(k+3)\delta})}
\biggl[
\biggl(\sqrt{\frac{{\frak d}_{\xi}(\Theta_{(k+2)\delta};s-1)}{n}}\biggr)^s 
+
\frac{\|\Sigma\|_{L_{\infty}(E)}^{1/2}}{n^{1/2}} \sqrt{\frac{{\frak d}_{\xi}(\Theta_{(k+2)\delta};s-1)}{n}}
\\
&
\ \ \ \ \ \ \ \ \ \ \ \ \ \ \ \ \ \ \ \ \ \ \ \ \ \ + \sqrt{\frac{{\frak d}_{\xi}(\Theta_{(k+2)\delta};s-1)}{n}}\tilde \psi^{1/2}({\mathbb P}\{\|\xi\|_{L_{\infty}(E)}\geq \delta \sqrt{n}\})
\biggr].
\end{align*}
\end{theorem}

In the proof, we will use some concentration bounds for locally Lipschitz functions of Gaussian 
random variables that go back to Maurey and Pisier. 
Let $g:{\mathbb R}^N\mapsto {\mathbb R}$ be a locally Lipschitz function. Define its local Lipschitz constant as
\begin{align*}
(Lg)(x):= \inf_{U\ni x} \sup_{x_1,x_2\in U}\frac{|g(x_1)-g(x_2)|}{\|x_1-x_2\|_{\ell_2}},
\end{align*} 
where the infimum is taken over all neighborhoods $U$ of $x.$

Define
\begin{align*} 
\psi^{\sharp}(u):= {\mathbb E}\psi (uZ), u\in {\mathbb R},\ Z\sim N(0,1). 
\end{align*}
Note that $\psi^{\sharp}$ is convex even function and it is increasing on ${\mathbb R}_+.$

The following version of Maurey-Pisier concentration bound will be used below (see Section 5.2 in \cite{Koltchinskii_Zhilova_19}).

\begin{lemma}
	\label{Maurey_Pisier}	
	Let $Z^{(N)}$ be a standard normal r.v. in ${\mathbb R}^N$ and let $g:{\mathbb R}^N\mapsto {\mathbb R}$ be a locally Lipschitz function. Then 
	\begin{align*}
	\|g(Z^{(N)})-{\mathbb E} g(Z^{(N)})\|_{L_{\psi}({\mathbb P})}\leq \frac{\pi}{2}\|(Lg)(Z^{(N)})\|_{L_{\psi^{\sharp}}({\mathbb P})}.
	\end{align*}
\end{lemma}

We will also use the following simple lemmas.

\begin{lemma}
\label{lemma_psi_simple_1}	
For all r.v. $\eta_1, \eta_2,$
\begin{align*}
\|\eta_1 \eta_2\|_{L_{\psi}({\mathbb P})}\leq \|\eta_1^2\|_{L_{\psi}({\mathbb P})}^{1/2}
\|\eta_2^2\|_{L_{\psi}({\mathbb P})}^{1/2}.
\end{align*}
\end{lemma}

\begin{lemma}
\label{lemma_psi_simple_2}	
For all events $E,$
\begin{align*}
\|I_E\|_{L_{\psi}({\mathbb P})}= \tilde\psi({\mathbb P}(E)).
\end{align*}
\end{lemma}

We are now ready to present the proof of Theorem \ref{conc_main}.

\begin{proof}
We will use the following representation:
\begin{align}
\label{main_repr_conc}
&
\nonumber
\tilde f_{\delta,k}(\tilde \theta_{\delta})-f(\theta)
\\
&
\nonumber
=\tilde f_{\delta,k}(\tilde \theta_{\delta})-
{\mathbb E}_{\theta}\tilde f_{\delta,k}(\tilde \theta_{\delta})
+{\mathbb E}_{\theta}\tilde f_{\delta,k}(\tilde \theta_{\delta})- f(\theta)
\\
&
\nonumber
=f(\tilde \theta_{\delta})-
{\mathbb E}_{\theta}f(\tilde \theta_{\delta})
+
(\tilde f_{\delta,k}-f)(\tilde \theta_{\delta})-
{\mathbb E}_{\theta}(\tilde f_{\delta,k}-f)(\tilde \theta_{\delta})
+{\mathbb E}_{\theta}\tilde f_{\delta,k}(\tilde \theta_{\delta})- f(\theta)
\\
&
\nonumber
= n^{-1/2} \langle f'(\theta), \xi_{\delta}(\theta)-{\mathbb E}\xi_{\delta}(\theta)\rangle 
+ 
S_{f}(\theta, n^{-1/2}\xi_{\delta}(\theta))- {\mathbb E} S_{f}(\theta, n^{-1/2}\xi_{\delta}(\theta))
\\
&
\nonumber
\ \ \ + (\tilde f_{\delta,k}-f)(\tilde \theta_{\delta})-
{\mathbb E}_{\theta}(\tilde f_{\delta,k}-f)(\tilde \theta_{\delta})
+
{\mathbb E}_{\theta}\tilde f_{\delta,k}(\tilde \theta_{\delta})- f(\theta)
\\
&
\nonumber
= n^{-1/2} \langle f'(\theta), \xi(\theta)\rangle+ 
n^{-1/2} \langle f'(\theta), \xi_{\delta}(\theta)-\xi(\theta)-{\mathbb E}(\xi_{\delta}(\theta)-\xi(\theta))\rangle
\\
&
\nonumber
\ \ \ +
S_{f}(\theta, n^{-1/2}\xi_{\delta}(\theta))- {\mathbb E} S_{f}(\theta, n^{-1/2}\xi_{\delta}(\theta))
+ (\tilde f_{\delta,k}-f)(\tilde \theta_{\delta})-
{\mathbb E}_{\theta}(\tilde f_{\delta,k}-f)(\tilde \theta_{\delta})
\\
&
\ \ \ + {\mathbb E}_{\theta}\tilde f_{\delta,k}(\tilde \theta_{\delta})- f(\theta),
\end{align}
where, for a differentiable function $g,$ 
\begin{align*}
S_g(\theta,h) := g(\theta+h)-g(\theta)- \langle g'(\theta), h\rangle 
\end{align*} 
denotes the remainder of the first order Taylor expansion of $g.$

Note that, under the assumption ${\frak d}_{\xi}(\Theta_{k\delta};s-1)\leq c_1 n$ with a sufficiently 
small constant $c_1,$ we have 
\begin{align*}
4 s^{s+1} {\mathbb E} \biggl(1+ \frac{\|\xi\|_{C^{s-1}(\Theta_{k\delta})}}{\sqrt{n}}\biggr)^{s-1}
\frac{\|\xi\|_{C^{s-1}(\Theta_{k\delta})}}{\sqrt{n}} \leq 1/2
\end{align*}
and Lemma \ref{bd_on_tilde_f_kdelta_B} could be used. Indeed,
\begin{align}
\label{bd_xi_s-1_kdelta}
&
\nonumber
{\mathbb E} \biggl(1+ \frac{\|\xi\|_{C^{s-1}(\Theta_{k\delta})}}{\sqrt{n}}\biggr)^{s-1}
\frac{\|\xi\|_{C^{s-1}(\Theta_{k\delta})}}{\sqrt{n}}
\leq 
\\
&
\nonumber
{\mathbb E}^{1/2} \biggl(1+ \frac{\|\xi\|_{C^{s-1}(\Theta_{k\delta})}}{\sqrt{n}}\biggr)^{2(s-1)}
\frac{{\mathbb E}^{1/2}\|\xi\|_{C^{s-1}(\Theta_{k\delta})}^2}{\sqrt{n}}
\\
&
\nonumber
\lesssim_s 
\Bigl(1+ \frac{{\mathbb E}^{1/2}\|\xi\|_{C^{s-1}(\Theta_{k\delta})}^{2(s-1)}}{n^{(s-1)/2}}\Bigr)
\frac{{\mathbb E}^{1/2}\|\xi\|_{C^{s-1}(\Theta_{k\delta})}^2}{\sqrt{n}}
\\
&
\lesssim_s \Bigl(1+ \Bigl(\frac{{\frak d}_{\xi}(\Theta_{k\delta};s-1)}{n}\Bigr)^{(s-1)/2}\Bigr)\sqrt{\frac{{\frak d}_{\xi}(\Theta_{k\delta};s-1)}{n}},
\end{align}
which implies the claim. Note that we also used here standard bounds on the norms of Gaussian random vectors (that follow, for instance, from Gaussian concentration). 

{\it Step 1}.  To control the bias ${\mathbb E}_{\theta}\tilde f_{\delta,k}(\tilde \theta_{\delta})- f(\theta),$ 
we will use a slight modification of the bound of Theorem 3.2 in \cite{Koltchinskii_Zhilova_19}
(see also bound 4 of Proposition \ref{prop_summary}). It takes into account that 
$\xi(\theta)$ is assumed to be smooth only in a neighborhood of set $\Theta$
(not in the whole space $E$ and even not in the whole parameter set $T$). 

\begin{lemma}
The following bound holds for all $\theta\in \Theta:$
\begin{align*}
&
|{\mathbb E}_{\theta}\tilde f_{\delta,k}(\tilde \theta_{\delta})- f(\theta)|
\lesssim_s 
\|f\|_{C^s(\Theta_{(k+1)\delta})} 
\Bigl(\Bigl(\sqrt{\frac{{\frak d}_{\xi}(\Theta_{k\delta};s-1)}{n}}\Bigr)^s
\\
&
+
\sqrt{\frac{{\frak d}_{\xi}(\Theta_{k\delta};s-1)}{n}}\Bigr)^k\frac{\|\Sigma(\theta)\|^{1/2}}{n^{1/2}}
{\mathbb P}^{1/2}\{\|\xi\|_{L_{\infty}(E)}\geq \delta\sqrt{n}\}\Bigr).
\end{align*}
\end{lemma}

\begin{proof}
Define a random homotopy between $\theta$ and $\tilde \theta_{\delta}$ as follows:
\begin{align*}
H(\theta;t):= \theta+ \frac{t \xi_{\delta}(\theta)}{\sqrt{n}}, \theta \in E, t\in [0,1].
\end{align*}
We have 
\begin{align*}
&
(\tilde {\mathcal B}_{\delta}^{k+1} f)(\theta)= {\mathbb E}_{\theta} (\tilde{\mathcal B}_{\delta}^k f)(\tilde \theta_{\delta})
-
(\tilde{\mathcal B}_{\delta}^k f)(\theta)
\\
&
= {\mathbb E} \Bigl((\tilde{\mathcal B}_{\delta}^k f)(H(\theta;1))- (\tilde{\mathcal B}_{\delta}^k f)(H(\theta;0))\Bigr)
= {\mathbb E} \int_0^1 \Bigl\langle (\tilde{\mathcal B}_{\delta}^k f)'(H(\theta;t)), \dot H(\theta;t)\Bigr\rangle dt
\\
&
=\Bigl\langle (\tilde{\mathcal B}_{\delta}^k f)' (\theta), {\mathbb E}\int_{0}^1 \dot H(\theta;t)dt\Bigr \rangle
+  {\mathbb E} \int_0^1 \Bigl\langle (\tilde{\mathcal B}_{\delta}^k f)' (H(\theta;t))-(\tilde{\mathcal B}_{\delta}^k f)' (\theta), \dot H(\theta;t)\Bigr\rangle dt.
\end{align*}
Note that 
\begin{align*}
{\mathbb E}\int_{0}^1 \dot H(\theta;t)dt= {\mathbb E} \frac{\xi_{\delta}(\theta)}{\sqrt{n}}
= {\mathbb E} \frac{\xi_{\delta}(\theta)-\xi(\theta)}{\sqrt{n}},
\end{align*}
implying that 
\begin{align*}
&
\Bigl\langle (\tilde{\mathcal B}_{\delta}^k f)' (\theta), {\mathbb E}\int_{0}^1 \dot H(\theta;t)dt\Bigr\rangle
=
n^{-1/2}{\mathbb E}\Bigl\langle (\tilde{\mathcal B}_{\delta}^k f)' (\theta), \xi_{\delta}(\theta)-\xi(\theta)\Bigr\rangle
\\
&
= 
-n^{-1/2}{\mathbb E}\langle (\tilde{\mathcal B}_{\delta}^k f)' (\theta), \xi(\theta)\rangle
I(\|\xi\|_{L_{\infty}(E)}\geq \delta\sqrt{n}).
\end{align*}
Therefore,
\begin{align*}
&
\Bigl|\Bigl\langle(\tilde{\mathcal B}_{\delta}^k f)' (\theta), {\mathbb E}\int_{0}^1 \dot H(\theta;t)dt\Bigr\rangle\Bigr|
\leq n^{-1/2} {\mathbb E}^{1/2}\langle (\tilde{\mathcal B}_{\delta}^k f)' (\theta), \xi(\theta)\rangle^2
\ {\mathbb P}^{1/2}\{\|\xi\|_{L_{\infty}(E)}\geq \delta\sqrt{n}\}
\\
&
= n^{-1/2} \langle \Sigma(\theta)(\tilde{\mathcal B}_{\delta}^k f)' (\theta),
(\tilde{\mathcal B}_{\delta}^k f)' (\theta)\rangle^{1/2} 
{\mathbb P}^{1/2}\{\|\xi\|_{L_{\infty}(E)}\geq \delta\sqrt{n}\}
\\
&
\leq \|(\tilde{\mathcal B}_{\delta}^k f)'\|_{L_{\infty}(\Theta)}
\frac{\|\Sigma(\theta)\|^{1/2}}{n^{1/2}}
{\mathbb P}^{1/2}\{\|\xi\|_{L_{\infty}(E)}\geq \delta\sqrt{n}\}.
\end{align*}
On the other hand, using the fact that $H(\theta;t)\in \Theta_{\delta}, \theta\in \Theta, t\in [0,1],$
we get
\begin{align*}
&
 \Bigl|{\mathbb E} \int_0^1 \Bigl\langle (\tilde{\mathcal B}_{\delta}^k f)' (H(\theta;t))-(\tilde{\mathcal B}_{\delta}^k f)' (\theta), \dot H(\theta;t)\Bigr\rangle
 dt\Bigr|
\\
&
\leq \|(\tilde{\mathcal B}_{\delta}^k f)'\|_{{\rm Lip}_{\rho}(\Theta_{\delta})}
{\mathbb E}\int_0^1 \|H(\theta;t)-\theta\|^{\rho} \|\dot H(\theta;t)\|dt
\\
& 
 \leq \|(\tilde{\mathcal B}_{\delta}^k f)'\|_{{\rm Lip}_{\rho}(\Theta_{\delta})}
\frac{{\mathbb E}\|\xi_{\delta}(\theta)\|^{1+\rho}}{n^{(1+\rho)/2}} \int_{0}^1 t^{\rho}dt
\leq \|(\tilde{\mathcal B}_{\delta}^k f)'\|_{{\rm Lip}_{\rho}(\Theta_{\delta})}
\frac{{\mathbb E}\|\xi(\theta)\|^{1+\rho}}{n^{(1+\rho)/2}}. 
\end{align*}
It remains to use bound \eqref{bd_B_delta^j} with $j=k$ to get 
\begin{align*}
&
\|(\tilde{\mathcal B}_{\delta}^k f)'\|_{L_{\infty}(\Theta)} \vee \|(\tilde{\mathcal B}_{\delta}^k f)'\|_{{\rm Lip}_{\rho}(\Theta_{\delta})}
\leq \|\tilde{\mathcal B}_{\delta}^k f\|_{C^{1+\rho}(\Theta_{\delta})}
\\
&
\lesssim_s \biggl({\mathbb E} \biggl(1+ \frac{\|\xi\|_{C^{s-1}(\Theta_{k\delta})}}{\sqrt{n}}\biggr)^{s-1}
\frac{\|\xi\|_{C^{s-1}(\Theta_{k\delta})}}{\sqrt{n}}\biggr)^k \|f\|_{C^s(\Theta_{(k+1)\delta})}
\end{align*}
and to use bound \eqref{bd_xi_s-1_kdelta} and the assumption ${\frak d}_{\xi}(\Theta_{k\delta};s-1)\leq c_1 n$
in order to complete the proof.
\end{proof}

{\it Step 2}. Next, we will bound the term  $\langle f'(\theta), \xi_{\delta}(\theta)-\xi(\theta)-{\mathbb E}(\xi_{\delta}(\theta)-\xi(\theta))\rangle:$
\begin{align*}
&
\Bigl\|n^{-1/2} \langle f'(\theta), \xi_{\delta}(\theta)-\xi(\theta)-{\mathbb E}(\xi_{\delta}(\theta)-\xi(\theta))\rangle\Bigr\|_{L_{\psi}({\mathbb P})}
\\
&
\leq 
n^{-1/2}\Bigl\|\langle f'(\theta), \xi(\theta)\rangle
I(\|\xi\|_{L_{\infty}(E)}\geq \delta\sqrt{n})\Bigr\|_{L_{\psi}({\mathbb P})}
\\
&
+
n^{-1/2}\Bigl\|{\mathbb E}\langle f'(\theta), \xi(\theta)I(\|\xi\|_{L_{\infty}(E)}\geq \delta\sqrt{n})\rangle
\Bigr\|_{L_{\psi}({\mathbb P})}
\\
&
\leq 
n^{-1/2} \|\langle f'(\theta), \xi(\theta)\rangle^2\|_{L_{\psi}({\mathbb P})}^{1/2}
\|I(\|\xi\|_{L_{\infty}(E)}\geq \delta\sqrt{n})\|_{L_{\psi}({\mathbb P})}^{1/2}
\\
&
+
n^{-1/2} \|1\|_{L_{\psi}({\mathbb P})} \|\langle f'(\theta), \xi(\theta)\rangle^2\|_{L_1({\mathbb P})}^{1/2}
\|I(\|\xi\|_{L_{\infty}(E)}\geq \delta\sqrt{n})\|_{L_1({\mathbb P})}^{1/2}
\\
&
\lesssim_{\psi} n^{-1/2} \|\langle f'(\theta), \xi(\theta)\rangle^2\|_{L_{\psi}({\mathbb P})}^{1/2}
\|I(\|\xi\|_{L_{\infty}(E)}\geq \delta\sqrt{n})\|_{L_{\psi}({\mathbb P})}^{1/2}
\\
&
=
n^{-1/2}\|Z^2\|_{L_{\psi}({\mathbb P})}^{1/2} 
\langle \Sigma(\theta) f'(\theta), f'(\theta)\rangle^{1/2}\ \tilde \psi^{1/2} ({\mathbb P}\{\|\xi\|_{L_{\infty}(E)}\geq \delta\sqrt{n}\})
\\
&
\leq \|Z^2\|_{L_{\psi}({\mathbb P})}^{1/2} \|f'(\theta)\|\frac{\|\Sigma(\theta)\|^{1/2}}{n^{1/2}}
\tilde \psi^{1/2} ({\mathbb P}\{\|\xi\|_{L_{\infty}(E)}\geq \delta\sqrt{n}\})
\\
&
\lesssim_{\psi} \|f\|_{C^s(\Theta)} 
\frac{\|\Sigma(\theta)\|^{1/2}}{n^{1/2}}
\tilde \psi^{1/2} ({\mathbb P}\{\|\xi\|_{L_{\infty}(E)}\geq \delta\sqrt{n}\}),
\end{align*}
where $Z\sim N(0,1).$

 A more difficult part of the proof is to use Gaussian concentration inequalities to bound the terms 
$$S_{f}(\theta, n^{-1/2}\xi_{\delta}(\theta))- {\mathbb E} S_{f}(\theta,n^{-1/2}\xi_{\delta}(\theta))\ {\rm and}\ (\tilde f_{\delta,k}-f)(\tilde \theta_{\delta})-
{\mathbb E}_{\theta}(\tilde f_{\delta,k}-f)(\tilde \theta_{\delta}).$$
It will be done in steps 3-5. 

\vskip 1mm

{\it Step 3}. Note that 
\begin{align*}
&
S_{f}(\theta, n^{-1/2}\xi_{\delta}(\theta))
=
S_{f}(\theta, n^{-1/2}\xi(\theta))I(\|\xi\|_{L_{\infty}(E)}< \delta\sqrt{n}).
\end{align*}
We will prove concentration bounds for a ``smoothed version" of the above random 
variable.  
Namely, let $\varphi :{\mathbb R}\mapsto [0,1]$ be a non-increasing function 
such that $\varphi(u):=1, u\leq 1,$ $\varphi(u):=0, u\geq 2$ and 
$\varphi(u):=2-u, u\in (1,2).$ Clearly, $\varphi $ is Lipschitz with constant $1.$  
We can write 
\begin{align}
\label{represent_S_tilde_f} 
\nonumber
S_{f}(\theta, n^{-1/2}\xi_{\delta}(\theta)) 
&= 
S_{f}(\theta, n^{-1/2}\xi(\theta))\varphi\biggl(\frac{\|\xi(\theta)\|}{\delta\sqrt{n}}\biggr)
\\
&
+ S_{f}(\theta, n^{-1/2}\xi_{\delta}(\theta))- S_{f}(\theta, n^{-1/2}\xi(\theta))\varphi\biggl(\frac{\|\xi(\theta)\|}{\delta\sqrt{n}}\biggr).
\end{align} 
with $S_{f}(\theta, n^{-1/2}\xi(\theta))\varphi\biggl(\frac{\|\xi(\theta)\|}{\delta\sqrt{n}}\biggr)$ being a
``smooth part" and the rest being a remainder. We will first control the remainder. 
Observe that if $\|\xi\|_{L_{\infty}(E)}< \delta\sqrt{n},$ then 
\begin{align*}
S_{f}(\theta, n^{-1/2}\xi_{\delta}(\theta))- S_{f}(\theta, n^{-1/2}\xi(\theta))\varphi\biggl(\frac{\|\xi(\theta)\|}{\delta\sqrt{n}}\biggr)=0.
\end{align*}
Therefore, using lemmas \ref{lemma_psi_simple_1} and \ref{lemma_psi_simple_2}, we get 		
\begin{align}
\label{S-Svarphi}
&
\nonumber
\biggl\|S_{f}(\theta, n^{-1/2}\xi_{\delta}(\theta))- S_{f}(\theta, n^{-1/2}\xi(\theta))\varphi\biggl(\frac{\|\xi(\theta)\|}{\delta\sqrt{n}}\biggr)\biggr\|_{L_{\psi}({\mathbb P})}
\\
&
\nonumber
=
\biggl\|S_{f}(\theta, n^{-1/2}\xi(\theta))\varphi\biggl(\frac{\|\xi(\theta)\|}{\delta\sqrt{n}}\biggr) I(\|\xi\|_{L_{\infty}(E)}\geq  \delta\sqrt{n})\biggr\|_{L_{\psi}({\mathbb P})}
\\
&
\nonumber
\leq \biggl\|S_{f}^2(\theta, n^{-1/2}\xi(\theta))\varphi^2\biggl(\frac{\|\xi(\theta)\|}{\delta\sqrt{n}}\biggr)\biggr\|_{L_{\psi}({\mathbb P})}^{1/2}
\Bigl\|I(\|\xi\|_{L_{\infty}(E)}\geq  \delta\sqrt{n})\biggr\|_{L_{\psi}({\mathbb P})}^{1/2}
\\
&
=
\biggl\|S_{f}^2(\theta, n^{-1/2}\xi(\theta))\varphi^2\biggl(\frac{\|\xi(\theta)\|}{\delta\sqrt{n}}\biggr)\biggr\|_{L_{\psi}({\mathbb P})}^{1/2} 
\tilde \psi^{1/2}({\mathbb P}\{\|\xi\|_{L_{\infty}(E)}\geq \delta\sqrt{n}\}).
\end{align}
Note that  $\varphi\biggl(\frac{\|\xi(\theta)\|}{\delta\sqrt{n}}\biggr)=0$ when $\|\xi(\theta)\|>2\delta\sqrt{n}.$
So, if $f$ is a Lipschitz function in $\Theta_{2\delta},$ we have
\begin{align*}
\biggl|S_{f}(\theta, n^{-1/2}\xi(\theta))\varphi\biggl(\frac{\|\xi(\theta)\|}{\delta\sqrt{n}}\biggr)\biggr| 
\leq 2\|f\|_{{\rm Lip}(\Theta_{2\delta})} n^{-1/2}\|\xi(\theta)\|.
\end{align*}
Therefore,
\begin{align*}
&
\biggl\|S_{f}^2(\theta, n^{-1/2}\xi(\theta))\varphi^2\biggl(\frac{\|\xi(\theta)\|}{\delta\sqrt{n}}\biggr)\biggr\|_{L_{\psi}({\mathbb P})}^{1/2} 
\leq  
2\|f\|_{{\rm Lip}(\Theta_{2\delta})}n^{-1/2}\Bigl\|\|\xi(\theta)\|^2\Bigr\|_{L_{\psi}({\mathbb P})}^{1/2}.
\end{align*}
Since $\psi(u)\leq \psi_1(cu), u\geq 0$ for some $c>0,$ we have 
\begin{align*}
\Bigl\|\|\xi(\theta)\|^2\Bigr\|_{L_{\psi}({\mathbb P})}^{1/2}
\lesssim_{\psi} \Bigl\|\|\xi(\theta)\|^2\Bigr\|_{L_{\psi_1}({\mathbb P})}^{1/2}
\leq \Bigl\|\|\xi(\theta)\|\Bigr\|_{L_{\psi_2}({\mathbb P})}.
\end{align*}
By a standard application of Gaussian concentration inequality, 
\begin{align}
\label{psi_2_on_xi}
&
\nonumber
\Bigl\|\|\xi(\theta)\|\Bigr\|_{L_{\psi_2}({\mathbb P})}\leq 
\Bigl\|\|\xi(\theta)\|-{\mathbb E}\|\xi(\theta)\|\Bigr\|_{L_{\psi_2}({\mathbb P})}
+{\mathbb E}\|\xi(\theta)\|\|1\|_{L_{\psi_2}({\mathbb P})}
\\
&
\lesssim \|\Sigma(\theta)\|^{1/2} + {\mathbb E}\|\xi(\theta)\| 
\lesssim {\mathbb E}^{1/2}\|\xi(\theta)\|^2.
\end{align}
As a result, 
we have
\begin{align*}
&
\biggl\|S_{f}^2(\theta, n^{-1/2}\xi(\theta))\varphi^2\biggl(\frac{\|\xi(\theta)\|}{\delta\sqrt{n}}\biggr)\biggr\|_{L_{\psi}({\mathbb P})}^{1/2} 
\lesssim 
\|f\|_{C^s(\Theta_{2\delta})}\frac{{\mathbb E}^{1/2}\|\xi(\theta)\|^2}{n^{1/2}}
\end{align*}
and plugging the last bound into \eqref{S-Svarphi}, we get 
\begin{align*}
&
\biggl\|S_{f}(\theta, n^{-1/2}\xi_{\delta}(\theta))- S_{f}(\theta, n^{-1/2}\xi(\theta))\varphi\biggl(\frac{\|\xi(\theta)\|}{\delta\sqrt{n}}\biggr)\biggr\|_{L_{\psi}({\mathbb P})}
\\
&
\lesssim_{\psi}  
\|f\|_{C^s(\Theta_{2\delta})}\frac{{\mathbb E}^{1/2}\|\xi(\theta)\|^2}{n^{1/2}}
\tilde \psi^{1/2}({\mathbb P}\{\|\xi\|_{L_{\infty}(E)}\geq \delta\sqrt{n}\}\}).
\end{align*}
This also implies that 
\begin{align}
\label{conc_bd_A} 
&
\nonumber
\biggl\|
S_{f}(\theta, n^{-1/2}\xi_{\delta}(\theta))- S_{f}(\theta, n^{-1/2}\xi(\theta))\varphi\biggl(\frac{\|\xi(\theta)\|}{\delta\sqrt{n}}\biggr)
\\
&
\nonumber
-{\mathbb E}\biggl(S_{f}(\theta, n^{-1/2}\xi_{\delta}(\theta))- S_{f}(\theta, n^{-1/2}\xi(\theta))\varphi\biggl(\frac{\|\xi(\theta)\|}{\delta\sqrt{n}}\biggr)\biggr)
\biggr\|_{L_{\psi}({\mathbb P})}
\\
&
\lesssim_{\psi}  
\|f\|_{C^s(\Theta_{2\delta})}\frac{{\mathbb E}^{1/2}\|\xi(\theta)\|^2}{n^{1/2}}  
\tilde \psi^{1/2}({\mathbb P}\{\|\xi\|_{L_{\infty}(E)}\geq \delta\sqrt{n}\}).
\end{align}

{\it Step 4.} Now we are ready to prove a concentration bound for 
r.v. 
$$
S_{f}(\theta, n^{-1/2}\xi_{\delta}(\theta))- 
{\mathbb E} S_{f}(\theta, n^{-1/2}\xi_{\delta}(\theta)).
$$

\begin{lemma}
The following bound holds:
\begin{align*}
&
\Bigl\|S_{f}(\theta, n^{-1/2}\xi_{\delta}(\theta))- 
{\mathbb E} S_{f}(\theta, n^{-1/2}\xi_{\delta}(\theta))\Bigr\|_{L_{\psi}({\mathbb P})}
\\
&
\lesssim_{s,\psi}
\|f\|_{C^s(\Theta_{3\delta})}
\frac{{\mathbb E}^{1/2}\|\xi(\theta)\|^2}{n^{1/2}}
\biggl(\frac{\|\Sigma(\theta)\|^{1/2}}{n^{1/2}}+
\tilde \psi^{1/2}({\mathbb P}\{\|\xi\|_{L_{\infty}(E)}\geq \delta\sqrt{n}\})\biggr).
\end{align*}
\end{lemma}

\begin{proof} 
We need some elementary bounds on the remainder of the first order Taylor expansion
$
S_g(\theta;h)= g(\theta+h)-g(\theta)-g^{\prime}(\theta)(h)
$
that will be used in the proof.

\begin{lemma}
\label{Taylor_remaind}
For any function $g\in C^{1+\rho}(\Theta_{\delta}),$ $\rho\in (0,1],$
the following bounds hold for all $\theta\in \Theta$ and all $h,h'$ with 
$\|h\|<\delta, \|h'\|<\delta:$
\begin{align*}
|S_g(\theta;h)| \lesssim \|g\|_{C^{1+\rho}(\Theta_{\delta})}\|h\|^{1+\rho}
\end{align*}
and 
\begin{align*}
|S_g(\theta;h)-S_g(\theta;h')|\lesssim \|g\|_{C^{1+\rho}(\Theta_{\delta})}(\|h\|^{\rho}\vee \|h'\|^{\rho})\|h-h'\|.
\end{align*}
\end{lemma}

In view of bound \eqref{conc_bd_A}, it is enough to control r.v.
\begin{align*}
S_{f}(\theta, n^{-1/2}\xi(\theta))\varphi\biggl(\frac{\|\xi(\theta)\|}{\delta\sqrt{n}}\biggr)
- {\mathbb E}S_{f}(\theta, n^{-1/2}\xi(\theta))\varphi\biggl(\frac{\|\xi(\theta)\|}{\delta\sqrt{n}}\biggr),
\end{align*}
which will rely on Gaussian concentration.
Denoting
$$
\bar g(h):= S_{f}(\theta, h) \varphi\biggl(\frac{\|h\|}{\delta}\biggr),
$$
we will bound $\|\bar g(n^{-1/2} \xi(\theta))-{\mathbb E}\bar g(n^{-1/2} \xi(\theta))\|_{L_\psi({\mathbb P})}.$
By Lemma \ref{Taylor_remaind} with $\rho=1$ and $\Theta_{2\delta}$ instead of $\Theta_{\delta},$ 
\begin{align*}
|S_{f}(\theta, h)|\lesssim \|f\|_{C^{2}(\Theta_{2\delta})}\|h\|^{2}
\end{align*}
and
\begin{align*}
|S_{f}(\theta, h)-S_{f}(\theta, h')|
\lesssim \|f\|_{C^{2}(\Theta_{2\delta})}(\|h\|\vee \|h'\|) \|h-h'\|
\end{align*}
for all $\theta\in \Theta$ an all $h,h'$ with $\|h\|<2\delta, \|h'\|<2\delta.$
Since, in addition, the function $h\mapsto \varphi\Bigl(\frac{\|h\|}{\delta}\Bigr)$
is Lipschitz with constant $\delta^{-1}$ and bounded by $1,$ we easily get that, for the same $\theta, h, h',$ 
\begin{align*}
&
|\bar g(h)-\bar g(h')|
\\
&
\lesssim \|f\|_{C^{2}(\Theta_{2\delta})}(\|h\|\vee \|h'\|) \|h-h'\| +
\|f\|_{C^{2}(\Theta_{2\delta})}\|h\|^{2}\frac{1}{\delta}\|h-h'\| 
\\
&
\lesssim  \|f\|_{C^{2}(\Theta_{2\delta})}(\|h\|\vee \|h'\|) \|h-h'\|.
\end{align*}
Obviously, the same bound with $\|f\|_{C^{2}(\Theta_{3\delta})}$
instead of $\|f\|_{C^{2}(\Theta_{2\delta})}$ holds for all $h,h'$ with 
$\|h\|<3\delta, \|h'\|<3\delta.$ Recall that $\bar g(h)=0$ for $\|h\|\geq 2\delta,$ so the bound trivially holds for 
all $\|h\|\geq 2\delta, \|h'\|\geq 2\delta.$ 
Thus, it remains to consider the case when $\|h\|<2\delta,$ $\|h'\|\geq 3\delta$ (and, hence, $\|h-h'\|>\delta$).
In this case,
\begin{align*}
&
|\bar g(h)-\bar g(h')| = |\bar g(h)| \lesssim \|f\|_{C^{2}(\Theta_{2\delta})}\|h\|^{2}
\\
&
\lesssim \|f\|_{C^{2}(\Theta_{2\delta})}\|h\| \delta
\lesssim \|f\|_{C^{2}(\Theta_{2\delta})}\|h\|\|h-h'\|.
\end{align*}
This proves the bound 
\begin{align}
\label{Lipschitz_on_barg}
|\bar g(h)-\bar g(h')| \lesssim \|f\|_{C^{2}(\Theta_{3\delta})}(\|h\|\vee \|h'\|) \|h-h'\|
\end{align}
for all $\theta\in \Theta, h,h'\in E.$ Note also that 
\begin{align*} 
\|\bar g\|_{L_{\infty}(E)} \lesssim \|f\|_{C^{2}(\Theta_{2\delta})}\delta^{2}.
\end{align*}

It is well known that Gaussian r.v. $\xi(\theta)$ in a separable Banach space $E$ could be represented as 
$$\xi(\theta)= \sum_{k\geq 1} Z_k x_k(\theta),$$
where the series converges a.s. in $E,$ $Z_k, k\geq 1$ are i.i.d. standard normal random variables, 
$x_k(\theta)\in E, k\geq 1$ and $\sum_{k\geq 1}\|x_k(\theta)\|^2<\infty.$ 
Denote $\xi^{(N)}(\theta)= \sum_{k=1}^N Z_k x_k(\theta).$ Since $\|\xi^{(N)}(\theta)-\xi(\theta)\|\to 0$
as $N\to\infty$ a.s. and $\bar g$ is a uniformly bounded continuous function on $E,$ we have 
\begin{align}
\label{approx_N_inf}
\Bigl\|\bar g(n^{-1/2}\xi(\theta))- {\mathbb E}\bar g(n^{-1/2}\xi(\theta))- (\bar g(n^{-1/2}\xi^{(N)}(\theta))- {\mathbb E}\bar g(n^{-1/2}\xi^{(N)}(\theta)))\Bigr\|_{L_{\psi}({\mathbb P})}\to 0
\end{align}
as $N\to\infty.$ Therefore, to control 
$\|\bar g(n^{-1/2} \xi(\theta))-{\mathbb E}\bar g(n^{-1/2} \xi(\theta))\|_{L_{\psi}({\mathbb P})},$
it would be enough to prove an upper bound 
on 
$$
\|\bar g(n^{-1/2} \xi^{(N)}(\theta))-{\mathbb E}\bar g(n^{-1/2} \xi^{(N)}(\theta))\|_{L_{\psi}({\mathbb P})}
$$
that holds uniformly in $N.$  

To this end, let $Z^{(N)}=(Z_1,\dots, Z_N)$ and denote 
\begin{align}
\label{def_g_z}
g(z):= \bar g\Bigl(n^{-1/2}\sum_{k=1}^N z_k x_k(\theta)\Bigr), z=(z_1,\dots, z_N)\in {\mathbb R}^N,
\end{align}
so that $g(Z^{(N)})= \bar g(n^{-1/2}\xi^{(N)}(\theta)).$
Let 
$$
(Lg)(z) := \inf_{U\ni z} \sup_{z_1,z_2\in U} \frac{|g(z_1)-g(z_2)|}{\|z_1-z_2\|_{\ell_2}},
$$
be the local Lipschitz constant of function $g$
(the infimum is taken over all neighborhoods $U$ of $z\in {\mathbb R}^N$).
It follows from \eqref{Lipschitz_on_barg} that the following bound holds for the local Lipschitz 
constant of function $\bar g:$
\begin{align*}
(L\bar g)(h)
\lesssim \|f\|_{C^{2}(\Theta_{3\delta})}\|h\|, h\in E. 
\end{align*}
Now note that
\begin{align*}
(Lg)(Z^{(N)}) \leq (L\bar g)(n^{-1/2}\xi^{(N)}(\theta)) n^{-1/2}\inf_{U\ni Z^{(N)}} \sup_{z_1,z_2\in U} \frac{\Bigl\|\sum_{k=1}^N (z_{1,k}-z_{2,k}) x_k(\theta)\Bigr\|}{\|z_1-z_2\|_{\ell_2}}
\end{align*}
and
\begin{align*}
&
\Bigl\|\sum_{k=1}^N (z_{1,k}-z_{2,k}) x_k(\theta)\Bigr\|= \sup_{\|u\|\leq 1}
\Bigl\langle \sum_{k=1}^N (z_{1,k}-z_{2,k}) x_k(\theta),u \Bigr\rangle
\\
&
\leq \sup_{\|u\|\leq 1}\biggl(\sum_{k=1}^N \langle x_k(\theta),u\rangle^2\biggr)^{1/2}
\|z_1-z_2\|_{\ell_2} \leq \|\Sigma(\theta)\|^{1/2} \|z_1-z_2\|_{\ell_2}.
\end{align*}
Thus, we get 
\begin{align*}
&
(Lg)(Z^{(N)}) \leq (L\bar g)(n^{-1/2}\xi^{(N)}(\theta))\frac{\|\Sigma(\theta)\|^{1/2}}{n^{1/2}} 
\\
&
\lesssim \|f\|_{C^{2}(\Theta_{3\delta})}\frac{\|\Sigma(\theta)\|^{1/2}}{n^{1/2}}
\frac{\|\xi^{(N)}(\theta)\|}{n^{1/2}}.
\end{align*}
Therefore, by Lemma \ref{Maurey_Pisier},
\begin{align}
\label{conc_for_bar_g}
&
\nonumber
\Bigl\|\bar g(n^{-1/2}\xi^{(N)}(\theta))- {\mathbb E}\bar g(n^{-1/2}\xi^{(N)}(\theta))\Bigr\|_{L_{\psi}({\mathbb P})}=
\|g(Z^{(N)})-{\mathbb E}g(Z^{(N)})\|_{L_{\psi}({\mathbb P})}
\\
&
\lesssim 
\|f\|_{C^2(\Theta_{3\delta})}
\frac{\|\Sigma(\theta)\|^{1/2}}{n^{1/2}}\Bigl\|\frac{\|\xi^{(N)}(\theta)\|}{n^{1/2}}
\Bigr\|_{L_{\psi^{\sharp}}({\mathbb P})}.
\end{align}
To control 
$
\|\|\xi^{(N)}(\theta)\|\|_{L_{\psi^{\sharp}}({\mathbb P})},
$
note that, for $\psi\in \Psi,$ we have 
$
\|\|\xi^{(N)}(\theta)\|\|_{L_{\psi^{\sharp}}({\mathbb P})} \lesssim_{\psi} 
\|\|\xi^{(N)}(\theta)\|\|_{L_{\psi_2}({\mathbb P})}, 
$
and, similarly to bound \eqref{psi_2_on_xi}, we get 
\begin{align*}
\|\|\xi^{(N)}(\theta)\|\|_{L_{\psi_2}({\mathbb P})}\lesssim \|\|\xi^{(N)}(\theta)\|\|_{L_2({\mathbb P})}
\lesssim {\mathbb E} \|\xi^{(N)}(\theta)\|.
\end{align*}
Moreover, 
\begin{align*}
&
{\mathbb E} \|\xi^{(N)}(\theta)\| = {\mathbb E}\Bigl\|\sum_{k=1}^N Z_k x_k(\theta)\Bigr\|
= {\mathbb E}\Bigl\|\sum_{k=1}^N Z_k x_k(\theta)+ {\mathbb E}\sum_{k\geq N+1} Z_k x_k(\theta)\Bigr\|
\\
&
\leq {\mathbb E}\Bigl\|\sum_{k\geq 1}Z_k x_k(\theta)\Bigr\|= {\mathbb E} \|\xi(\theta)\|.
\end{align*}
Therefore,
\begin{align*}
\|\|\xi^{(N)}(\theta)\|\|_{L_{\psi^{\sharp}}({\mathbb P})} \lesssim_{\psi} {\mathbb E} \|\xi(\theta)\|
\leq  {\mathbb E}^{1/2}\|\xi(\theta)\|^2.
\end{align*}
Substituting the last bound in \eqref{conc_for_bar_g} and passing to the limit as $N\to\infty$ yields
\begin{align}
\label{conc_bd_B}
&
\nonumber
\biggl\|S_{f}(\theta, n^{-1/2}\xi(\theta))\varphi\biggl(\frac{\|\xi(\theta)\|}{\delta\sqrt{n}}\biggr)
- {\mathbb E}S_{f}(\theta, n^{-1/2}\xi(\theta))\varphi\biggl(\frac{\|\xi(\theta)\|}{\delta\sqrt{n}}\biggr)\biggr\|_{L_{\psi}({\mathbb P})}
\\
&
\nonumber
= \Bigl\|\bar g(n^{-1/2}\xi(\theta))-{\mathbb E} \bar g(n^{-1/2}\xi(\theta))\Bigr\|_{L_{\psi}({\mathbb P})}
\\
&
\lesssim_{s,\psi} 
\|f\|_{C^s(\Theta_{3\delta})}
\frac{\|\Sigma(\theta)\|^{1/2}}{n^{1/2}}\frac{{\mathbb E}^{1/2}\|\xi(\theta)\|^2}{n^{1/2}}.
\end{align}
It follows from \eqref{represent_S_tilde_f}, \eqref{conc_bd_A} and \eqref{conc_bd_B}
that  
\begin{align*}
&
\Bigl\|S_{f}(\theta, n^{-1/2}\xi_{\delta}(\theta))- 
{\mathbb E} S_{f}(\theta, n^{-1/2}\xi_{\delta}(\theta))\Bigr\|_{L_{\psi}({\mathbb P})}
\\
&
\lesssim_{s,\psi}
\|f\|_{C^s(\Theta_{3\delta})}
\frac{\|\Sigma(\theta)\|^{1/2}}{n^{1/2}}\frac{{\mathbb E}^{1/2}\|\xi(\theta)\|^2}{n^{1/2}}
\\
&
+
\|f\|_{C^s(\Theta_{2\delta})} \frac{{\mathbb E}^{1/2}\|\xi(\theta)\|^2}{n^{1/2}}
\tilde \psi^{1/2}({\mathbb P}\{\|\xi\|_{L_{\infty}(E)}\geq \delta\sqrt{n}\}),
\end{align*}
completing the proof.
\end{proof}

{\it Step 5}. We will now sketch the proof of concentration bound for 
\begin{align*}
&
(\tilde f_{\delta,k}-f)(\tilde \theta_{\delta})-
{\mathbb E}_{\theta}(\tilde f_{\delta,k}-f)(\tilde \theta_{\delta})
\\
&
=
(\tilde f_{\delta,k}-f)(\tilde \theta_{\delta})-(\tilde f_{\delta,k}-f)(\theta)
-{\mathbb E}_{\theta}((\tilde f_{\delta,k}-f)(\tilde \theta_{\delta})- (\tilde f_{\delta,k}-f)(\theta))
\end{align*}
(the argument is similar to concentration bounds of steps 3, 4). 
Note that, under the assumption that ${\frak d}_{\xi}(\Theta_{(k+2)\delta};s-1)\leq c_1 n$ with a sufficiently 
small constant $c_1,$ bounds \eqref{bd_on_tilde_f_kdelta_B_two} with $\Theta_{3\delta}$ instead of $\Theta$ and 
\eqref{bd_xi_s-1_kdelta} with $\Theta_{(k+2)\delta}$ instead of $\Theta_{k \delta}$
imply that 
\begin{align*}
\|\tilde f_{\delta,k}-f\|_{C^{1+\rho}(\Theta_{3\delta})}
&\lesssim_s 
\|f\|_{C^s(\Theta_{(k+3)\delta})}{\mathbb E} \biggl(1+ \frac{\|\xi\|_{C^{s-1}(\Theta_{(k+2)\delta})}}{\sqrt{n}}\biggr)^{s-1}
\frac{\|\xi\|_{C^{s-1}(\Theta_{(k+2)\delta})}}{\sqrt{n}}
\\
&
\lesssim_s \|f\|_{C^s(\Theta_{(k+3)\delta})}\sqrt{\frac{{\frak d}_{\xi}(\Theta_{(k+2)\delta};s-1)}{n}}.
\end{align*}
It follows that 
\begin{align}
\label{L_infty_tilde_f_delata_k}
\|\tilde f_{\delta,k}-f\|_{L_{\infty}(\Theta_{3\delta})} \lesssim_s \|f\|_{C^s(\Theta_{(k+3)\delta})}\sqrt{\frac{{\frak d}_{\xi}(\Theta_{(k+2)\delta};s-1)}{n}}
\end{align}
and 
\begin{align}
\label{Lip_tilde_f_delata_k}
\|\tilde f_{\delta,k}-f\|_{{\rm Lip}(\Theta_{3\delta})} \lesssim_s \|f\|_{C^s(\Theta_{(k+3)\delta})}\sqrt{\frac{{\frak d}_{\xi}(\Theta_{(k+2)\delta};s-1)}{n}}.
\end{align}

Denote 
\begin{align*}
\lambda (h):= (\tilde f_{\delta,k}-f)(\theta+h)- (\tilde f_{\delta,k}-f)(\theta).
\end{align*}
Then, \eqref{L_infty_tilde_f_delata_k} and \eqref{Lip_tilde_f_delata_k} imply 
\begin{align}
\label{L_infty_lambda}
\|\lambda\|_{L_{\infty}(U(3\delta))}
\lesssim_s \|f\|_{C^s(\Theta_{(k+3)\delta})}\sqrt{\frac{{\frak d}_{\xi}(\Theta_{(k+2)\delta};s-1)}{n}}
\end{align}
and 
\begin{align}
\label{Lip_lambda}
\|\lambda\|_{{\rm Lip}(U(3\delta))} \lesssim_s \|f\|_{C^s(\Theta_{(k+3)\delta})}\sqrt{\frac{{\frak d}_{\xi}(\Theta_{(k+2)\delta};s-1)}{n}}.
\end{align}

Similarly to \eqref{represent_S_tilde_f} in Step 3, we can write 
\begin{align}
\label{represent_tilde f_delta,k-f} 
\nonumber
&
(\tilde f_{\delta,k}-f)(\tilde \theta_{\delta})-(\tilde f_{\delta,k}-f)(\theta)
=  \lambda (n^{-1/2}\xi_{\delta}(\theta))
\\
&
=\lambda(n^{-1/2} \xi(\theta))
\varphi\biggl(\frac{\|\xi(\theta)\|}{\delta\sqrt{n}}\biggr)
+\lambda(n^{-1/2} \xi_{\delta}(\theta))- \lambda(n^{-1/2} \xi(\theta))\varphi\biggl(\frac{\|\xi(\theta)\|}{\delta\sqrt{n}}\biggr).
\end{align} 
Using bound \eqref{L_infty_lambda} and arguing as in the proof of \eqref{conc_bd_A} in Step 3, we can show that 
\begin{align}
\label{lambda_conc_A}
&
\nonumber
\biggl\|\lambda(n^{-1/2} \xi_{\delta}(\theta))- \lambda(n^{-1/2} \xi(\theta))\varphi\biggl(\frac{\|\xi(\theta)\|}{\delta\sqrt{n}}\biggr)
\\
&
\nonumber
-{\mathbb E}\Biggl(\lambda(n^{-1/2} \xi_{\delta}(\theta))- \lambda(n^{-1/2} \xi(\theta))\varphi\biggl(\frac{\|\xi(\theta)\|}{\delta\sqrt{n}}\biggr)\biggr)\biggr\|_{L_{\psi}({\mathbb P})}
\\
&
\lesssim_{\psi}
\|f\|_{C^s(\Theta_{(k+3)\delta})}\sqrt{\frac{{\frak d}_{\xi}(\Theta_{(k+2)\delta};s-1)}{n}}
\tilde \psi({\mathbb P}\{\|\xi\|_{L_{\infty}(E)}\geq \delta\sqrt{n}\}).
\end{align}

To control 
$$\lambda(n^{-1/2} \xi(\theta))
\varphi\biggl(\frac{\|\xi(\theta)\|}{\delta\sqrt{n}}\biggr)- {\mathbb E}\lambda(n^{-1/2} \xi(\theta))
\varphi\biggl(\frac{\|\xi(\theta)\|}{\delta\sqrt{n}}\biggr),
$$
we argue as in the proof of \eqref{conc_bd_B} in Step 4. Namely, we define 
$\bar g(h):=\lambda(h)
\varphi\Bigl(\frac{\|h\|}{\delta}\Bigr).$
Using bound \eqref{Lip_lambda}, it is easy to show that 
\begin{align*}
|\bar g(h)-\bar g(h')| \lesssim_s \|f\|_{C^s(\Theta_{(k+3)\delta})}\sqrt{\frac{{\frak d}_{\xi}(\Theta_{(k+2)\delta};s-1)}{n}}\|h-h'\|
\end{align*}
for all $h,h'\in E.$
We then follow the rest of the argument of Step 4, introducing function $g(z)$ defined by \eqref{def_g_z} 
so that $g(Z^{(N)})= \bar g(n^{-1/2}\xi^{(N)}(\theta)).$ This argument yields the following bound on the local Lipschitz 
constant of function $g:$
\begin{align*}
(Lg)(Z^{(N)}) \lesssim_s \|f\|_{C^s(\Theta_{(k+3)\delta})}\sqrt{\frac{{\frak d}_{\xi}(\Theta_{(k+2)\delta};s-1)}{n}}\frac{\|\Sigma(\theta)\|^{1/2}}{n^{1/2}}.
\end{align*}
As a consequence, we get the concentration bound 
\begin{align*}
&
\Bigl\|\bar g(n^{-1/2}\xi^{(N)}(\theta))- {\mathbb E}\bar g(n^{-1/2}\xi^{(N)}(\theta))\Bigr\|_{L_{\psi}({\mathbb P})}=
\|g(Z^{(N)})-{\mathbb E}g(Z^{(N)})\|_{L_{\psi}({\mathbb P})}
\\
&
\lesssim_s \|f\|_{C^s(\Theta_{(k+3)\delta})}\sqrt{\frac{{\frak d}_{\xi}(\Theta_{(k+2)\delta};s-1)}{n}}\frac{\|\Sigma(\theta)\|^{1/2}}{n^{1/2}},
\end{align*}
and, passing to the limit as $N\to\infty,$ 
\begin{align}
\label{lambda_conc_B}
&
\nonumber
\biggl\|\lambda(n^{-1/2} \xi(\theta))
\varphi\biggl(\frac{\|\xi(\theta)\|}{\delta\sqrt{n}}\biggr)- {\mathbb E}\lambda(n^{-1/2} \xi(\theta))
\varphi\biggl(\frac{\|\xi(\theta)\|}{\delta\sqrt{n}}\biggr)
\biggr\|_{L_{\psi}({\mathbb P})} 
\\
&
\lesssim_{s,\psi} 
\|f\|_{C^s(\Theta_{(k+3)\delta})}\sqrt{\frac{{\frak d}_{\xi}(\Theta_{(k+2)\delta};s-1)}{n}}\frac{\|\Sigma(\theta)\|^{1/2}}{n^{1/2}}.
\end{align}
Combining 
\eqref{represent_tilde f_delta,k-f}, \eqref{lambda_conc_A} and \eqref{lambda_conc_B} yields
\begin{align}
\label{lambda_conc_fin}
&
\nonumber
\Bigl\|(\tilde f_{\delta,k}-f)(\tilde \theta_{\delta})-{\mathbb E}_{\theta}(\tilde f_{\delta,k}-f)(\tilde \theta_{\delta})\Bigr\|_{L_{\psi}({\mathbb P})}
= \Bigl\|\lambda (n^{-1/2}\xi_{\delta}(\theta))- {\mathbb E}\lambda (n^{-1/2}\xi_{\delta}(\theta))\Bigr\|_{L_{\psi}({\mathbb P})}
\\
&
\lesssim_{s,\psi} 
\|f\|_{C^s(\Theta_{(k+3)\delta})}\sqrt{\frac{{\frak d}_{\xi}(\Theta_{(k+2)\delta};s-1)}{n}}
\biggl(\frac{\|\Sigma(\theta)\|^{1/2}}{n^{1/2}}+
\tilde \psi({\mathbb P}\{\|\xi\|_{L_{\infty}(E)}\geq \delta\sqrt{n}\})\biggr).
\end{align}

{\it Step 6}. 
It remains to use representation \eqref{main_repr_conc} and 
combine the bounds obtained in steps 1-5 to get 
\begin{align}
\label{conc_konec}
&
\nonumber
\sup_{\theta\in \Theta}
\Bigl\|\tilde f_{\delta,k}(\tilde \theta_{\delta})-f(\theta)-n^{-1/2}\langle f'(\theta), \xi(\theta)\rangle
\Bigr\|_{L_\psi({\mathbb P})}
\\
&
\nonumber
\lesssim_{s,\psi} 
\|f\|_{C^s(\Theta_{(k+1)\delta})} 
\Bigl(\Bigl(\sqrt{\frac{{\frak d}_{\xi}(\Theta_{k\delta};s-1)}{n}}\Bigr)^s
\\
&
\nonumber
+
\Bigl(\sqrt{\frac{{\frak d}_{\xi}(\Theta_{k\delta};s-1)}{n}}\Bigr)^k\frac{\|\Sigma(\theta)\|^{1/2}}{n^{1/2}}
{\mathbb P}^{1/2}\{\|\xi\|_{L_{\infty}(E)}\geq \delta\sqrt{n}\}\Bigr)\Bigr)
\\
&
\nonumber
+\|f\|_{C^s(\Theta)} 
\frac{\|\Sigma(\theta)\|^{1/2}}{n^{1/2}}
\tilde \psi^{1/2} ({\mathbb P}\{\|\xi\|_{L_{\infty}(E)}\geq \delta\sqrt{n}\}\})
\\
&
\nonumber
+\|f\|_{C^s(\Theta_{3\delta})}
\frac{{\mathbb E}^{1/2}\|\xi(\theta)\|^2}{n^{1/2}}
\biggl(\frac{\|\Sigma(\theta)\|^{1/2}}{n^{1/2}}+
\tilde \psi^{1/2}({\mathbb P}\{\|\xi\|_{L_{\infty}(E)}\geq \delta\sqrt{n}\})\biggr)
\\
&
+\|f\|_{C^s(\Theta_{(k+3)\delta})}\sqrt{\frac{{\frak d}_{\xi}(\Theta_{(k+2)\delta};s-1)}{n}}
\biggl(\frac{\|\Sigma(\theta)\|^{1/2}}{n^{1/2}}+
\tilde \psi({\mathbb P}\{\|\xi\|_{L_{\infty}(E)}\geq \delta\sqrt{n}\})\biggr).
\end{align}
It is easy to see that 
\begin{align*}
{\mathbb P}^{1/2}\{\|\xi\|_{L_{\infty}(E)}\geq \delta\sqrt{n}\}\Bigr)\Bigr)
\lesssim \tilde \psi^{1/2}({\mathbb P}\{\|\xi\|_{L_{\infty}(E)}\geq \delta\sqrt{n}\})
\end{align*}
(since $\psi(u)\gtrsim u$) and also that 
\begin{align*}
\tilde \psi({\mathbb P}\{\|\xi\|_{L_{\infty}(E)}\geq \delta\sqrt{n}\})\lesssim \tilde \psi^{1/2}({\mathbb P}\{\|\xi\|_{L_{\infty}(E)}\geq \delta\sqrt{n}\}).
\end{align*}
We also have 
\begin{align*}
\|\Sigma(\theta)\|\leq {\mathbb E}\|\xi(\theta)\|^2 \leq {\frak d}_{\xi}(\Theta_{(k+2)\delta};s-1), \theta\in \Theta.
\end{align*}
Under the assumption that, for a small enough constant $c_1>0,$ 
$${\frak d}_{\xi}(\Theta_{(k+2)\delta};s-1)\leq c_1 n,$$
dropping the terms in bound \eqref{conc_konec} that are dominated by other terms,
we get the following bound:
\begin{align*}
&
\sup_{\theta\in \Theta}
\Bigl\|\tilde f_{\delta,k}(\tilde \theta_{\delta})-f(\theta)-n^{-1/2}\langle f'(\theta), \xi(\theta)\rangle
\Bigr\|_{L_\psi({\mathbb P})}
\\
&
\lesssim_{s, \psi}  
\|f\|_{C^s(\Theta_{(k+3)\delta})}
\biggl[
\biggl(\sqrt{\frac{{\frak d}_{\xi}(\Theta_{(k+2)\delta};s-1)}{n}}\biggr)^s 
+
\frac{\|\Sigma\|_{L_{\infty}(E)}^{1/2}}{n^{1/2}} \sqrt{\frac{{\frak d}_{\xi}(\Theta_{(k+2)\delta};s-1)}{n}}
\\
&
\ \ \ \ \ \ \ \ \ \ \ \ \ \ \ \ \ \ \ \ \ \ \ \ \ \ + \sqrt{\frac{{\frak d}_{\xi}(\Theta_{(k+2)\delta};s-1)}{n}}\tilde \psi^{1/2}({\mathbb P}\{\|\xi\|_{L_{\infty}(E)}\geq \delta \sqrt{n}\})
\biggr].
\end{align*}
This completes the proof of the theorem.
\end{proof}

\section{Proofs  of the main results}
\label{proof_main}

First, we prove Proposition \ref{prop_Delta_H_Delta_F}.

\begin{proof}	
	For $h\in {\mathcal H}:=\{h: \|h\|_{C^s(\Theta_{\delta})}\leq 1\},$ denote 
	$g(x):= h\Bigl(\theta + \frac{x}{\sqrt{n}}\Bigr), x\in E.$ 
	Then 
	\begin{align*}
	\|g\|_{C^{0,s}(U_{\delta\sqrt{n}})} \leq n^{-1/2} \|h\|_{C^{s}(\Theta_{\delta})}\leq n^{-1/2}.	
	\end{align*}	
	Note also that 
	\begin{align*}
	\Bigl|{\mathbb E}_{\theta} h(\hat \theta)- {\mathbb E}_{\theta} h(\tilde \theta)\Bigr|
	=\Bigl|{\mathbb E}_{\theta} g(\sqrt{n}(\hat \theta-\theta))- 
	{\mathbb E}_{\theta} g(\xi(\theta))\Bigr| 
	\end{align*}
	and 
	\begin{align*}
	\Bigl|\|h(\hat \theta)\|_{L_{\psi}({\mathbb P}_{\theta})}- \|h(\tilde \theta)\|_{L_{\psi}({\mathbb P}_{\theta})}\Bigr|
	=\Bigl|\|g(\sqrt{n}(\hat \theta-\theta))\|_{L_{\psi}({\mathbb P}_{\theta})}- 
	\|g(\xi(\theta))\|_{L_{\psi}({\mathbb P}_{\theta})}\Bigr|. 
	\end{align*}
	It easily follows that 
	\begin{align*}
	\Delta_{{\mathcal H},\Theta_{\delta}}(\hat \theta,\tilde \theta)
	\leq 
	\frac{\Delta_{{\mathcal F},\Theta_{\delta}}\Bigl(\sqrt{n}(\hat \theta-\theta), \xi(\theta))\Bigr)}{\sqrt{n}}
	\end{align*}
	and 
	\begin{align*}
	\Delta_{{\mathcal H},\psi, \Theta_{\delta}}(\hat \theta,\tilde \theta)
	\leq 
	\frac{\Delta_{{\mathcal F},\psi, \Theta_{\delta}}\Bigl(\sqrt{n}(\hat \theta-\theta), \xi(\theta))\Bigr)}{\sqrt{n}}.
	\end{align*}
	implying the claims. 
\end{proof}

We will now provide the proof of Theorem \ref{Main_Th_AA}. 
The proofs of theorems \ref{simple_case_sleq2_a} and \ref{simple_case_s<2_B}
are its very simplified versions. 
Note also that the bound of Theorem \ref{risk_efficient_psi} immediately follows 
form bound \eqref{psi_error} of Theorem \ref{Main_Th_AA}.

\begin{proof}
First we will prove bound  \eqref{psi_error}.	
Note that, for some $c_k>0,$ 
\begin{align*}
\|f_k\|_{L_{\infty}(E)} \leq c_k \|f\|_{L_{\infty}(E)}, \ \  
\|\tilde f_{\delta,k}\|_{L_{\infty}(E)} \leq c_k \|f\|_{L_{\infty}(E)}.
\end{align*}
Using the bounds of Theorem \ref{uniform_bd} and Lemma \ref{lemma_psi_simple_2}, we get that, 
for all $\theta\in \Theta,$
\begin{align}	
\label{psi_odin}
&
\nonumber
\|f_k(\hat \theta)- \tilde f_{\delta,k}(\hat \theta)\|_{L_{\psi}({\mathbb P}_{\theta})}
\\
&
\nonumber
\leq \|(f_k(\hat \theta)- \tilde f_{\delta,k}(\hat \theta))I(\hat \theta\in \Theta_{\delta})\|_{L_{\psi}({\mathbb P}_{\theta})}+ 	\|(f_k(\hat \theta)- \tilde f_{\delta,k}(\hat \theta))I(\hat \theta\not\in \Theta_{\delta})\|_{L_{\psi}({\mathbb P}_{\theta})}
\\
&
\nonumber
\leq \|1\|_{L_{\psi}({\mathbb P}_{\theta})}
\|f_k- \tilde f_{\delta,k}\|_{L_{\infty}(\Theta_{\delta})} 
+ 2c_k  \|f\|_{L_{\infty}(E)} \|I(\hat \theta\not\in \Theta_{\delta})\|_{L_{\psi}({\mathbb P}_{\theta})}
\\
&
\nonumber
\lesssim_{\psi, s} 	\|f\|_{C^s(\Theta_{(k+1)\delta})}\biggl(1+ \frac{{\mathbb E}\|\xi\|_{C^s(\Theta_{k\delta})}^s}{n^{s/2}}\biggr)^{k-1}
\Bigl[\Delta_{s,\delta,\Theta_{\delta}}(\hat \theta, \tilde \theta)+\frak{Q}_n(\Theta_{k\delta},\delta)
\Bigr] 
\\
&
\nonumber
+ \|f\|_{L_{\infty}(E)}  \tilde \psi({\mathbb P}_{\theta}\{\hat \theta\not\in \Theta_{\delta}\})
\\
&
\nonumber
\lesssim_{\psi, s} 	\|f\|_{C^s(\Theta_{(k+1)\delta})}\biggl(1+ \frac{{\mathbb E}\|\xi\|_{C^s(\Theta_{k\delta})}^s}{n^{s/2}}\biggr)^{k-1}
\Bigl[\Delta_{s,\delta,\Theta_{\delta}}(\hat \theta, \tilde \theta)+\frak{Q}_n(\Theta_{k\delta},\delta)
\Bigr] 
\\
&
+ \|f\|_{L_{\infty}(E)}  \tilde \psi({\mathbb P}_{\theta}\{\|\hat \theta-\theta\|\geq \delta\}).
\end{align}	
Let $g(x):= \tilde f_{\delta,k}(x)-f(\theta), x\in E.$ Using the bound of Lemma \ref{f_delta_k_first},
we get 
\begin{align*}
\|g\|_{C^s(\Theta_{\delta})} \lesssim_s \biggl(1+ \frac{{\mathbb E}\|\xi\|_{C^s(\Theta_{k\delta})}^s}{n^{s/2}}\biggr)^{k}
\|f\|_{C^s(\Theta_{(k+1)\delta})}.
\end{align*}
Therefore,
\begin{align}
\label{psi_dva}
&
\nonumber
\sup_{\theta\in \Theta} \Bigl|\|\tilde f_{\delta,k}(\hat \theta)-f(\theta)\|_{L_{\psi}({\mathbb P}_{\theta})}  
-\|\tilde f_{\delta,k}(\tilde \theta)-f(\theta)\|_{L_{\psi}({\mathbb P}_{\theta})}
\Bigr|
\\
&
\nonumber
= \sup_{\theta\in \Theta} \Bigl|\|g(\hat \theta)\|_{L_{\psi}({\mathbb P}_{\theta})}
-\|g(\tilde \theta)\|_{L_{\psi}({\mathbb P}_{\theta})}\Bigr|
\\
&
\nonumber
\lesssim_{s}
\biggl(1+ \frac{{\mathbb E}\|\xi\|_{C^s(\Theta_{k\delta})}^s}{n^{s/2}}\biggr)^{k}
\|f\|_{C^s(\Theta_{(k+1)\delta})}
\sup_{\theta \in \Theta} \sup_{\|h\|_{C^s(\Theta_\delta)}\leq 1}\Bigl|\|h(\hat \theta)\|_{L_{\psi}({\mathbb P}_{\theta})}
-\|h(\tilde \theta)\|_{L_{\psi}({\mathbb P}_{\theta})}\Bigr|
\\
&
\lesssim_{s}
\biggl(1+ \frac{{\mathbb E}\|\xi\|_{C^s(\Theta_{k\delta})}^s}{n^{s/2}}\biggr)^{k}
\|f\|_{C^s(\Theta_{(k+1)\delta})}\Delta_{{\mathcal H}, \psi, \Theta}(\hat \theta, \tilde \theta),
\end{align}
where ${\mathcal H}:=\{h: \|h\|_{C^s(\Theta_\delta)}\leq 1\}.$	
In addition, for all $\theta\in \Theta,$ we get, using Lemma \ref{lemma_psi_simple_2}, 
\begin{align}
\label{psi_tri}
&
\nonumber
\|\tilde f_{\delta,k}(\tilde \theta)-\tilde f_{\delta,k}(\tilde \theta_{\delta})\|_{L_{\psi}({\mathbb P}_{\theta})}
=\Bigl \|(\tilde f_{\delta,k}(\tilde \theta)-\tilde f_{\delta,k}(\theta))I(\|\xi\|_{L_{\infty}(E)}\geq \delta \sqrt{n})\Bigr\|_{L_{\psi}({\mathbb P})}
\\
&
\leq 2c_k \|f\|_{L_{\infty}(E)}\|I(\|\xi\|_{L_{\infty}(E)}\geq \delta \sqrt{n})\|_{L_{\psi}({\mathbb P})}
\lesssim_k 
\|f\|_{L_{\infty}(E)} \tilde \psi({\mathbb P}\{\|\xi\|_{L_{\infty}(E)}\geq \delta \sqrt{n}\}).
\end{align}
Finally, by the bound of Theorem \ref{conc_main},
\begin{align}
\label{psi_chetire}
&
\nonumber
\sup_{\theta\in \Theta}
\Bigl\|\tilde f_{\delta,k}(\tilde \theta_{\delta})-f(\theta)-n^{-1/2}\langle f'(\theta), \xi(\theta)\rangle
\Bigr\|_{L_\psi({\mathbb P})}
\\
&
\nonumber
\lesssim_{s, \psi}  
\|f\|_{C^s(\Theta_{(k+3)\delta})}
\biggl[
\biggl(\sqrt{\frac{{\frak d}_{\xi}(\Theta_{(k+2)\delta};s-1)}{n}}\biggr)^s 
+
\frac{\|\Sigma\|_{L_{\infty}(E)}^{1/2}}{n^{1/2}}\sqrt{\frac{{\frak d}_{\xi}(\Theta_{(k+2)\delta};s-1)}{n}}
\\
&
\ \ \ \ \ \ \ \ \ \ \ \ \ \ \ \ \ \ \ \ \ \ \ \ \ \ + \sqrt{\frac{{\frak d}_{\xi}(\Theta_{(k+2)\delta};s-1)}{n}}\tilde \psi^{1/2}({\mathbb P}\{\|\xi\|_{L_{\infty}(E)}\geq \delta \sqrt{n}\})
\biggr].
\end{align}
Recall that, under the condition ${\frak d}_{\xi}(\Theta_{k\delta};s)\leq c_1 n,$  
\begin{align*}
\biggl(1+ \frac{{\mathbb E}\|\xi\|_{C^s(\Theta_{k\delta})}^s}{n^{s/2}}\biggr)^{k}\lesssim_s 1.
\end{align*}
Also, 
\begin{align*}
\frak{Q}_n(\Theta_{k\delta},\delta) \lesssim_{\psi} \tilde \psi\Bigl(\sup_{\theta\in \Theta_{k\delta}}{\mathbb P}_{\theta}\{\|\hat \theta-\theta\|\geq \delta\}\Bigr)
+
\tilde \psi({\mathbb P}\{\|\xi\|_{L_{\infty}(E)}\geq \delta \sqrt{n}\})
\end{align*}
(since $\psi(u)\gtrsim_{\psi} u$). Taking this into account and combining bounds \eqref{psi_odin}, \eqref{psi_dva}, \eqref{psi_tri}
and \eqref{psi_chetire} yields 
\begin{align*}
&
\sup_{\theta\in \Theta}
\Bigl|\|f_k(\hat \theta)-f(\theta)\|_{L_\psi({\mathbb P}_{\theta})}-n^{-1/2}\sigma_f(\theta)\|Z\|_{L_{\psi}({\mathbb P})}\Bigr|
\\
&
\lesssim_{s, \psi}  
\|f\|_{C^s(\Theta_{(k+3)\delta})}
\biggl[
\biggl(\sqrt{\frac{{\frak d}_{\xi}(\Theta_{(k+2)\delta};s-1)}{n}}\biggr)^s 
+
\frac{\|\Sigma\|_{L_{\infty}(E)}^{1/2}}{n^{1/2}}\sqrt{\frac{{\frak d}_{\xi}(\Theta_{(k+2)\delta};s-1)}{n}}
\\
&
\ \ \ \ \ \ \ \ \ \ \ \ \ \ \ \ \ \ \ \ \ \ \ \ \ \ \ +\Delta_{s,\delta,\Theta_{\delta}}(\hat \theta, \tilde \theta)
+\Delta_{{\mathcal H}, \psi, \Theta}(\hat \theta, \tilde \theta)
\\
&
\ \ \ \ \ \ \ \ \ \ \ \ \ \ \ \ \ \ \ \ \ \ \ \ \ \ \ +\tilde \psi\Bigl(\sup_{\theta\in \Theta_{(k+2)\delta}}{\mathbb P}_{\theta}\{\|\hat \theta-\theta\|\geq \delta\}\Bigr) + 
\tilde \psi^{1/2}({\mathbb P}\{\|\xi\|_{L_{\infty}(E)}\geq \delta \sqrt{n}\})
\biggr]
\\
&
\lesssim_{s, \psi}  
\|f\|_{C^s(\Theta_{(k+3)\delta})}
\biggl[
\biggl(\sqrt{\frac{{\frak d}_{\xi}(\Theta_{(k+2)\delta};s-1)}{n}}\biggr)^s 
+
\frac{\|\Sigma\|_{L_{\infty}(E)}^{1/2}}{n^{1/2}}\sqrt{\frac{{\frak d}_{\xi}(\Theta_{(k+2)\delta};s-1)}{n}}
\\
& 
+\Delta_{{\mathcal H},\psi, \Theta_{\delta}}^{+}(\hat \theta,\tilde \theta)
+\tilde \psi\Bigl(\sup_{\theta\in \Theta_{(k+2)\delta}}{\mathbb P}_{\theta}\{\|\hat \theta-\theta\|\geq \delta\}\Bigr) + 
\tilde \psi^{1/2}({\mathbb P}\{\|\xi\|_{L_{\infty}(E)}\geq \delta \sqrt{n}\})
\biggr],
\end{align*}
where ${\mathcal H}:=\{h: \|h\|_{C^s(\Theta_\delta)}\leq 1\}.$		

This completes the proof of bound  \eqref{psi_error} (subject to a change of variable $(k+3)\delta\mapsto \delta$). 

We now need to prove bound \eqref{bd_Delta_s'_sgeq2}.
	We start with bounding the distance 
	$$
	\Delta_{s'}\Bigl(\sqrt{n}(f_k(\hat \theta)-f(\theta)),
	\sqrt{n}(\tilde f_{\delta,k}(\hat \theta)-f(\theta))\Bigr).
	$$
	For all $\varphi$ with $\|\varphi\|_{C^{s'}({\mathbb R})}\leq 1,$ we have 
	\begin{align}
		\label{bd_f_k_f_delta,k_hattheta_AAA}
		&
		\nonumber
		{\mathbb E}_{\theta}|\varphi(\sqrt{n}(f_k(\hat \theta)-f(\theta)))-
		\varphi(\sqrt{n}(\tilde f_{\delta,k}(\hat \theta))-f(\theta))|
		\leq \|\varphi\|_{{\rm Lip}({\mathbb R})}
		\sqrt{n}{\mathbb E}_{\theta}|f_k(\hat \theta)-\tilde f_{\delta,k}(\hat \theta)| 
		\\
		&
		\nonumber
		\leq \sqrt{n}{\mathbb E}_{\theta}|f_k(\hat \theta)-\tilde f_{\delta,k}(\hat \theta)| I(\hat \theta \in \Theta_{\delta})
		+\sqrt{n}{\mathbb E}_{\theta}|f_k(\hat \theta)-\tilde f_{\delta,k}(\hat \theta)| I(\hat \theta \not\in \Theta_{\delta})
		\\
		&
		\nonumber
		\lesssim_{s,k}  
		\|f\|_{C^s(\Theta_{(k+1)\delta})}\biggl(1+ \frac{{\mathbb E}\|\xi\|_{C^s(\Theta_{k\delta})}^s}{n^{s/2}}\biggr)^{k-1}
		\Bigl[\sqrt{n}\Delta_{s,\delta, \Theta_{\delta}}(\hat \theta, \tilde \theta)+\sqrt{n}\frak{Q}_n(\Theta_{k\delta},\delta)
		\Bigr] 
		\\
		&
		\nonumber
		+ \|f\|_{L_{\infty}(E)} \sqrt{n}{\mathbb P}_{\theta}\{\hat \theta\not\in \Theta_{\delta}\}
		\\
		&
		\lesssim_{s,k}  
		\|f\|_{C^s(\Theta_{(k+1)\delta})}\biggl(1+ \frac{{\mathbb E}\|\xi\|_{C^s(\Theta_{k\delta})}^s}{n^{s/2}}\biggr)^{k-1}
		\Bigl[\sqrt{n}\Delta_{s,\delta, \Theta_{\delta}}(\hat \theta, \tilde \theta)+\sqrt{n}\frak{Q}_n(\Theta_{k\delta},\delta)
		\Bigr]. 
	\end{align}

	To obtain an upper bound on 
	$$
	\Delta_{s,\delta, \Theta_{\delta}}(\hat \theta, \tilde \theta)
	= \Delta_{{\mathcal H}, \Theta_{\delta}}(\hat \theta, \tilde \theta),
	$$
	where ${\mathcal H}=\{h: \|h\|_{C^{s}(\Theta_{\delta})}\leq 1\},$
	we use Proposition \ref{prop_Delta_H_Delta_F}.
	Substituting its first bound into bound \eqref{bd_f_k_f_delta,k_hattheta_AAA}, we easily get 
	\begin{align}
		\label{bd_A}
		&
		\nonumber
		\sup_{\theta \in \Theta}\Delta_{s'}\Bigl(\sqrt{n}(f_k(\hat \theta)-f(\theta)),
		\sqrt{n}(\tilde f_{\delta,k}(\hat \theta)-f(\theta))\Bigr)
		\\
		&
		\lesssim_{s,k}  
		\|f\|_{C^s(\Theta_{(k+1)\delta})}\biggl(1+ \frac{{\mathbb E}\|\xi\|_{C^s(\Theta_{k\delta})}^s}{n^{s/2}}\biggr)^{k-1}
		\Bigl[\Delta_{{\mathcal F}_1,\Theta_{\delta}} (\sqrt{n}(\hat \theta-\theta), \xi(\theta)) +\sqrt{n}\frak{Q}_n(\Theta_{k\delta},\delta)
		\Bigr],
	\end{align}
	where $\mathcal F_1:=\{h: \|h\|_{C^{0,s}(U_{2\delta\sqrt{n}})}\leq 1\}.$
	
	The next step is to control the distance 
	$$
	\Delta_{s'}\Bigl(\sqrt{n}(\tilde f_{\delta,k}(\hat \theta)-f(\theta)), 
	\sqrt{n}(\tilde f_{\delta,k}(\tilde \theta_{\delta})-f(\theta))\Bigr).
	$$
	As before, let $\varphi :{\mathbb R}\mapsto {\mathbb  R}$ be a function such that $\|\varphi\|_{C^{s'}({\mathbb  R})}\leq 1.$ 
	Denote 
	\begin{align*}
		g_{\varphi}(x):= \varphi (\sqrt{n} (\tilde f_{\delta,k}(\theta+n^{-1/2}x)-f(\theta))), x\in E.
	\end{align*}
	We can write 
	\begin{align*}
		\varphi (\sqrt{n}(\tilde f_{\delta,k}(\hat \theta)-f(\theta))) = g_{\varphi}(\sqrt{n}(\hat \theta-\theta))
	\end{align*}
	and 
	\begin{align*}
		\varphi (\sqrt{n}(\tilde f_{\delta,k}(\tilde \theta_{\delta})-f(\theta))) = g_{\varphi}(\xi_{\delta}(\theta)).
	\end{align*}
	
	The next lemma will be used (see also Lemma \ref{superposition_h_g} and Remark \ref{proofs_superposition} in Section \ref{Gaussian_approx}).  
	
	\begin{lemma}
		\label{lemma_g_varphi}
		The following bound holds:
		\begin{align*}
			\|g_{\varphi}\|_{C^{s'}(U_{\delta\sqrt{n}})} 
			\lesssim_{s'}
			\|\varphi\|_{C^{s'}({\mathbb  R})}
			\Bigl(1\vee \|\tilde f_{\delta,k}\|_{C^{s'}(\Theta_{\delta})}^{s'}\Bigr).
		\end{align*}
	\end{lemma}
	
	Taking into account that $\|\varphi\|_{C^{s'}({\mathbb  R})}\leq 1,$
	lemmas \ref{f_delta_k_first} and \ref{lemma_g_varphi} imply that, under the assumption ${\frak d}_{\xi}(\Theta_{k\delta};s)\lesssim  n,$
	\begin{align*}
		\|g_{\varphi}\|_{C^{s'}(U_{\delta\sqrt{n}})} 
		\lesssim_{s} 1\vee 
		\|\tilde f_{\delta,k}\|_{C^s(\Theta_{\delta})}^s
		\lesssim_{s} 
		1\vee \|f\|_{C^s(\Theta_{\delta(k+1)})}^s.
	\end{align*}
	Therefore, 
	\begin{align*}
		&
		\Delta_{s'}(\sqrt{n}(\tilde f_{\delta,k}(\hat \theta)-f(\theta)), \sqrt{n}(\tilde f_{\delta,k}(\tilde \theta_{\delta})-f(\theta)))
		\\
		&
		= \sup_{\theta\in \Theta} \sup_{\|\varphi\|_{C^{s'}({\mathbb  R})}\leq 1} 
		\Bigl|{\mathbb E}_{\theta} \varphi (\sqrt{n}(\tilde f_{\delta, k}(\hat \theta)-f(\theta)))- \varphi (\sqrt{n}(\tilde f_{\delta, k}(\tilde \theta_{\delta})-f(\theta)))\Bigr|
		\\
		&
		\leq 
		\sup_{\theta\in \Theta} \sup_{\|\varphi\|_{C^{s'}({\mathbb  R})}\leq 1} 
		\Bigl|{\mathbb E}_{\theta}g_{\varphi}(\sqrt{n}(\hat \theta-\theta)) - {\mathbb  E} g_{\varphi}(\xi_{\delta}(\theta))\Bigr|
		\\
		&
		\leq 
		\sup_{\|\varphi\|_{C^{s'}({\mathbb  R})}\leq 1}\|g_{\varphi}\|_{C_{s'}(U_{\delta\sqrt{n}})}
		\Delta_{{\mathcal F}, \Theta}(\sqrt{n}(\hat \theta-\theta),\xi_{\delta}(\theta))
		\\
		&
		\lesssim_{s} (1\vee \|f\|_{C^s(\Theta_{(k+1)\delta})}^s)\Delta_{{\mathcal F}, \Theta}(\sqrt{n}(\hat \theta-\theta),\xi_{\delta}(\theta)),
	\end{align*}
	where ${\mathcal F}:=\{g: \|g\|_{C^{s'}(U_{\delta\sqrt{n}})}\leq 1\}.$
	
	In addition, for any Lipschitz function $g$ in $E,$
	\begin{align*}
		&
		|{\mathbb E}_{\theta} g(\xi(\theta)) -{\mathbb E}_{\theta} g(\xi_{\delta}(\theta))|
		\leq \|g\|_{{\rm Lip}(E)}{\mathbb E}\|\xi(\theta)-\xi_{\delta}(\theta)\|
		\\
		&
		\leq   \|g\|_{{\rm Lip}(E)} {\mathbb E}^{1/2}\|\xi\|_{L_{\infty}(E)}^2
		{\mathbb P}^{1/2}\{\|\xi\|_{L_{\infty}(E)}\geq \delta\sqrt{n}\}.
	\end{align*}
	This easily implies that 
	\begin{align*}
		&
		\Delta_{{\mathcal F}, \Theta}(\xi(\theta), \xi_{\delta}(\theta))
		\leq \sup_{\theta\in \Theta}{\mathbb E}^{1/2}\|\xi\|_{L_{\infty}(E)}^2 
		\sup_{\theta \in \Theta}{\mathbb P}^{1/2}\{\|\xi\|_{L_{\infty}(E)}\geq \delta\sqrt{n}\}.
	\end{align*}
	Under the assumption ${\frak d}_{\xi}(\Theta_{k\delta};s)\lesssim  n,$
	we also have $\sup_{\theta\in \Theta}{\mathbb E}^{1/2}\|\xi\|_{L_{\infty}(E)}^2\lesssim \sqrt{n},$
	implying  
	\begin{align*}
		&
		\Delta_{{\mathcal F}, \Theta}(\xi(\theta), \xi_{\delta}(\theta))
		\lesssim 
		\sqrt{n}\sup_{\theta \in \Theta}{\mathbb P}^{1/2}\{\|\xi\|_{L_{\infty}(E)}\geq \delta\sqrt{n}\}.
	\end{align*}
	As a consequence,
	\begin{align}
		\label{bd_B}
		&
		\nonumber
		\Delta_{s'}(\sqrt{n}(\tilde f_{\delta,k}(\hat \theta)-f(\theta)), \sqrt{n}(\tilde f_{\delta,k}(\tilde \theta_{\delta})-f(\theta)))
		\\
		&
		\lesssim_{s} (1\vee \|f\|_{C^s(\Theta_{(k+1)\delta})}^s)
		\Bigl(\Delta_{{\mathcal F}, \Theta}(\sqrt{n}(\hat \theta-\theta),\xi(\theta)) + \sqrt{n}{\mathbb P}^{1/2}\{\|\xi\|_{L_{\infty}(E)}\geq \delta\sqrt{n}\}\Bigr).
	\end{align}
	
	Finally, it follows from bound of Theorem \ref{conc_main} applied to $\psi(u)=u^2$ that 
	\begin{align}
		\label{bd_C}
		&
		\nonumber
		\Delta_{s'}(\sqrt{n}(\tilde f_{\delta, k}(\tilde \theta_{\delta})-f(\theta)), \langle f'(\theta), \xi(\theta)\rangle)
		\\
		&
		\nonumber
		=
		\sup_{\theta \in \Theta} \sup_{\|\varphi\|_{C^{s'}({\mathbb R})}\leq 1}
		\Bigl|{\mathbb E}_{\theta} \varphi(\sqrt{n}(\tilde f_{\delta, k}(\tilde \theta_{\delta})-f(\theta))) 
		-{\mathbb E}\varphi(\langle f'(\theta), \xi(\theta)\rangle)
		\Bigr|
		\\
		&
		\nonumber
		\leq \sup_{\theta \in \Theta} {\mathbb E}_{\theta}
		\Bigl|\sqrt{n}(\tilde f_{\delta,k}(\tilde \theta_{\delta})-f(\theta))
		-\langle f'(\theta), \xi(\theta)\rangle\Bigr|
		\\
		&
		\nonumber
		\leq \sup_{\theta\in \Theta}\Bigl\|
		\sqrt{n}(\tilde f_{\delta,k}(\tilde \theta_{\delta})-f(\theta))
		-\langle f'(\theta), \xi(\theta)\rangle
		\Bigr\|_{L_2({\mathbb P})}
		\\
		&
		\nonumber
		\lesssim_s 
		\|f\|_{C^s(\Theta_{(k+3)\delta})}
		\biggl[
		\sqrt{n}\biggl(\sqrt{\frac{{\frak d}_{\xi}(\Theta_{(k+2)\delta};s-1)}{n}}\biggr)^s 
		+ 
		\sqrt{n}{\mathbb P}^{1/4}\{\|\xi\|_{L_{\infty}(E)}\geq \delta \sqrt{n}\}
		\\
		&
		+\|\Sigma\|_{L_{\infty}(E)}^{1/2}\sqrt{\frac{{\frak d}_{\xi}(\Theta_{(k+2)\delta};s-1)}{n}}		
		\biggr].
	\end{align}
Bound \eqref{bd_Delta_s'_sgeq2} follows from bounds \eqref{bd_A}, \eqref{bd_B} and \eqref{bd_C}.

Note that an obvious change of variable $(k+3)\delta\mapsto \delta$ is needed to rewrite 
bounds \eqref{psi_error} and \eqref{bd_Delta_s'_sgeq2} the way they are stated in the theorem.
\end{proof}

We prove corollaries \ref{Main_Th_AA_cor_cor} and \ref{risk_efficient_psi_cor_cor}.

\begin{proof}
First note that all the functions of set ${\mathcal F}$	defined in Theorem \ref{Main_Th_AA} 
are Lipschitz with constant $1.$ This implies that the distance 
$$\Delta_{{\mathcal F}, {\mathbb P}_{\theta}}(\sqrt{n}(\hat \theta-\theta), \xi(\theta))$$
is dominated by the $\zeta_1$-distance between $\sqrt{n}(\hat \theta-\theta)$
and $\xi(\theta),$ which coincides with $W_{1,{\mathbb P}_{\theta}}(\sqrt{n}(\hat \theta-\theta), \xi(\theta)).$ Therefore,
\begin{align*}
\Delta_{{\mathcal F}, \Theta_{\delta}}(\sqrt{n}(\hat \theta-\theta), \xi(\theta))
\leq W_{1,\Theta_{\delta}}(\sqrt{n}(\hat \theta-\theta), \xi(\theta))
\leq W_{2,\Theta_{\delta}}(\sqrt{n}(\hat \theta-\theta), \xi(\theta)).
\end{align*}	
	
Next, we will bound ${\mathbb P}_{\theta}\{\|\hat \theta-\theta\|\geq \delta\}$
in terms of $W_{1,\Theta_{\delta}}(\sqrt{n}(\hat \theta-\theta), \xi(\theta)).$	To this end,  
let $\lambda$ be a function in the real line such that $\lambda (u)=1, u\geq \delta\sqrt{n},$ $\lambda(u)=0, u\leq (\delta/2)\sqrt{n},$
	$\lambda(u)\in [0,1], u\in {\mathbb R}$ and $\|\lambda\|_{{\rm Lip}({\mathbb R})}\leq \frac{2}{\delta\sqrt{n}}.$
	Let $\varphi (x):= \lambda(\|x\|), x\in E.$ Then $\|\varphi\|_{{\rm Lip}(E)}\leq \frac{2}{\delta\sqrt{n}}$
	and the following bound holds
	\begin{align*}
	&
	{\mathbb P}_{\theta}\{\|\hat \theta-\theta\|\geq \delta\}
	\leq {\mathbb E}_{\theta}\varphi (\sqrt{n}(\hat \theta-\theta))
	\\
	&
	\leq |{\mathbb E}_{\theta}\varphi (\sqrt{n}(\hat \theta-\theta))- {\mathbb E}\varphi (\xi(\theta))|
	+ {\mathbb P}\{\|\xi(\theta)\|\geq (\delta/2)\sqrt{n}\}
	\\
	&
	\leq \frac{1}{\delta \sqrt{n}} \Delta_{\mathcal G, {\mathbb P}_{\theta}}\Bigl(\sqrt{n}(\hat \theta -\theta),\xi(\theta)\Bigr)
	+ {\mathbb P}\{\|\xi(\theta)\|\geq (\delta/2)\sqrt{n}\},
	\end{align*}
	where ${\mathcal G}:= \{g: \|g\|_{{\rm Lip}(E)}\leq 1\}.$ This implies that 
\begin{align*}
&
{\mathbb P}_{\theta}\{\|\hat \theta-\theta\|\geq \delta\}
\leq \frac{1}{\delta \sqrt{n}} W_{1,{\mathbb P}_{\theta}}\Bigl(\sqrt{n}(\hat \theta -\theta),\xi(\theta)\Bigr)
+  {\mathbb P}\{\|\xi(\theta)\|\geq (\delta/2)\sqrt{n}\}.	
\end{align*}
Note that, for some constant $c_2>0,$
\begin{align}
\label{bd_on_xi_L_infty}
{\mathbb P}\{\|\xi\|_{L_{\infty}(E)}\geq  \delta\sqrt{n}\}\leq 
\exp\biggl\{-\frac{c_2\delta^2 n}{\|\Sigma\|_{L_{\infty}(E)}}\biggr\}.
\end{align}
The last bound holds under the assumption that 
\begin{align*}
{\mathbb E} \|\xi\|_{L_{\infty}(E)} \leq \frac{\delta \sqrt{n}}{2}
\end{align*}
(that itself holds if ${\frak d}_{\xi}(\Theta_{\delta};s) \leq c_1\delta^2 n$ with small enough $c_1>0$)
and follows from the Gaussian concentration. 
Therefore, we get 
\begin{align}
\label{norm_hat_theta_xi}
&
\sup_{\theta\in \Theta_{\delta}}{\mathbb P}_{\theta}\{\|\hat \theta-\theta\|\geq \delta\}
\leq \frac{1}{\delta \sqrt{n}} W_{1,\Theta_{\delta}}\Bigl(\sqrt{n}(\hat \theta -\theta),\xi(\theta)\Bigr)
+  \exp\biggl\{-\frac{c_2\delta^2 n}{\|\Sigma\|_{L_{\infty}(E)}}\biggr\}.
\end{align}
This allows one to complete the proof of Corollary \ref{Main_Th_AA_cor_cor}.	
	
To prove Corollary \ref{risk_efficient_psi_cor_cor}, it remains to observe 
that, for ${\mathcal H}$ defined in the statement of Theorem  \ref{Main_Th_AA} 
and $\psi(u)=u^2,$
\begin{align*}	
\Delta_{{\mathcal H},\psi, \Theta_{\delta}}(\hat \theta, \tilde \theta)
\leq \frac{W_{2,\Theta_{\delta}}(\sqrt{n}(\hat \theta-\theta), \xi(\theta))}{\sqrt{n}}.	
\end{align*}	
This follows from Proposition \ref{prop_Delta_H_Delta_F} and bounds \eqref{Delta_psi_W_psi_comp} and \eqref{Delta_F_psi_W_psi}. Using again bound \eqref{norm_hat_theta_xi},
we could complete the proof of 	Corollary \ref{risk_efficient_psi_cor_cor}.
\end{proof}

We now prove Corollary \ref{Cor_Main_Th_AA}.

\begin{proof}
The first claim immediately follows bound \eqref{psi_error_W_2}	of Corollary \ref{risk_efficient_psi_cor_cor}.
To prove the second claim, we need two very simple lemmas.
	
	\begin{lemma}
		\label{lem_DDD}
		Let $\xi, \eta$ be random variables. 
		For all $s'\geq 1$ and all $a>0,$
		\begin{align*}
		\Delta_{s'}(a\xi, a\eta) \leq (a^{s'}\vee 1) \Delta_{s'}(\xi,\eta).
		\end{align*}
	\end{lemma}
	
	\begin{proof}
		Immediate from the definition of $\Delta_{s'}.$
	\end{proof}
	
	\begin{lemma}
		\label{lem_UUU}
		Let $\eta$ be a r.v. and let $Z\sim N(0,1).$
		Then 
		\begin{align*}
		d_K(\eta, Z)
		\lesssim \Delta_{s'}^{1/(1+s')}(\eta,Z).
		\end{align*}
	\end{lemma}
	
	\begin{proof}
		Note that, by the definition, $\Delta_{s'}(\xi,\eta)\leq 2.$
		Clearly, there exists a function $\varphi \in C^{s'}({\mathbb R})$ such that $\varphi (t)=1, t\leq 0,$
		$\varphi (t)=1, t\geq 1$ and $\varphi (t)\in [0,1], t\in [0,1].$ For $x\in {\mathbb R},$ $\eps\in (0,1),$
		define 
		\begin{align*}
		\varphi_{x,\eps}(y):= \varphi(\eps^{-1}(y-x)), y\in {\mathbb R}.
		\end{align*}
		Since $I_{(-\infty, x]}\leq \varphi_{x,\eps}\leq I_{(-\infty, x+\eps]},$
		and $\|\varphi_{x,\eps}\|_{C^{s'}}\lesssim \eps^{-s'},$
		we have, with some constant $c'>0,$ 
		\begin{align*}
		&
		{\mathbb P}\{\eta \leq x\} \leq {\mathbb E}\varphi_{x,\eps}(\xi)
		\leq {\mathbb E}\varphi_{x,\eps}(Z) + c' \eps^{-s'} \Delta_{s'}(\xi,Z)
		\\
		&
		\leq {\mathbb P}\{Z\leq x+\eps\} + c' \eps^{-s'} \Delta_{s'}(\xi,Z)
		\leq {\mathbb P}\{Z\leq x\} +\eps + c' \eps^{-s'} \Delta_{s'}(\xi,Z)
		\end{align*}
		and, similarly, 
		\begin{align*}
		&
		{\mathbb P}\{\eta \leq x\} \geq 
		{\mathbb P}\{Z\leq x\} -\eps - c' \eps^{-s'} \Delta_{s'}(\xi,Z).
		\end{align*}
		It remains to set $\eps := \Delta_{s'}^{1/(1+s')}(\xi,Z)\wedge 1$
		to complete the proof.
	\end{proof}

	Using bound \eqref{bd_Delta_s'_sgeq2} of Theorem \ref{Main_Th_AA}, Lemma \ref{lem_DDD} and bound \eqref{bd_on_xi_L_infty},
	we get under the conditions of the corollary that 
	\begin{align*}
	&
	\sup_{\|f\|_{C^s(\Theta_{\delta})}\leq 1}
	\sup_{\theta \in \Theta,\sigma_f(\theta)\geq \sigma_0}
	\Delta_{s'}\Bigl(\frac{\sqrt{n}(f_{k}(\hat \theta)-f(\theta))}{\sigma_{f}(\theta)}, Z\Bigr)
	\\
	&
	\lesssim 
	(1\vee \sigma_0^{-s'})\sup_{\|f\|_{C^s(\Theta_{\delta})}\leq 1}
	\sup_{\theta\in \Theta}\Delta_{s'}\Bigl(\sqrt{n}(f_{k}(\hat \theta)-f(\theta)), \sigma_{f}(\theta)Z\Bigr)
	\to 0\ {\rm as}\ n\to\infty.
	\end{align*}
	The result now follows from the bound of Lemma \ref{lem_UUU}.
\end{proof}

Our next goal is to prove Corollary \ref{indep_hat_theta_j}.

\begin{proof}
We will use the bounds of corollaries Corollary \ref{Main_Th_AA_cor_cor} and \ref{risk_efficient_psi_cor_cor} in the case of $\Theta=T=E$ (see bounds 
\eqref{Delta_s'_T=E}, \eqref{psi_error_bdd'''} and \eqref{psi_error'''}). We start with providing an upper bound on ${\frak d}_{\xi}(s)={\mathbb E}\|\xi\|_{C^s(E)}^2.$

\begin{lemma}
For all $s\in (0,1],$
\begin{align*}
{\mathbb E}\|\xi\|_{C^s(E)}^2\leq 
2 \sum_{i=1}^d {\mathbb E}\|\xi_i\|_{C^s(E_i)}^2
\end{align*}
and, for all $s>1,$
\begin{align*}
{\mathbb E}\|\xi\|_{C^s(E)}^2\lesssim
\sum_{i=1}^d {\mathbb E}\|\xi_i\|_{C^1(E_i)}^2 + \max_{1\leq i\leq d} {\mathbb E}  \|\xi_i\|_{C^{1,s}(E_i)}^2 \log(2d).
\end{align*}
\end{lemma}

\begin{proof}
Since $\xi(\theta)= (\xi_1(\theta_1), \dots, \xi_d(\theta_d)), \theta=(\theta_1,\dots, \theta_d),$ 
we have 
\begin{align*}
\|\xi\|_{L_{\infty}(E)}
=\sup_{\theta\in E}\Bigl(\sum_{i=1}^d \|\xi_i(\theta_i)\|^2\Bigr)^{1/2}
\leq \Bigl(\sum_{i=1}^d \|\xi_i\|^2_{L_{\infty}(E_i)}\Bigr)^{1/2}.
\end{align*}
In addition,
\begin{align*}
\|\xi(\theta)-\xi(\theta')\| = \Bigl(\sum_{i=1}^d \|\xi_i(\theta_i)-\xi_i(\theta_i')\|^2\Bigr)^{1/2},
\end{align*}
implying that, for all $\beta\in (0,1],$
\begin{align*}
\|\xi\|_{{\rm Lip}_{\beta}(E)} \leq \Bigl(\sum_{i=1}^d \|\xi_i\|_{{\rm Lip}_{\beta}(E_i)}^2\Bigr)^{1/2}.
\end{align*}
The $j$-th derivative of $\xi(\theta)$ is  
\begin{align*}
\xi^{(j)}(\theta)[u_1,\dots, u_j]
= \Bigl(\xi_1^{(j)}(\theta_1)[u_{1,1}, \dots, u_{j,1}], \dots, \xi_d^{(j)}(\theta_d)[u_{1,d}, \dots, u_{j,d}]\Bigr),
\end{align*}
where $u_{l}=(u_{l,1}, \dots, u_{l,d})\in E_1\times \dots \times E_d,\ l=1,\dots, j.$
Therefore, for all $j\geq 1,$
\begin{align*}
&
\|\xi^{(j)}(\theta)[u_1,\dots, u_j]-\xi^{(j)}(\theta')[u_1,\dots, u_j]\|^2 
\\
&
\leq \sum_{i=1}^d \|\xi^{(j)}_i(\theta_i)-\xi^{(j)}_i(\theta_i')\|^2 \|u_{1,i}\|^2 \dots \|u_{j,i}\|^2
\\
&
\leq 
\max_{1\leq i\leq d}\|\xi^{(j)}_i(\theta_i)-\xi^{(j)}_i(\theta_i')\|^2 \sum_{i=1}^d \|u_{1,i}\|^2 \dots \|u_{j,i}\|^2
\\
&
\leq 
\max_{1\leq i\leq d}\|\xi^{(j)}_i(\theta_i)-\xi^{(j)}_i(\theta_i')\|^2 \sum_{i=1}^d \|u_{1,i}\|^2 \dots \sum_{i=1}^d \|u_{j,i}\|^2
\\
&
=
\max_{1\leq i\leq d}\|\xi^{(j)}_i(\theta_i)-\xi^{(j)}_i(\theta_i')\|^2 \|u_1\|^2\dots \|u_d\|^2,
\end{align*}
implying that
\begin{align*}
\|\xi^{(j)}(\theta)-\xi^{(j)}(\theta')\| 
\leq \max_{1\leq i\leq d}\|\xi^{(j)}_i(\theta_i)-\xi^{(j)}_i(\theta_i')\|, \theta, \theta'\in E.
\end{align*}
It easily follows from the last bound that, for all $j\geq 1$ and all $\beta\in (0,1],$ 
\begin{align*}
\|\xi^{(j)}\|_{{\rm Lip}_{\beta}(E)} \leq  \max_{1\leq i\leq d} \|\xi_i^{(j)}\|_{{\rm Lip}_{\beta}(E_i)}.
\end{align*}
Thus, it is easy to conclude that, for $s=m+\rho, m\geq 1, \rho\in (0,1],$
we have 
\begin{align*}
\|\xi\|_{C^s(E)} &\leq \Bigl(\sum_{i=1}^d \|\xi_i\|^2_{L_{\infty}(E_i)}\Bigr)^{1/2}
\vee 
\Bigl(\sum_{i=1}^d \|\xi_i\|_{{\rm Lip}(E_i)}^2\Bigr)^{1/2}
\\
&
\vee \max_{1\leq j\leq m-1}  \max_{1\leq i\leq d} \|\xi_i^{(j)}\|_{{\rm Lip}(E_i)}
\vee \max_{1\leq i\leq d} \|\xi_i^{(m)}\|_{{\rm Lip}_{\rho}(E_i)}
\\
&
=
\Bigl(\sum_{i=1}^d \|\xi_i\|^2_{L_{\infty}(E_i)}\Bigr)^{1/2}
\vee 
\Bigl(\sum_{i=1}^d \|\xi_i\|_{{\rm Lip}(E_i)}^2\Bigr)^{1/2}
\vee \max_{1\leq i\leq d}\|\xi_i\|_{C^{1,s}(E_i)}
\end{align*}
and, for $s\in (0,1],$ we have 
\begin{align*}
&
\|\xi\|_{C^s(E)} \leq \Bigl(\sum_{i=1}^d \|\xi_i\|^2_{L_{\infty}(E_i)}\Bigr)^{1/2}
\vee 
\Bigl(\sum_{i=1}^d \|\xi_i\|_{{\rm Lip}_s(E_i)}^2\Bigr)^{1/2}.
\end{align*}
As a consequence, for $s\in (0,1],$
\begin{align}
\label{repk_repk_1}
{\mathbb E}\|\xi\|_{C^s(E)}^2 \leq \sum_{i=1}^d {\mathbb E}\Bigl[\|\xi_i\|_{L_{\infty}(E_i)}^2 + \|\xi_i\|_{{\rm Lip}_s(E_i)}^2\Bigr]
\leq 2 \sum_{i=1}^d {\mathbb E}\|\xi_i\|_{C^s(E_i)}^2
\end{align}
and for $s=m+\rho>1,$ we have 
\begin{align}
\label{repk_repk_2}
{\mathbb E}\|\xi\|_{C^s(E)}^2\leq 
2 \sum_{i=1}^d {\mathbb E}\|\xi_i\|_{C^1(E_i)}^2 + {\mathbb E}  \max_{1\leq i\leq d} \|\xi_i\|_{C^{1,s}(E_i)}^2.
\end{align}

We will show that 
\begin{align}
\label{repka}
{\mathbb E}  \max_{1\leq i\leq d} \|\xi_i\|_{C^{1,s}(E_i)}^2
\lesssim \max_{1\leq i\leq d} {\mathbb E}  \|\xi_i\|_{C^{1,s}(E_i)}^2 \log(2d).
\end{align}
Let $\eta_i :=\|\xi_i\|_{C^{1,s}(E_i)}.$ It easily follows from the Gaussian concentration inequality, 
that $\|\eta_i\|_{\psi_2} \lesssim {\mathbb E}^{1/2}  \|\xi_i\|_{C^{1,s}(E_i)}^2.$ Denote $\sigma:= \max_{1\leq i\leq d}\|\eta_i\|_{\psi_2}.$
Then, by Jensen's inequality, 
\begin{align*}
\exp\Bigl\{{\mathbb E} \max_{1\leq i\leq d} \frac{\eta_i^2}{\sigma^2}\Bigr\}
\leq  {\mathbb E}\exp\Bigl\{\max_{1\leq i\leq d} \frac{\eta_i^2}{\sigma^2}\Bigr\}
\leq \sum_{i=1}^d {\mathbb E}\exp\Bigl\{\frac{\eta_i^2}{\sigma^2}\Bigr\}\leq 2d,
\end{align*}
implying 
\begin{align*}
{\mathbb E}\max_{1\leq i\leq d}\eta_i^2  \leq \sigma^2 \log(2d),
\end{align*}
and bound \eqref{repka} follows.

Bounds \eqref{repk_repk_1}, \eqref{repk_repk_2} and \eqref{repka} imply the claim of the lemma.
\end{proof}

It follows from \eqref{cov_Sigma_Sigma_j} that 
\begin{align*}
\|\Sigma(\theta)\|
&=\sup_{\|u\|, \|v\|\leq 1}|\langle\Sigma(\theta)u,v\rangle| \leq  
\sup_{\|u\|, \|v\|\leq 1}\sum_{j=1}^d |\langle \Sigma_j(\theta_j) u_j,v_j\rangle|
\\
&
\leq \sup_{\|u\|, \|v\|\leq 1}\sum_{j=1}^d \|\Sigma_j(\theta_j)\| \|u_j\|\|v_j\|
\\
&
\leq \max_{1\leq j\leq d}\|\Sigma_j(\theta_j)\| \sup_{\|u\|\leq 1}\Bigl(\sum_{j=1}^d \|u_j\|^2\Bigr)^{1/2}
\sup_{\|v\|\leq 1}\Bigl(\sum_{j=1}^d \|v_j\|^2\Bigr)^{1/2}
\\
&= \max_{1\leq j\leq d}\|\Sigma_j(\theta_j)\|.
\end{align*}
This implies that 
\begin{align*}
\|\Sigma\|_{L_{\infty}(E)}\leq \max_{1\leq j\leq d}\|\Sigma_j\|_{L_{\infty}(E_j)}.
\end{align*} 

Finally, we need to bound $W_{2,E}(\sqrt{n}(\hat \theta-\theta), \xi(\theta)).$ To this end, choose $\eta_j(\theta_j), \zeta_j(\theta_j), j=1,\dots, d$ so that 
$\eta_j(\theta_j)\overset{d}{=}\sqrt{n}(\hat \theta_j-\theta_j),$ $\zeta_j(\theta_j)\overset{d}{=} \xi_j(\theta_j)$ and 
$$W_{2,{\mathbb P}_{\theta_j}}(\sqrt{n}(\hat \theta_j-\theta_j), \xi_j(\theta_j))= {\mathbb E}^{1/2}\|\eta_j(\theta_j)-\zeta_j(\theta_j)\|^2.$$
Moreover, we can assume that $\eta_j(\theta_j), j=1,\dots, d$ are independent r.v. and so are $\zeta_j(\theta_j), j=1,\dots, d.$
Define $\eta(\theta):=(\eta_1(\theta_1), \dots, \eta_d(\theta_d))$ and $\zeta(\theta):=(\zeta_1(\theta_1), \dots, \zeta_d(\theta_d)).$
Clearly, $\eta(\theta)\overset{d}{=} \sqrt{n}(\hat \theta-\theta)$ and $\zeta(\theta)\overset{d}{=}\xi(\theta).$
Then, we have 
\begin{align*}
&
W_{2,{\mathbb P}_{\theta}} (\sqrt{n}(\hat \theta-\theta), \xi(\theta))\leq  {\mathbb E}^{1/2}\|\eta(\theta)-\zeta(\theta)\|^2
\leq \Bigl({\mathbb E}\sum_{j=1}^d \|\eta_j(\theta_j)-\zeta_j(\theta_j)\|^2\Bigr)^{1/2}
\\
&
=\Bigl(\sum_{j=1}^d {\mathbb E}\|\eta_j(\theta_j)-\zeta_j(\theta_j)\|^2\Bigr)^{1/2} =\Bigl(\sum_{j=1}^d W_{2, {\mathbb P}_{\theta_j}}^2(\sqrt{n}(\hat \theta_j-\theta_j), \xi_j(\theta_j))\Bigr)^{1/2}.
\end{align*}
The last bound easily implies that 
\begin{align*}
W_{2,E} (\sqrt{n}(\hat \theta-\theta), \xi(\theta))\leq \Bigl(\sum_{j=1}^d W_{2, E_j}^2(\sqrt{n}(\hat \theta_j-\theta_j), \xi_j(\theta_j))\Bigr)^{1/2}.
\end{align*}
It remains to substitute the above bounds into bounds \eqref{Delta_s'_T=E}, \eqref{psi_error_bdd'''} 
and \eqref{psi_error'''} to complete the proof.
\end{proof}

We turn now to the proof of Proposition \ref{prop_ind_comp_00}.

\begin{proof}
The proof will follows from corollaries \ref{Main_Th_AA_cor_cor}
and \ref{risk_efficient_psi_cor_cor}.
To use these corollaries, we just need
to control the distance 
\begin{align*}
W_{2,E}(\sqrt{n}(\hat \theta-\theta), \xi(\theta))
=W_{2,E}\biggl(A(\theta)\sum_{j=1}^d V_j x_j,A(\theta)\sum_{j=1}^d W_j x_j\biggr),  	
\end{align*} 
where 
\begin{align*}
	V_j := n^{-1/2} \sum_{i=1}^n \eta_{i,j},\ W_j:=n^{-1/2}\sum_{i=1}^n \zeta_{i,j}, 
\end{align*}
$\{\eta_{i,j}: i=1,\dots, n\}$ being i.i.d. copies of $\eta_j$
and $\{\zeta_{i,j}: i=1,\dots, n\}$ being i.i.d. copies of $\zeta_j.$
If we choose $V_j', W_j'$ to be the copies of $V_j, W_j$ such that
\begin{align*}
W_2(V_j,W_j)= {\mathbb E}^{1/2}(V_j'-W_j')^2
\end{align*}
and $(V_j', W_j'), j=1,\dots, d$ are independent, we get 
\begin{align*}
&
W_{2,E}\biggl(A(\theta)\sum_{j=1}^d V_j x_j,A(\theta)\sum_{j=1}^d W_j x_j\biggr)
\leq \sup_{\theta\in E} {\mathbb E}^{1/2} \Bigl\|A(\theta)\sum_{j=1}^d(V_j'-W_j')x_j\Bigr\|^2
\\
&
=\sup_{\theta\in E} {\mathbb E}^{1/2} 
\sup_{\|u\|\leq 1} \Bigl(\sum_{j=1}^d (V_j'-W_j')\langle A(\theta)x_j,u \rangle\Bigr)^2
\\
&
\leq \sup_{\theta\in E}	\sup_{\|u\|\leq 1} 
\biggl(\sum_{j=1}^d \langle A(\theta) x_j, u\rangle^2 {\mathbb E}\sum_{j=1}^d (V_j'-W_j')^2\biggr)^{1/2}
\\
&
=\sup_{\theta\in E}	\sup_{\|u\|\leq 1}\Bigl(\sum_{j=1}^d \langle A(\theta) x_j, u\rangle^2\Bigr)^{1/2}
\Bigl(\sum_{j=1}^d {\mathbb E}(V_j'-W_j')^2\Bigr)^{1/2}
\\
&
= \sup_{\theta\in E}	\sup_{\|u\|\leq 1}\langle \Sigma(\theta)u,u\rangle^{1/2} 
\Bigl(\sum_{j=1}^d W_2^2(V_j,W_j)\Bigr)^{1/2}
\leq \|\Sigma\|_{L_{\infty}(E)}^{1/2} \max_{1\leq j\leq d}W_2(V_j,W_j) \sqrt{d}.
\end{align*}
The following bound follows from the results of \cite{Rio} (see Theorem 4.1):
\begin{align*}
\max_{1\leq j\leq d}W_2(V_j,W_j)\lesssim \frac{\beta_4^{1/2}}{\sqrt{n}}.
\end{align*}
Therefore, 
\begin{align}
\label{W_2,E}	
W_{2,E}(\sqrt{n}(\hat \theta-\theta), \xi(\theta))
\lesssim \beta_4^{1/2}\|\Sigma\|_{L_{\infty}(E)}^{1/2}\sqrt{\frac{d}{n}}.
\end{align}
To complete 
the proof of Proposition \ref{prop_ind_comp_00},
it remains to observe that 
\begin{align}
\label{frak_frak_d}	
{\frak d}_{\xi}(s)\leq  \|A\|_{C^s(E)}^2 {\frak d}_d
\end{align}
and to substitute bound \eqref{W_2,E} in the bounds of corollaries \ref{Main_Th_AA_cor_cor}
and \ref{risk_efficient_psi_cor_cor} (in the case when $\Theta=T=E$). 

\end{proof}

Next we prove Proposition \ref{prop_ind_comp}. 

\begin{proof} 
We will deduce the result from the bounds of theorems \ref{risk_efficient_psi} and \ref{Main_Th_AA} 
(in the case of $\Theta=T=E$). To control the distance $\Delta_{{\mathcal F}, E}(\sqrt{n}(\hat \theta-\theta), \xi(\theta)),$ where ${\mathcal F}= \{g: \|g\|_{C^{0,s'}(E)}\leq 1\},$ note that for $\hat \theta =\bar X$ and a function $g\in {\mathcal F},$
\begin{align*}
{\mathbb E} g(\sqrt{n}(\bar X-\theta))- {\mathbb E} g(\xi(\theta))
=
{\mathbb E} \bar g(V_1,\dots, V_d) - {\mathbb E} \bar g(W_1,\dots, W_d),
\end{align*}
where 
\begin{align*}
	\bar g(u_1,\dots, u_d):= g\Bigl(\theta + A(\theta)\sum_{j=1}^d u_j x_j\Bigr), u=(u_1,\dots, u_d)\in {\mathbb R}^{d}
\end{align*}
and $V_j, W_j, j=1,\dots, d$ are the random variables defined in the proof of  Proposition \ref{prop_ind_comp_00}. 

For $h: {\mathbb R}^d \mapsto {\mathbb R},$ $j=1,\dots, d$
and $s'=m+\gamma, \gamma\in (0,1],$ define 
\begin{align*}
&
\|h\|_{s',j}:= 
\\
&
\sup_{u_i\in {\mathbb R}, i\neq j}\sup_{u_j', u_j''\in {\mathbb R}, u_j'\neq u_j''} 
\frac{\Bigl|\frac{\partial^m h}{\partial u_j^m}(\dots, u_{j-1}, u_j', u_{j+1}, \dots)- 
\frac{\partial^m h}{\partial u_j^m}(\dots, u_{j-1}, u_j'', u_{j+1}, \dots )\Bigr|}
{|u_j'-u_j''|^{\gamma}}.
\end{align*}
Note that, for $g\in {\mathcal F},$
\begin{align*}
\|\bar g\|_{s',j} \leq \|A(\theta)x_j\|^{s'}\leq \|A(\theta)\|^{s'}\|x_j\|^{s'}, j=1,\dots, d.
\end{align*}
Therefore, for $V=(V_1,\dots, V_d), W=(W_1,\dots, W_d),$ we have 
\begin{align*}
|{\mathbb E} \bar g(V_1,\dots, V_d) - {\mathbb E} \bar g(W_1,\dots, W_d)|\leq 
\|A(\theta)\|^{s'}\max_{1\leq j\leq d}\|x_j\|^{s'}\Delta_{{\mathcal G}}(V, W),
\end{align*}
where ${\mathcal G}:=\{h:{\mathbb R}^d\mapsto {\mathbb R}: \max_{1\leq j\leq d}\|h\|_{s',j}\leq 1\}.$
This also implies that 
\begin{align*}
\Delta_{{\mathcal F}, {\mathbb P}_{\theta}}(\sqrt{n}(\hat \theta-\theta), \xi(\theta))
\leq \|A(\theta)\|^{s'}\max_{1\leq j\leq d}\|x_j\|^{s'} \Delta_{{\mathcal G}}(V, W).
\end{align*}

We need the following elementary lemma.

\begin{lemma}
If $V=(V_1,\dots, V_d)$ and $W=(W_1,\dots, W_d)$ are random vectors with independent 
components, then 
\begin{align*}
\Delta_{{\mathcal G}}(V, W)\leq \sum_{j=1}^d \zeta_{s'}(V_j, W_j).
\end{align*}	
\end{lemma}	

\begin{proof}
	For simplicity, assume that $d=2.$ Let $h:{\mathbb R}^2\mapsto {\mathbb R}$
	be a function such that $\max_{j}\|h\|_{s',j}\leq 1.$ 
	Denote 
	\begin{align*}
		h_{y_1, \cdot}(y_2):= h(y_1,y_2), \ 
		h_{\cdot, y_2}(y_1) :=h(y_1,y_2).
	\end{align*} 
	Then 
		\begin{align*}
		&
		|{\mathbb E}h(V_1,V_2) - {\mathbb E}h(W_1, W_2)| 
		\\
		&
		\leq {\mathbb E}_{V_2} |{\mathbb E}_{V_1}h_{\cdot, V_2}(V_1) - 
		{\mathbb E}_{W_1}h_{\cdot, V_2}(W_1)|+
		{\mathbb E}_{W_1} |{\mathbb E}_{V_2}h_{W_1,\cdot}(V_2) - 
		{\mathbb E}_{W_2}h_{W_1, \cdot}(W_2)|
		\\
		&
		\leq \zeta_{s'}(V_1,W_1)+ \zeta_{s'}(V_2,W_2),
	\end{align*}
	which implies the claim.
\end{proof}

Using this lemma for $s'=3$ yields the bound
\begin{align*}
\Delta_{{\mathcal F}, E}(\sqrt{n}(\hat \theta-\theta), \xi(\theta))
\leq \|A\|_{L_{\infty}(E)}^{3}U^3 \sum_{j=1}^d \zeta_{s'}(V_j, W_j).
\end{align*}
It remains to apply Lindeberg's trick to control the distances $\zeta_{3}(V_j, W_j)$ as follows:
$
\zeta_{3}(V_j, W_j)\lesssim \frac{\beta_{3}}{n^{1/2}}.
$
As a result, we get 
\begin{align}
\label{Delta_FE_hat}	
	\Delta_{{\mathcal F}, E}(\sqrt{n}(\hat \theta-\theta), \xi(\theta))
	\lesssim \frac{\|A\|_{L_{\infty}(E)}^{3}U^3\beta_3 d}{n^{1/2}}.
\end{align}
By bound \eqref{bd_Delta_s'_sgeq2} of Theorem \ref{Main_Th_AA}, this implies the first bound of Proposition \ref{prop_ind_comp}. 

To prove the second bound, we use bound \eqref{psi_error} of Theorem \ref{Main_Th_AA}.
To this end, we need to control the distances $\Delta_{{\mathcal H}, E}(\hat \theta, \tilde \theta)$
and $\Delta_{{\mathcal H},\psi, E}(\hat \theta, \tilde \theta)$ for $\psi(u)=u^2$ (their sum is $\Delta_{{\mathcal H},\psi, E}^+(\hat \theta, \tilde \theta)$), where 
${\mathcal H}:=\{g: \|g\|_{C^{s}(E)}\leq 1\}.$ Similarly to \eqref{Delta_FE_hat}, we have (due to a different scaling of r.v. in the sums $V_j$
and $W_j$) that, for $s\geq 3,$
\begin{align}
	\label{Delta_HE_hat}	
	\Delta_{{\mathcal H}, E}(\hat \theta, \tilde \theta)
	\lesssim \frac{\|A\|_{L_{\infty}(E)}^{3}U^3 \beta_{3}d}{n^2}.
\end{align}
To control the distance $\Delta_{{\mathcal H},\psi, E}(\hat \theta, \tilde \theta)$ for $\psi(u)=u^2,$
note that 
\begin{align*}
&
|\|g(\hat \theta)\|_{L_2({\mathbb P}_{\theta})}- \|g(\tilde \theta)\|_{L_2({\mathbb P}_{\theta})}|
\\
&
\leq \sqrt{|\|g(\hat \theta)\|_{L_2({\mathbb P}_{\theta})}^2- \|g(\tilde \theta)\|_{L_2({\mathbb P}_{\theta})}^2|}	
=\sqrt{|{\mathbb E}_{\theta}g^2(\hat \theta)- {\mathbb E}_{\theta}g^2(\tilde \theta)|}.
\end{align*}
Note also that for all $g\in {\mathcal H},$ $\|g^2\|_{C^s(E)}\lesssim_s 1.$ 
Therefore, denoting ${\mathcal H}^2 := \{g^2: g\in {\mathcal H}\},$ we have 
\begin{align*}
\Delta_{{\mathcal H},\psi,{\mathbb P}_{\theta}}(\hat \theta, \tilde \theta)\leq 
\sqrt{\Delta_{{\mathcal H}^2, {\mathbb P}_{\theta}}(\hat \theta, \tilde \theta)}
\lesssim_s \sqrt{\Delta_{{\mathcal H}, {\mathbb P}_{\theta}}(\hat \theta, \tilde \theta)}.
\end{align*}
This implies that 
\begin{align*}
\Delta_{{\mathcal H},\psi, E}(\hat \theta, \tilde \theta)\lesssim_s 
\sqrt{\Delta_{{\mathcal H}, E}(\hat \theta, \tilde \theta)}
\lesssim_s \frac{\|A\|_{L_{\infty}(E)}^{3/2}U^{3/2} \beta_{3}^{1/2}\sqrt{d}}{n}
\end{align*}
and 
\begin{align*}
	\Delta_{{\mathcal H},\psi, E}^{+}(\hat \theta, \tilde \theta)
	&
	= \Delta_{{\mathcal H},E}(\hat \theta, \tilde \theta)+
	\Delta_{{\mathcal H},\psi, E}(\hat \theta, \tilde \theta)
	\lesssim_s 
	\Delta_{{\mathcal H}, E}(\hat \theta, \tilde \theta)+
	\sqrt{\Delta_{{\mathcal H}, E}(\hat \theta, \tilde \theta)}
\\
&
	\lesssim_s \frac{\|A\|_{L_{\infty}(E)}^{3/2}U^{3/2} \beta_{3}^{1/2}\sqrt{d}}{n}
	+\frac{\|A\|_{L_{\infty}(E)}^{3}U^3 \beta_{3}d}{n^2}.
\end{align*}
It remains to substitute the last bound in bound \eqref{psi_error} of Theorem \ref{Main_Th_AA} 
(for $\Theta=T=E$) and to use bound \eqref{frak_frak_d} to complete the proof.
\end{proof}

We will now prove Proposition \ref{Cor_Main_Th_DD}.

\begin{proof}
The proof will be given in the case when 
$\Psi(\theta)=\theta.$ In this case, $X=\theta +\eta,$ where $\eta$ is a mean zero noise sampled 
from some distribution $\mu_{\theta}$ in ${\mathbb R}^d$ for which $C_P(\mu_{\theta})=C_P(P_{\theta}).$
The general case could be reduced to this special case by the change of parameter 
$\vartheta= \Psi(\theta), \theta\in T$ and considering the problem of estimation of 
the functional $(f\circ \Psi^{-1})(\vartheta), \vartheta\in \Psi(T),$ 
for which we could use the estimator 
$$
(f\circ \Psi^{-1})_k(\hat \vartheta)= (f_k\circ \Psi^{-1})(\hat \vartheta)= f_k(\hat \theta).
$$

We will use a recent result \cite{Fathi} (see also \eqref{bd_Fathi_1}) on the accuracy of normal approximation 
for sums of i.i.d. random variables sampled from a distribution satisfying Poincar\'e 
inequality. It is easy to deduce from this result that if $Y_1,\dots, Y_n$ i.i.d. $\sim \mu,$ where $\mu$ is a distribution with mean $0$ and nonsingular covariance 
$\Sigma,$ and $Z\sim N(0,\Sigma),$ then 
\begin{align*}
W_2\Bigl(\frac{Y_1+\dots+Y_n}{\sqrt{n}}, Z\Bigr) \leq \sqrt{C_P(\mu)-1}\|\Sigma\|^{1/2}\|\Sigma^{-1}\|^{1/2} \sqrt{\frac{d}{n}}.
\end{align*}

Let $\xi(\theta)= \Sigma^{1/2}(\theta)W,$ $W\sim N(0,I_d).$
Recall that $\|\Sigma\|_{C^s(\Theta_{\delta})}<\infty$ and 
the spectrum of operators $\Sigma(\theta), \theta\in \Theta_{\delta}$
is uniformly bounded from above and bounded away from zero (see conditions \eqref{cond_on_Sigma}).
Thus, there exists an interval $(a,b)$ with $0<a<b<\infty$ that contains the spectrum of all the operators 
$\Sigma(\theta), \theta\in \Theta_{\delta}.$ Therefore,  
$\Sigma^{1/2}(\theta)=\psi(\Sigma(\theta)), \theta \in \Theta_{\delta},$
where $\psi$ is a $C^{\infty}$ function in ${\mathbb R},$ $\psi(u)=\sqrt{u}, u\in (a,b)$
and $\psi(u)=0, u\not\in (a/2,2b).$ This easily implies that 
\begin{align*}
\|\Sigma^{1/2}\|_{C^s(\Theta_{\delta})}\lesssim \|\Sigma\|_{C^s(\Theta_{\delta})}.
\end{align*}
with a constant that depends on the upper bound on $\|\Sigma\|_{L_{\infty}(\Theta_{\delta})}\vee
\|\Sigma^{-1}\|_{L_{\infty}(\Theta_{\delta})}.$
Thus, we have 
\begin{align*}
{\frak d}_{\xi}(\theta, \delta) = {\mathbb E}\|\Sigma^{1/2}(\cdot)W\|^2_{C^s(\Theta_{\delta})}
\lesssim \|\Sigma^{1/2}\|_{C^s(\Theta_{\delta})}^2 {\mathbb E}\|W\|^2 
\lesssim \|\Sigma\|_{C^s(\Theta_{\delta})}^2d \lesssim \|\Sigma\|_{C^s(\Theta_{\delta})}^2 n^{\alpha}.
\end{align*}

It is well known that, for all $\theta$ 
\begin{align*}
\sup_{\|u\|\leq 1}\|\langle \eta, u\rangle\|_{\psi_1} \lesssim \sqrt{C_P(\mu_{\theta})}, \eta\sim \mu_{\theta}.
\end{align*}
By a standard application of Bernstein inequality for sub-exponential r.v. and 
a discretization argument (see, e.g., \cite{Vershynin}, Chapter 4), we get the following proposition.

\begin{proposition}
\label{Berstein_for_hattheta}
There exists constants $c>0, c_1>0$ such that, for all 
\begin{align*}
\delta\geq c_1\sqrt{\sup_{\theta\in \Theta}C_P(\mu_{\theta})}\sqrt{\frac{d}{n}},
\end{align*}
\begin{align*}
\sup_{\theta \in \Theta} {\mathbb P}_{\theta}\{\|\bar X-\theta\|\geq \delta\}
\leq \exp\Bigl\{-c\Bigl(\frac{\delta^2 n}{\sup_{\theta\in \Theta}C_P(\mu_{\theta})}\wedge 
\frac{\delta n}{\sqrt{\sup_{\theta\in \Theta}C_P(\mu_{\theta})}}\Bigr)\Bigr\}.
\end{align*}
\end{proposition}

If $\Theta_{\delta}\subset T,$ then 
\begin{align*}
&
\sup_{\theta \in \Theta} {\mathbb P}_{\theta}\{\|\hat \theta-\theta\|\geq \delta\}
\leq 
\sup_{\theta \in \Theta} {\mathbb P}_{\theta}\{\|\bar X-\theta\|\geq \delta\}
+ \sup_{\theta \in \Theta} {\mathbb P}_{\theta}\{\bar X\not\in T\} 
\\
&
\leq 
2\sup_{\theta \in \Theta} {\mathbb P}_{\theta}\{\|\bar X-\theta\|\geq \delta\}
\leq 2\exp\Bigl\{-c\Bigl(\frac{\delta^2 n}{\sup_{\theta\in \Theta}C_P(\mu_{\theta})}\wedge 
\frac{\delta n}{\sqrt{\sup_{\theta\in \Theta}C_P(\mu_{\theta})}}\Bigr)\Bigr\}.
\end{align*}
This implies that, under the assumptions $\Theta_{2\delta}\subset T$ and 
$
\sup_{\theta\in \Theta_{\delta}}C_P(\mu_{\theta}) = o(n^{1-\alpha}),
$
we get 
\begin{align*}
\sqrt{n}\sup_{\theta\in \Theta_{\delta}} {\mathbb P}_{\theta}\{\|\hat \theta-\theta\|\geq \delta\}
\to 0\ {\rm as}\ n\to\infty.
\end{align*}

Note also that 
\begin{align*}
&
\sup_{\theta\in \Theta_{\delta}}W_{2,{\mathbb P}_{\theta}}\Bigl(\sqrt{n}(\bar X-\theta), \xi(\theta)\Bigr)  
\\
&
\leq \sqrt{\sup_{\theta\in \Theta_{\delta}}C_P(\mu_{\theta})-1}
\|\Sigma\|_{L_{\infty}(\Theta_{\delta})}^{1/2}\|\Sigma^{-1}\|_{L_{\infty}(\Theta_{\delta})}^{1/2}
\sqrt{\frac{d}{n}}.
\end{align*}
Since $d\lesssim n^{\alpha}$ for some $\alpha\in (0,1),$
$
\sup_{\theta\in \Theta_{\delta}}C_P(\mu_{\theta}) = o(n^{1-\alpha})\ {\rm as}\ n\to\infty
$
and 
$
\|\Sigma\|_{L_{\infty}(\Theta_{\delta})}\vee \|\Sigma^{-1}\|_{L_{\infty}(\Theta_{\delta})}
\lesssim 1,
$
we get 
\begin{align*}
&
\sup_{\theta \in \Theta_{\delta}}W_{2,{\mathbb P}_{\theta}}\Bigl(\sqrt{n}(\bar X-\theta), \xi(\theta)\Bigr)
\to 0 \ {\rm as}\ n\to\infty.
\end{align*}
In addition, 
\begin{align*}
&
\sup_{\theta \in \Theta_{\delta}}
W_{2,{\mathbb P}_{\theta}}\Bigl(\sqrt{n}(\bar X-\theta), \sqrt{n}(\hat \theta-\theta)\Bigr)
\leq 
\sqrt{n}
\sup_{\theta \in \Theta_{\delta}}W_{2,{\mathbb P}_{\theta}}(\bar X-\theta, \tilde \theta-\theta)
\\
&
\leq 
\sqrt{n}\sup_{\theta \in \Theta_{\delta}}
{\mathbb E}_{\theta}^{1/2}\|\bar X-\hat \theta\|^2= 
\sqrt{n}\sup_{\theta \in \Theta_{\delta}}
{\mathbb E}_{\theta}^{1/2}\|\bar X-\theta_0\|^2I(\|\bar X-\theta\|\geq\delta)
\\
&
\leq \sqrt{n}\sup_{\theta \in \Theta_{\delta}}
{\mathbb E}_{\theta}^{1/4}\|\bar X-\theta_0\|^4 \ 
{\mathbb P}_{\theta}^{1/2}\{\|\bar X-\theta\|\geq \delta\}
\\
&
\leq 2^{3/4}\sqrt{n}\sup_{\theta \in \Theta_{\delta}}
{\mathbb E}_{\theta}^{1/4}\|\bar X-\theta\|^4 
\sup_{\theta \in \Theta_{\delta}}{\mathbb P}_{\theta}^{1/2}\{\|\bar X-\theta\|\geq \delta\}
\\
&
+ 
2^{3/4}\sqrt{n}\sup_{\theta \in \Theta_{\delta}}\|\theta-\theta_0\|\sup_{\theta \in \Theta_{\delta}} {\mathbb P}_{\theta}\{\|\bar X-\theta\|\geq \delta\}
\\
&
\leq 2^{3/4}\sqrt{n}\Bigl(\|\Sigma\|_{L_{\infty}(\Theta_{\delta})}^{1/2}\sqrt{\sup_{\theta\in \Theta_{\delta}}C_P(\mu_{\theta})}\sqrt{\frac{d}{n}}+\sup_{\theta \in \Theta_{\delta}}\|\theta-\theta_0\|\Bigr)
\sup_{\theta \in \Theta_{\delta}} {\mathbb P}_{\theta}^{1/2}\{\|\bar X-\theta\|\geq \delta\}
\to 0
\end{align*}
as $n\to\infty,$ where we used the bound of Proposition \ref{Berstein_for_hattheta}
and the conditions $d\lesssim n^{\alpha},$ $\sup_{\theta\in \Theta_{\delta}}C_P(\mu_{\theta})=o(n^{1-\alpha})$
and ${\rm Diam}(\Theta) \lesssim n^{A}.$ 
Therefore, we get
\begin{align*}
&
\sup_{\theta \in \Theta_{\delta}}W_{2,{\mathbb P}_{\theta}}\Bigl(\sqrt{n}(\hat \theta-\theta), \xi(\theta)\Bigr)
\to 0 \ {\rm as}\ n\to\infty.
\end{align*}

Both claims of the proposition now  follow from Corollary \ref{Cor_Main_Th_AA}.
\end{proof}

{\bf Acknowledgment}. The author is very thankful to Cl\'ement Deslandes for careful reading of the manuscript and suggesting a number of corrections
and to the referees for helpful comments.


\begingroup
\begin{thebibliography}{10}
	
\bibitem{Anastasiou_0} A.~Anastasiou.  	
\newblock  Assessing the multivariate normal approximation of the maximum likelihood estimator from high-dimensional, heterogeneous data. 
\newblock {\it Electronic J. of Statistics}, 2018, 12, 2, 3794--3828.
	
\bibitem{Anastasiou} A.~Anastasiou and 	R.~Gaunt.
\newblock 
Wasserstein distance error bounds for the multivariate normal approximation of the maximum likelihood estimator. 2020, {\it arXiv:2005.0520}.

\bibitem{Houdre} B.~Arras and C.~Houdr\'e. 
\newblock
On Stein's Method for Multivariate 
Self-Decomposable Laws. 
\newblock
{\it Electron. J. Probab.}, 2019, 24, 128, 1–63.

\bibitem{Bentkus} V.~Bentkus, M.~Bloznelis and F.~G\"otze.
\newblock A Berry-Esseen Bound for M-estimators.
\newblock{\em Scandinavian Journal of Statistics}, 1997, 24, 4, 485--502.

\bibitem{Bickel_Ritov} P.~Bickel and Y.~Ritov.
\newblock Estimating integrated square density derivatives: sharp best order of convergence 
estimates. 
\newblock {\em Sankhya}, 1988, 50, 381--393. 

\bibitem{BKRW} P.J.~Bickel, C.A.J.~Klaassen, Y.~Ritov and J.A.~Wellner. 
\newblock {\em Efficient and Adaptive Estimation for Semiparametric Models}.
\newblock Johns Hopkins University Press, Baltimore, 1993.

\bibitem{Birge} L.~Birg\'e and P.~Massart. 
\newblock Estimation of integral functionals of a density. 
\newblock {\em Annals of Statistics}, 1995, 23, 11-29.

\bibitem{Bhattacharya} R.N.~Bhattacharya and R.~Ranga Rao.
\newblock {\em Normal Approximation and Asymptotic Expansions}. 
\newblock John Willey \&Sons, New York, 1976. 

\bibitem{Bobkov} S.~Bobkov and M.~Ledoux. 
 \newblock 
 From Brunn-Minkowski to Brascamp-Lieb and to logarithmic 
 Sobolev inequalities. 
 \newblock {\em Geom. Funct. Anal.}, 2000, 10(5), 1028--1052.

\bibitem{Cai_Low_2005a} T.T.~Cai and M.~Low. On adaptive estimation of linear 
functionals. {\em Annals of Statistics}, 2005, 33, 2311--2343. 

\bibitem{Cai_Low_2005b} T.T.~Cai and M.~Low. Non-quadratic estimators 
of a quadratic functional. {\em Annals of Statistics}, 2005, 33, 2930--2956.

\bibitem{Chencov} N.N.~\v Cencov. 
\newblock {\em Statistical decision rules and optimal inference.}
American Mathematical Society, 1982.

\bibitem{Yuansi_Chen} Y.~Chen. 
\newblock An Almost Constant Lower Bound of the Isoperimetric Coefficient in the KLS Conjecture.
{\em Geometric and Functional Analysis}, 2021, 31, 34--61. 

\bibitem{Chernozhukov} V.~Chernozhukov, D.~Chentverikov and K.~Kato. 
\newblock Central limit theorems and bootstrap in high dimensions.
\newblock {\em Annals of Probability}, 2017, 45, 4, 2309--2352.


\bibitem{Chernozhukov_2020} V.~Chernozhukov, D.~Chetverikov and Y.~Koike.
\newblock Nearly optimal central limit theorem and bootstrap approximation in high dimensions. 
{\em arXiv: 2012.09513}

\bibitem{Chernozhukov_2022} V.~Chernozhukov, D.~Chetverikov, K.~Kato and Y.~Koike.
\newblock Improved central limit theorem and bootstrap approximation in high dimensions. 
\newblock {\em Annals of Statistics}, 2022, to appear.


\bibitem{C_C_Tsybakov} O.~Collier, L.~Comminges and A.~Tsybakov.
\newblock
Minimax estimation of linear and quadratic functionals on sparsity classes.
{\em Annals of Statistics}, 2017, 45, 3, 923--958. 

\bibitem{Eldan} R.~ Eldan.  
\newblock Thin  shell  implies spectral  gap up  to  polylog  via  a  stochastic  localization scheme.
\newblock{\em Geometric and Functional Analysis}, 2013, 23, 2, 532--569.	

\bibitem{Eldan} R.~Eldan, D.~Mikulincer and A. Zhai. 
\newblock
The CLT in high dimensions: quantitative bounds via martingale embedding.
{\em Annals of Probability}, 2020, 48, 5, 2494--2524.

\bibitem{Fathi} T.A.~Courtade, M. Fathi and A. Pananjadi. 
\newblock 
Existence of Stein Kernels under a Spectral Gap, and Discrepancy Bounds. 
\newblock {\em Ann. Inst. H. Poincar\'e,} 2019, 55, 2, 777--790.

\bibitem{Donoho_1} D.~Donoho and R.~Liu.
\newblock On minimax estimation of linear functionals. 
\newblock {Technical Report N 105. Department of Statistics, UC Berkeley}, August 1987. 

\bibitem{Donoho_2} D.~Donoho and R.~Liu.
\newblock Geometrizing rates of convergence, II.
\newblock {\em Annals of Statistics}, 1991, 19, 2, 633-667.

\bibitem{Donoho_Nussbaum} D.~Donoho and M.~Nussbaum. 
\newblock Minimax quadratic estimation 
of a quadratic functional. 
\newblock {\em J. Complexity}, 1990, 6, 290--323.  
	
\bibitem{Fathi_1} M. Fathi. 
\newblock 
Higher order Stein Kernels for Gaussian approximation. 
{\em Studia Mathematica}, 2019, 256, 241--258. 		
\bibitem{GillLevit}
R.D.~Gill and B.Y.~Levit. 
\newblock Applications of the van Trees inequality: a Bayesian Cram\'er-Rao bound.
\newblock {\em Bernoulli,} 1995, 1(1-2), 59--79. 

\bibitem{Girko} V.L.~Girko. 
\newblock Introduction to general statistical analysis. 
\newblock{\em Theory Probab. Appl.}, 1987, 32, 2: 229--242.

\bibitem{Girko-1} V.L.~Girko. 
\newblock {\em Statistical analysis of observations of increasing dimension}.  
\newblock Springer, 1995. 

 \bibitem{Han} Y.~Han, J.~Jiao and R. Mukherjee. On estimation of $L_r$-norms in Gaussian white noise 
model. {\em Probability Theory and Related Fields}, 2020, 177, 1243--1294.

\bibitem{Hall} P.~Hall. 
\newblock
{\em The Bootstrap and Edgeworth Expansion.} 
Springer-Verlag, New York, 1992.

\bibitem{Hall_1} P.~Hall and M.A. Martin. 
\newblock On Bootstrap Resampling and Iteration.
\newblock {\em Biometrika}, 1988, 75, 4, 661--671.

\bibitem{Ibragimov} I. A.~Ibragimov and R.Z.~Khasminskii.
\newblock{\em Statistical Estimation: Asymptotic Theory}. 
\newblock Springer-Verlag, New York, 1981. 

\bibitem{Ibragimov_Khasm_Nemirov} I.A.~Ibragimov, A.S.~Nemirovski and R.Z.~Khasminskii. 
\newblock Some problems of nonparametric estimation in Gaussian white noise. 
\newblock {\em Theory of Probab. and Appl.}, 1987, 31, 391--406.  

\bibitem{Jiao} J.~Jiao and Y.~Han.  
\newblock
Bias correction with Jackknife, Bootstrap and Taylor Series. {\em IEEE Transactions on Information Theory},  2020, 66, 7, 4392--4418.

\bibitem{Klemela} J.~Klemel\"a. 
\newblock Sharp adaptive estimation of quadratic functionals. 
\newblock {\em Probability Theory and Related Fields},
\newblock 2006, 134, 539--564.

\bibitem{Koike} Y.~Koike. Notes on the dimension dependence in high-dimensional central limit theorems for hyperrectangles. 
{\em Japanese J. of Statistics and Data Science}, 2021, 4(1), 257--297. 

\bibitem{Koltchinskii_2017} V.~Koltchinskii. 
\newblock Asymptotically Efficient Estimation of Smooth Functionals of Covariance Operators. 
\newblock{\em J. European Mathematical Society}, 2021, 23, 3, 765--843.

\bibitem{Koltchinskii_2018} V.~Koltchinskii. 
\newblock Asymptotic Efficiency in High-Dimensional Covariance Estimation.
\newblock {\em Proc. ICM 2018}, Rio de Janeiro, 2018, vol. 3, 2891--2912.

\bibitem{Koltchinskii_Nickl} V.~Koltchinskii, M.~L\"offler and R.~Nickl.
\newblock Efficient Estimation of Linear Functionals of Principal Components.
\newblock{\em Annals of Statistics}, 2020, 48, 1, 464--490. 

\bibitem{Koltchinskii_Zhilova} V.~Koltchinskii and M.~Zhilova. 
\newblock Efficient estimation of 
smooth functionals in Gaussian shift models. 
\newblock
{\it Ann. Inst. H. Poincar\`e - Probab. et Statist.}, 2021, 57, 1, 351--386. 

\bibitem{Koltchinskii_Zhilova_19} V.~Koltchinskii and M.~Zhilova.
\newblock Estimation of Smooth Functionals in Normal Models: Bias Reduction and Asymptotic Efficiency. {\em Annals of Statistics}, 2021, to appear. {\em arXiv:1912.08877}.

\bibitem{Koltchinskii_Zhilova_2020} V.~Koltchinskii and M.~Zhilova. 
\newblock Estimation of smooth functionals of location parameter 
in Gaussian and Poincar\'e random shift models. {\em Sankhya}, 2021, v. 83, issue 2, no. 4, 569--596. 


\bibitem{Kuchib} A.K~Kuchibhotla, S.~Mukherjea and D.~Banerjee.
\newblock
High-dimensional CLT: Improvements, Non-uniform Extensions and Large Deviations. {\em Bernoulli}, 2021, 27, 1, 192--217.


\bibitem{Laurent} B.~Laurent.
\newblock Efficient estimation of integral functionals of a density. 
\newblock {\em Annals of Statistics}, 1996, 24, 659--681. 

\bibitem{Lee_Vempala} Y.-T.~Lee and S. Vempala. 
\newblock Eldan’s Stochastic Localization and the KLS Hyperplane Conjecture: An Improved Lower Bound for Expansion. 
\newblock  {\em 58th Annual IEEE Symposium on Foundations of Computer Science FOCM 2017}. 

\bibitem{Levit_1} B.~Levit. 
\newblock
On the efficiency of a class of non-parametric estimates.
\newblock {\em Theory of Prob. and applications}, 1975, 20(4), 723--740.

\bibitem{Levit_2} B.~Levit. 
\newblock 
Asymptotically efficient estimation of nonlinear functionals.
\newblock {\em Probl. Peredachi Inf. (Problems of Information Transmission)}, 1978, 14(3), 65--72. 

\bibitem{Lepski} O.~Lepski, A. Nemirovski and V. Spokoiny.
\newblock On estimation of the $L_r$ norm of a regression function. 
\newblock {\em Probab. Theory Relat. Fields}, 1999, 113, 221--253. 

\bibitem{Milman} E.~Milman. 
\newblock On the role of convexity in isoperimetry, spectral gap and concentration. 
\newblock {\it Invent. Math.}, 2009, 177(1), 1-43.

\bibitem{Mukherjee} R.~Mukherjee, W.~Newey and J.~Robins.
\newblock Semiparametric Efficient Empirical Higher Order Influence Function Estimators.
\newblock 2017, {\em arXiv:1705.07577.}

\bibitem{Nemirovski_1990} A.~Nemirovski.
\newblock On necessary conditions for the efficient estimation of functionals of a nonparametric signal which is observed in white noise.
\newblock {\em Theory of Probab. and Appl.}, 1990, 35, 94--103.

\bibitem{Nemirovski} A.~Nemirovski.
\newblock{\em Topics in Non-parametric Statistics}.
\newblock{Ecole d'Ete de Probabilit\'es de Saint-Flour}.
\newblock
Lecture Notes in Mathematics, v. 1738, Springer, New York, 2000.

\bibitem{Paulauskas} V.~Paulauskas and A. Rachkauskas. 
\newblock {\em Approximation theory in central limit theorems. Exact results in Banach spaces}.
\newblock Kluwer Academic Publishers, 1989.

\bibitem{Pfanzagl} J.~Pfanzagl. 
\newblock The Berry-Esseen bound for minimum contrast estimates. 
\newblock {\em Metrika},  1971, 17,  82--91.

\bibitem{Pinelis} I.~Pinelis. 
\newblock Optimal-order uniform and nonuniform bounds on the rate of convergence to
normality for maximum likelihood estimators.
\newblock{\em Electronic Journal of Statistics}, 2017, 11, 1160--1179.

\bibitem{Portnoy} S.~Portnoy.
\newblock On the central limit theorem in ${\mathbb R}^p$ when $p\to\infty.$
\newblock {\em Probability Theory and Related Fields}, 1986, 73, 581--583.


\bibitem{Portnoy_1} S.~Portnoy. 
\newblock Asymptotic behavior of likelihood methods for exponential families 
when the number of parameters tends to infinity.
\newblock{\em Annals of Statistics}, 1988, 16, 1, 356--366. 

\bibitem{Rio} E.~Rio. 
\newblock{Upper bounds for minimal distances in the central limit theorem}.
\newblock{\em Ann. Inst. H. Poincar\'e - Probab. et Statist.}, 2009, 45, 3, 802--817.

\bibitem{Robins} J.~Robins, L.~Li, E. Tchetgen and A. van der Vaart.
\newblock Higher order influence functions and minimax estimation of nonlinear 
functionals. 
\newblock {\em IMS Collections Probability and Statistics: Essays in Honor of David. A. Freedman}, 
2008, vol. 2, 335-421.

\bibitem{Robins_1} J.~Robins, L.~Li, E. Tchetgen and A. van der Vaart.
\newblock Asymptotic Normality of Quadratic Estimators.
\newblock {\em Stochastic Processes and Their Applications.}
2016, 126(12), 3733--3759.

\bibitem{Senatov} V.~Senatov. 
\newblock {\em Normal Approximation: New Results, Methods and Problems}.
VSP, Utrecht, The Netherlands, 1998. 

\bibitem{van der Vaart} A.~van der Vaart. Higher order tangent spaces and influence functions.
{\em Statistical Science}, 2014, 29, 4, 679--686.

\bibitem{Vershynin} R.~Vershynin.  {\em High-Dimensional Probability: An Introduction with Applications in Data Science}. 
Cambridge University Press, 2018. 

\bibitem{Villani} C.~Villani. {\em Optimal Transport. Old and New}. Springer, 2009.

\bibitem{Zolotarev} V. M. Zolotarev. Metric distances in spaces of random variables and their distributions. {\em Mat. Sb. (N.S.)}, 1976, 101(143), 3(11), 416--454. 		 	

\end{thebibliography}
\endgroup
\end{document}